\documentclass{amsart}

\usepackage{amssymb}
\usepackage{amsmath}
\usepackage{amsthm}
\usepackage{amsfonts}

\usepackage[OT2,OT1]{fontenc}

\usepackage{a4wide}
\usepackage{longtable}
\usepackage{multirow}

\newtheorem{lemma}{Lemma}[section]
\newtheorem{proposition}{Proposition}[section]
\newtheorem{theorem}{Theorem}[section]
\newtheorem{corollary}{Corollary}[theorem]
\newtheorem{conjecture}{Conjecture}[section]

\theoremstyle{definition}
\newtheorem{definition}{Definition}[section]

\theoremstyle{remark}
\newtheorem{remark}{Remark}[section]
\newtheorem{example}{Example}[section]

\newcommand\cyrroman{
\renewcommand\rmdefault{wncyr}
\renewcommand\sfdefault{wncyss}
\renewcommand\encodingdefault{OT2}
\normalfont
\selectfont}
\DeclareTextFontCommand{\cyrrm}{\cyrroman}

\def\cprime{\char"7E}

\begin{document}

\title[Principalization algorithm]
{Principalization algorithm\\via class group structure}

\author{Daniel C. Mayer}
\address{Naglergasse 53\\8010 Graz\\Austria}
\email{algebraic.number.theory@algebra.at}
\urladdr{http://www.algebra.at}
\thanks{Research supported by the
Austrian Science Fund,
Grant Nr. J0497-PHY}


\date{April 26, 2013}


\begin{abstract}
For an algebraic number field \(K\)
with \(3\)-class group \(\mathrm{Cl}_3(K)\) of type \((3,3)\),
the structure of the \(3\)-class groups \(\mathrm{Cl}_3(N_i)\) of the
four unramified cyclic cubic extension fields \(N_i\), \(1\le i\le 4\), of \(K\)
is calculated with the aid of presentations for the
metabelian Galois group \(\mathrm{G}_3^2(K)=\mathrm{Gal}(\mathrm{F}_3^2(K)\vert K)\)
of the second Hilbert \(3\)-class field \(\mathrm{F}_3^2(K)\) of \(K\).
In the case of a quadratic base field \(K=\mathbb{Q}(\sqrt{D})\)
it is shown that the structure of the \(3\)-class groups of the
four \(S_3\)-fields \(N_1,\ldots,N_4\)
frequently determines the
type of principalization of the \(3\)-class group of \(K\)
in \(N_1,\ldots,N_4\).
This provides an alternative to
the classical principalization algorithm by Scholz and Taussky.
The new algorithm, which is easily automatizable and executes very quickly,
is implemented in PARI/GP and
is applied to all \(4\,596\) quadratic fields \(K\) with
\(3\)-class group of type \((3,3)\)
and discriminant \(-10^6<D<10^7\)
to obtain extensive statistics of their principalization types
and the distribution of their second \(3\)-class groups \(\mathrm{G}_3^2(K)\)
on various coclass trees of the
coclass graphs \(\mathcal{G}(3,r)\), \(1\le r\le 6\),
in the sense of Eick, Leedham-Green, and Newman.
\end{abstract}

\maketitle


\section{Introduction}
\label{s:Intro}

The principal ideal theorem,
which has been conjectured by Hilbert in 1898
\cite[p. 14]{Fw},
states that each ideal of a number field \(K\)
becomes principal when it is extended to the Hilbert class field
\(\mathrm{F}^1(K)\) of \(K\),
that is the maximal abelian unramified extension field of \(K\).
Inspired by the Artin-Furtw\"angler proof
\cite{Ar2,Fw}
of the principal ideal theorem,
Scholz and Taussky investigated the principalization
in intermediate fields \(K<N<\mathrm{F}_3^1(K)\) between
a base field \(K\) with \(3\)-class group of type \((3,3)\)
and its Hilbert \(3\)-class field \(\mathrm{F}_3^1(K)\).
They developed an algorithm
for computing the principalization of \(K\)
in its four unramified cyclic cubic extension fields \(N_1,\ldots,N_4\)
for the case of a complex quadratic base field \(K\)
\cite{SoTa}.
This algorithm is probabilistic,
since it decides whether an ideal \(\mathfrak{a}\) of \(K\)
becomes principal in \(N_i\), for some \(1\le i\le 4\),
by testing local cubic residue characters
of a principal ideal cube \((\alpha)=\mathfrak{a}^3\),
associated with the ideal \(\mathfrak{a}\),
and of a fundamental unit \(\varepsilon_i\)
of the non-Galois cubic subfield \(L_i\) of the complex \(S_3\)-field \(N_i\)
with respect to a series of rational test primes \((p_\ell)_{\ell\ge 1}\)
and terminating when a critical test prime occurs
\cite[Algorithm, Step 8, p. 81]{Ma}.
An upper bound for the minimal critical test prime \(p_{\ell_0}\)
cannot be given effectively.
\(p_{\ell_0}\) can only be estimated by means of
Chebotar\"ev's density theorem
\cite{SoTa,HeSm},
and thus causes uncertainty.

An entirely different approach to the principalization problem
will be presented in this article.
It is based on a purely group theoretical connection between
the structure of the abelianizations \(M_i/M_i^\prime\)
of the four maximal normal subgroups \(M_i\)
of an arbitrary metabelian \(3\)-group \(G\)
with abelianization \(G/G^\prime\) of type \((3,3)\)
and the kernels \(\ker(\mathrm{T}_i)\)
of the transfers \(\mathrm{T}_i:G/G^\prime\longrightarrow M_i/M_i^\prime\), \(1\le i\le 4\).
By the Artin reciprocity law of class field theory
\cite{Ar1},
a corresponding number theoretical connection is established between
the structure of the \(3\)-class groups \(\mathrm{Cl}_3(N_i)\)
of the four unramified cyclic cubic extension fields \(N_i\)
of an arbitrary algebraic number field \(K\)
with \(3\)-class group \(\mathrm{Cl}_3(K)\) of type \((3,3)\)
and the principalization kernels \(\ker(\mathrm{j}_{N_i\vert K})\)
of the class extension homomorphisms
\(\mathrm{j}_{N_i\vert K}:\mathrm{Cl}_3(K)\longrightarrow\mathrm{Cl}_3(N_i)\),
\(1\le i\le 4\).
The correspondence is obtained
by applying the group theoretical statements about \(G\) and \(M_1,\ldots,M_4\)
to the second \(3\)-class group
\(\mathrm{G}_3^2(K)=\mathrm{Gal}(\mathrm{F}_3^2(K)\vert K)\) of \(K\)
\cite{Ma1},
that is the Galois group of the second Hilbert \(3\)-class field
\(\mathrm{F}_3^2(K)=\mathrm{F}_3^1(\mathrm{F}_3^1(K))\) of \(K\),
and its maximal subgroups \(\mathrm{Gal}(\mathrm{F}_3^2(K)\vert N_i)\),
\(1\le i\le 4\).
We call the family \(\tau(G)=(M_i/M_i^\prime)_{1\le i\le 4}\)
the \textit{transfer target type} (TTT) of \(G\)
and the family \(\varkappa(G)=(\ker(\mathrm{T}_i))_{1\le i\le 4}\)
the \textit{transfer kernel type} (TKT) of \(G\)
(briefly called transfer type in
\cite[\S\ 2.2, p. 476]{Ma1}).

We begin by comparing
the four little two-stage towers
\(K<\mathrm{F}_3^1(K)<\mathrm{F}_3^1(N_i)\), \(1\le i\le 4\),
and the big two-stage tower
\(K<\mathrm{F}_3^1(K)<\mathrm{F}_3^2(K)\)
of \(3\)-class fields in \S\
\ref{s:TwoStgTow}.
Based on these relationships, \S\
\ref{s:NearHom}
is devoted to proving
that the \(3\)-class groups \(\mathrm{Cl}_3(N_i)\), \(1\le i\le 4\),
have the structure of nearly homocyclic abelian \(3\)-groups of \(3\)-rank two,
provided the index of nilpotency of the second \(3\)-class group
\(G=\mathrm{G}_3^2(K)=\mathrm{Gal}(\mathrm{F}_3^2(K)\vert K)\) is not too small.
The structure of the remaining \(3\)-class groups,
which are partially of \(3\)-rank three as predicted at the beginning of \S\
\ref{s:MaxLowExo},
is determined in \S\
\ref{ss:MaxExo},
if \(G\) is of coclass \(1\).
The central results of this paper
concern groups \(G\) of coclass \(\mathrm{cc}(G)\ge 2\) and are developed successively,
beginning in \S\
\ref{ss:LowExo}
on the general method of proof,
for sporadic groups of coclass \(2\) with bicyclic centre in \S\
\ref{ss:SmpSec}
and with cyclic centre in \S\
\ref{ss:SmlSec},
for groups of coclass \(2\) on coclass trees
\cite{ELNO}
in \S\
\ref{ss:Sec},
and finally, for all groups of coclass \(\mathrm{cc}(G)\ge 3\) in \S\
\ref{ss:Low}.
These results establish theoretical background
for our new principalization algorithm in \S\
\ref{s:CompTech}
which is based on the invariant \(\varepsilon=\varepsilon(K)\),
the number of elementary abelian \(3\)-groups of type \((3,3,3)\)
among the \(3\)-class groups \(\mathrm{Cl}_3(N_i)\), \(1\le i\le 4\).
In view of future generalization
\cite{Ma5}
to base fields \(K\) with
\(\mathrm{Cl}_3(K)\) of type \((9,3)\),
we can also define \(\varepsilon\)
as the number of \(3\)-class groups of elevated \(3\)-rank at least three.
As opposed to Scholz and Taussky's classical algorithm for determining
the \(3\)-principalization over complex quadratic fields
\cite[\S\ 1, pp. 20--31]{SoTa},
which is also described in
\cite[Algorithm, pp. 80--83]{Ma}
and in
\cite[\S\ 3, pp. 9--29, and Appendix A, pp. 96--113]{Br},
and is extended to the \(p\)-principalization with an odd prime \(p\)
over arbitrary quadratic fields in
\cite[\S\ 1, pp. 3--6, and \S\ 3--4, pp. 12--24]{HeSm},
the new \(3\)-principalization algorithm for arbitrary quadratic fields
is easily automatizable and executes very quickly.
In \S\
\ref{s:NumerTab}
an extensive application
of our principalization algorithm via class group structure is presented.




\section{Little and big two-stage towers of \(3\)-class fields}
\label{s:TwoStgTow}

Let \(K\) be an algebraic number field
with \(3\)-class group \(\mathrm{Cl}_3(K)\) of type \((3,3)\).
By class field theory
\cite[Cor. 3.1, p. 838]{Ma3},
there exist
four unramified cyclic cubic extension fields \(N_1,\ldots,N_4\) of \(K\)
within the first Hilbert \(3\)-class field \(\mathrm{F}_3^1(K)\) of \(K\).
Consequently,
the first Hilbert \(3\)-class field \(\mathrm{F}_3^1(N_i)\) of \(N_i\)
is an intermediate field between \(\mathrm{F}_3^1(K)\) and
the second Hilbert \(3\)-class field \(\mathrm{F}_3^2(K)\) of \(K\),
for each \(1\le i\le 4\).

By
\(\Gamma_i=\mathrm{Gal}(\mathrm{F}_3^1(N_i)\vert K)\)
we denote the Galois groups of the four \textit{little two-stage towers} of \(K\),
\(K<\mathrm{F}_3^1(K)\le\mathrm{F}_3^1(N_i)\),
where \(1\le i\le 4\),
and by
\(G=\mathrm{Gal}(\mathrm{F}_3^2(K)\vert K)\)
the Galois group of the \textit{big two-stage tower} of \(K\),
\(K<\mathrm{F}_3^1(K)\le\mathrm{F}_3^2(K)\).

\begin{proposition}
\label{prp:TwoStgTow}

\(G\) and \(\Gamma_i\)
are metabelian \(3\)-groups with abelianizations
\(G/G^\prime\) and \(\Gamma_i/\Gamma_i^\prime\)
of type \((3,3)\).
They are non-abelian,
except for a single-stage tower
\(\mathrm{F}_3^2(K)=\mathrm{F}_3^1(N_i)=\mathrm{F}_3^1(K)\).

\(G\) contains four maximal normal subgroups \(M_1,\ldots,M_4\)
and each \(\Gamma_i\) contains an abelian maximal subgroup
\(A_i=\mathrm{Gal}(\mathrm{F}_3^1(N_i)\vert N_i)\) isomorphic to the
\(3\)-class group \(\mathrm{Cl}_3(N_i)\) of \(N_i\), for \(1\le i\le 4\).

The connection between \(\Gamma_i\) and \(G\),
resp. between \(A_i\) and \(M_i\), for \(1\le i\le 4\), is given by

\begin{equation}
\label{eqn:TwoStgTow}
\begin{array}{rcl}
\Gamma_i & \simeq & G/M_i^\prime\,,\\
A_i      & \simeq & M_i/M_i^\prime\,.
\end{array}
\end{equation}

\end{proposition}

\begin{proof}

Since \(\mathrm{F}_3^1(K)\)
is the maximal abelian unramified extension field of \(K\)
with a power of \(3\) as relative degree,
the Galois groups \(G=\mathrm{Gal}(\mathrm{F}_3^2(K)\vert K)\)
and \(\Gamma_i=\mathrm{Gal}(\mathrm{F}_3^1(N_i)\vert K)\),
\(1\le i\le 4\), are non-abelian,
except in the degenerate case
\(\mathrm{F}_3^2(K)=\mathrm{F}_3^1(N_i)=\mathrm{F}_3^1(K)\).
That \(\Gamma_i\) is also non-abelian when \(G\) is non-abelian
follows from the relation \(\Gamma_i\simeq G/M_i^\prime\),
which is shown at the end of this proof,
together with Cor. 3.1 and Cor. 3.2 in
\cite[pp. 476 and 480]{Ma1},
where it is proved that \(M_i^\prime\) is contained in
the third member of the lower central series of \(G\)
and thus strictly smaller than \(G^\prime\),
for each \(1\le i\le 4\).
The abelian commutator subgroups of \(G\) and \(\Gamma_i\) are given by
\[G^\prime=\mathrm{Gal}(\mathrm{F}_3^2(K)\vert\mathrm{F}_3^1(K))
\simeq\mathrm{Cl}_3(\mathrm{F}_3^1(K))\]
and
\[\Gamma_i^\prime=\mathrm{Gal}(\mathrm{F}_3^1(N_i)\vert\mathrm{F}_3^1(K))
<A_i=\mathrm{Gal}(\mathrm{F}_3^1(N_i)\vert N_i)\simeq\mathrm{Cl}_3(N_i),\]
by the Artin reciprocity law \cite{Ar1}.
Hence, \(G\) and \(\Gamma_i\) are metabelian \(3\)-groups with abelianizations
\(G/G^\prime\) and \(\Gamma_i/\Gamma_i^\prime\) isomorphic to
\(\mathrm{Gal}(\mathrm{F}_3^1(K)\vert K)\simeq\mathrm{Cl}_3(K)\)
and thus of type \((3,3)\).

Since the four maximal normal subgroups \(M_1,\ldots,M_4\) of \(G\)
are associated with the extensions \(N_1,\ldots,N_4\) by
\(M_i=\mathrm{Gal}(\mathrm{F}_3^2(K)\vert N_i)\),
their commutator subgroups are given by
\[M_i^\prime=\mathrm{Gal}(\mathrm{F}_3^2(K)\vert\mathrm{F}_3^1(N_i))\]
and their abelianizations by
\[M_i/M_i^\prime
=\mathrm{Gal}(\mathrm{F}_3^2(K)\vert N_i)
/\mathrm{Gal}(\mathrm{F}_3^2(K)\vert\mathrm{F}_3^1(N_i))
\simeq A_i=\mathrm{Gal}(\mathrm{F}_3^1(N_i)\vert N_i)\simeq\mathrm{Cl}_3(N_i).\]
Since \(M_i^\prime\) is a characteristic subgroup of \(M_i\),
it is a normal subgroup of \(G\) and we have the relation
\[G/M_i^\prime
=\mathrm{Gal}(\mathrm{F}_3^2(K)\vert K)
/\mathrm{Gal}(\mathrm{F}_3^2(K)\vert\mathrm{F}_3^1(N_i))
\simeq\mathrm{Gal}(\mathrm{F}_3^1(N_i)\vert K)
=\Gamma_i.\]

\end{proof}

In the following \S\
\ref{s:NearHom},
we determine the standard structure of the abelian maximal normal subgroups
\(A_i\simeq\mathrm{Cl}_3(N_i)\)
of \(\Gamma_i\), \(1\le i\le 4\),
for a given second \(3\)-class group \(G\) of \(K\).


\section{Nearly homocyclic \(3\)-class groups of \(3\)-rank two}
\label{s:NearHom}

The concept of a nearly homocyclic abelian \(p\)-group with an arbitrary prime \(p\ge 2\)
appears in
\cite[Thm. 3.4, p. 68]{Bl}
(see our Appendix) and
is treated systematically in
\cite[\S\ 2.4]{Ne1}.
For our purpose, it suffices to consider the special case \(p=3\).

\begin{definition}
\label{d:HomAbl}
By the \textit{nearly homocyclic abelian \(3\)-group}
\(\mathrm{A}(3,n)\) of order \(3^n\),
for an integer \(n\ge 2\),
we understand the abelian group of type
\((3^{q+r},3^q)\),
where \(n=2q+r\) with integers \(q\ge 1\) and \(0\le r<2\).
Additionally, including two degenerate cases, we define that
\(\mathrm{A}(3,1)\) denotes the cyclic group \(C_3\) of order \(3\)
and \(\mathrm{A}(3,0)\) the trivial group \(1\).
\end{definition}

The application of Blackburn's well-known Theorem 3.4 in
\cite[p. 68]{Bl}
to the Galois groups \(\Gamma_i=\mathrm{Gal}(\mathrm{F}_3^1(N_i)\vert K)\)
of the four little two-stage towers \(K<\mathrm{F}_3^1(K)\le\mathrm{F}_3^1(N_i)\)
with abelian maximal normal subgroups
\(A_i=\mathrm{Gal}(\mathrm{F}_3^1(N_i)\vert N_i)\simeq\mathrm{Cl}_3(N_i)\)
will show that, in general,
the \(3\)-class groups \(\mathrm{Cl}_3(N_i)\), \(1\le i\le 4\),
are nearly homocyclic abelian \(3\)-groups
\(\mathrm{A}(3,u+v)\) of type \((3^u,3^v)\) with \(1\le v\le u\le v+1\).
The phrase {\lq in general\rq} is made precise in the following theorems,
where we distinguish second \(3\)-class groups
\(G=\mathrm{Gal}(\mathrm{F}_3^2(K)\vert K)\)
of coclass \(\mathrm{cc}(G)=1\) in \S\
\ref{ss:MaxHom}
and of coclass \(\mathrm{cc}(G)\ge 2\) in \S\
\ref{ss:LowHom},
and use concepts and notation of our papers
\cite{Ma1,Ma2},
as recalled in the sequel.


\subsection{Second \(3\)-class groups \(G\) of coclass \(\mathrm{cc}(G)=1\)}
\label{ss:MaxHom}

Let \(G\) be a metabelian \(3\)-group
of order \(\lvert G\rvert=3^n\) and
nilpotency class \(\mathrm{cl}(G)=m-1\), where \(n=m\ge 3\).
Then \(G\) is of coclass \(\mathrm{cc}(G)=n-\mathrm{cl}(G)=1\)
and the commutator factor group \(G/G^\prime\) of \(G\) is of type \((3,3)\)
\cite{Bl,Mi}.
The lower central series of \(G\) is defined
recursively by \(\gamma_1(G)=G\) and
\(\gamma_j(G)=\lbrack\gamma_{j-1}(G),G\rbrack\) for \(j\ge 2\).
In particular, \(\gamma_2(G)=\lbrack G,G\rbrack=G^\prime\)
denotes the commutator subgroup.

The centralizer
\(\chi_2(G)
=\lbrace g\in G\mid\lbrack g,u\rbrack\in\gamma_4(G)\text{ for all }u\in\gamma_2(G)\rbrace\)
of the two-step factor group \(\gamma_2(G)/\gamma_4(G)\), that is,
\[\chi_2(G)/\gamma_4(G)
=\mathrm{Centralizer}_{G/\gamma_4(G)}(\gamma_2(G)/\gamma_4(G))\,,\]
is the biggest subgroup of \(G\) such that
\(\lbrack\chi_2(G),\gamma_2(G)\rbrack\le\gamma_4(G)\).
It is characteristic, contains the commutator group \(\gamma_2(G)\), and
coincides with \(G\) if and only if \(m=3\).
Let the isomorphism invariant \(k=k(G)\),
the \textit{defect of commutativity} of \(G\), be defined by
\[\lbrack\chi_2(G),\gamma_2(G)\rbrack=\gamma_{m-k}(G)\,,\]
where \(k=0\) for \(3\le m\le 4\),
and \(0\le k\le 1\) for \(m\ge 5\),
according to Miech
\cite[p. 331]{Mi}.

Suppose that generators of \(G=\langle x,y\rangle\) are selected such that
\(x\in G\setminus\chi_2(G)\), if \(m\ge 4\), and \(y\in\chi_2(G)\setminus\gamma_2(G)\).
We define the main commutator
\(s_2=\lbrack y,x\rbrack\in\gamma_2(G)\)
and the higher commutators
\(s_j=\lbrack s_{j-1},x\rbrack=s_{j-1}^{x-1}\in\gamma_j(G)\) for \(j\ge 3\).
Then \(G\) satisfies two relations for third powers of the generators \(x\) and \(y\) of \(G\),

\begin{equation}
\label{eqn:MaxRel}
x^3=s_{m-1}^w\quad\text{ and }\quad
y^3s_2^3s_3=s_{m-1}^z
\quad \text{ with exponents }\quad -1\le w,z\le 1\,,
\end{equation}

\noindent
according to Miech
\cite[Thm. 2, (3), p. 332]{Mi}.
Blackburn uses the notation \(\delta=w\) and \(\gamma=z\)
for these relational parameters
\cite[(36)--(37), p. 84]{Bl}.

Additionally, the group \(G\) satisfies
relations for third powers of higher commutators,
\[s_{j+1}^3s_{j+2}^3s_{j+3}=1\quad\text{ for }1\le j\le m-2\,,\]
and the main commutator relation of Miech
\cite[Thm. 2, (2), p. 332]{Mi},

\begin{equation}
\label{eqn:MaxComRel}
\lbrack y,s_2\rbrack=s_{m-1}^a
\in\lbrack\chi_2(G),\gamma_2(G)\rbrack=\gamma_{m-k}(G)\,,
\end{equation}

\noindent
with exponent \(-1\le a\le 1\).
Blackburn uses the notation \(\beta=a\)
\cite[(33), p. 82]{Bl}.

By \(G_a^m(z,w)\) we denote 
the representative of an isomorphism class of
metabelian \(3\)-groups \(G\) of coclass \(\mathrm{cc}(G)=1\)
and of order \(\lvert G\rvert=3^m\),
which satisfies the relations
(\ref{eqn:MaxComRel})
and
(\ref{eqn:MaxRel})
with a fixed system of parameters
\(a\), \(w\), and \(z\).
Obviously, the defect is \(k=0\) if and only if \(a=0\).

The maximal normal subgroups \(M_i\) of \(G\)
contain the commutator group \(\gamma_2(G)\) of \(G\)
as a normal subgroup of index \(3\) and thus
are of the shape \(M_i=\langle g_i,\gamma_2(G)\rangle\).
We define a fixed order by
\(g_1=y\), \(g_2=x\), \(g_3=xy\) and \(g_4=xy^{-1}\).
The commutator subgroups \(\gamma_2(M_i)\)
are of the general form
\(\gamma_2(M_i)=\langle s_2,\ldots,s_{m-1}\rangle^{g_i-1}\),
according to
\cite[Cor. 3.1, p. 476]{Ma1},
and in particular

\begin{equation}
\label{eqn:MaxComGrp}
\begin{array}{rcl}
\gamma_2(M_1)&=&
\begin{cases}
1,&\text{ if }k=0,\\
\gamma_{m-1}(G),&\text{ if }k=1,
\end{cases}\\
\gamma_2(M_i)&=&\gamma_3(G)\quad\text{ for }2\le i\le 4.
\end{array}
\end{equation}

\begin{theorem}
\label{t:MaxHom}
(Transfer target type \(\tau(G)\) for groups \(G\) with \(\mathrm{cc}(G)=1\), \(\mathrm{cl}(G)\ge 5\))

The structure of the \(3\)-class groups \(\mathrm{Cl}_3(N_i)\)
of the four unramified cyclic cubic extension fields \(N_1,\ldots,N_4\)
of an arbitrary base field \(K\)
having a \(3\)-class group \(\mathrm{Cl}_3(K)\) of type \((3,3)\)
and a second \(3\)-class group \(G=\mathrm{Gal}(\mathrm{F}_3^2(K)\vert K)\)
of coclass \(\mathrm{cc}(G)=1\), order \(\lvert G\rvert=3^n\), and class \(\mathrm{cl}(G)=m-1\),
where \(m=n\ge 3\),
is given by the following nearly homocyclic abelian \(3\)-groups.

\begin{eqnarray}
\label{e:MaxHom}
\mathrm{Cl}_3(N_1)&\simeq&
\begin{cases}
\mathrm{A}(3,m-1),&\text{ if }\lbrack\chi_2(G),\gamma_2(G)\rbrack=1,\ k=0,\ m\ge 5\,,\\
\mathrm{A}(3,m-2),&\text{ if }\lbrack\chi_2(G),\gamma_2(G)\rbrack=\gamma_{m-1}(G),\ k=1,\ m\ge 6\,,
\end{cases}\\
\mathrm{Cl}_3(N_i)&\simeq&\mathrm{A}(3,2)\text{ for }2\le i\le 4,\text{ if }m\ge 4\,.
\end{eqnarray}

\end{theorem}

\begin{proof}

The metabelian \(3\)-groups
\(\Gamma_i=\mathrm{Gal}(\mathrm{F}_3^1(N_i)\vert K)\),
having an abelian maximal subgroup \(A_i\simeq\mathrm{Cl}_3(N_i)\),
are of coclass \(\mathrm{cc}(G)=1\), according to
Heider and Schmithals
\cite[Kor., p. 9]{HeSm}
or also to
\cite[Cor. 3.1, p. 476, and Cor. 3.2, p. 480]{Ma1},
since we are dealing with \(3\)-groups here.
This is the crucial condition for the applicability of
Blackburn's Theorem 3.4
\cite[p. 68]{Bl}
(see our Appendix).
It is also the reason why we need the connection
between \(G\) and \(\Gamma_1,\ldots,\Gamma_4\).

We begin by investigating the distinguished extension \(N_1\).

Suppose first that \(\lbrack\chi_2(G),\gamma_2(G)\rbrack=1\), that is \(k=0\).
Then \(\gamma_2(M_1)=1\), by
\cite[Cor. 3.1]{Ma1},
the group \(\Gamma_1\) is isomorphic to \(G\),
by formula
(\ref{eqn:TwoStgTow}),
and has the order \(\lvert\Gamma_1\rvert=3^m\).
For \(m\ge 5\),
the abelian normal subgroup \(A_1\) of \(\Gamma_1\) is isomorphic to \(\mathrm{A}(3,m-1)\),
according to Blackburn
\cite[Thm. 3.4]{Bl},
where the lower bound \(p+2=5\) for the index of nilpotency \(m\) is due to the specialisation \(p=3\).

Next we consider the case
\(\lbrack\chi_2(G),\gamma_2(G)\rbrack=\gamma_{m-1}(G)\), that is \(k=1\).
Then \(\gamma_2(M_1)=\gamma_{m-1}(G)\), by
\cite[Cor. 3.1]{Ma1},
the group \(\Gamma_1\) is isomorphic to \(G/\gamma_{m-1}(G)\),
by formula
(\ref{eqn:TwoStgTow}),
and is therefore the immediate predecessor of \(G\)
\cite[p. 182]{Ne1}
on the coclass graph \(\mathcal{G}(3,1)\) in Figure
\ref{fig:CoCl1},
thus being isomorphic to the mainline group \(G_0^{m-1}(0,0)\)
and of order \(\lvert\Gamma_1\rvert=3^{m-1}\).
For \(m\ge 6\), it follows that \(A_1\simeq\mathrm{A}(3,m-2)\), according to
\cite[Thm. 3.4]{Bl}.

However, Blackburn's result cannot be applied to
the other three extensions \(N_i\) with \(2\le i\le 4\),
since the three isomorphic groups \(\Gamma_i\simeq G/\gamma_3(G)\)
\cite[Cor. 3.1]{Ma1}
are of order \(\lvert\Gamma_i\rvert=3^\mu\) with exponent \(\mu=3<5\).
Therefore we must determine the structure of the abelian normal subgroup \(A_i\) of \(\Gamma_i\)
by the following consideration.
In the case \(m\ge 4\), the group \(\gamma_3(G)>1\) is non-trivial.
Since \(\Gamma_i\simeq G/\gamma_3(G)\) is a predecessor of \(G\)
\cite[p. 182]{Ne1}
on the coclass graph \(\mathcal{G}(3,1)\) in Figure
\ref{fig:CoCl1},
it can only be isomorphic to the
extra special \(3\)-group \(G_0^3(0,0)\) of exponent \(3\) on the mainline,
whose four maximal normal subgroups are all abelian of type \((3,3)\)
and are thus isomorphic to \(\mathrm{A}(3,2)\).
\end{proof}

The vertices of the coclass graph \(\mathcal{G}(3,1)\) in Figure
\ref{fig:CoCl1}
represent all isomorphism classes of finite \(3\)-groups \(G\) with coclass \(\mathrm{cc}(G)=1\).
Two vertices are connected by a directed edge \(H\to G\) 
if \(G\) is isomorphic to the last lower central quotient \(H/\gamma_c(H)\)
where \(c=\mathrm{cl}(H)\) denotes the nilpotency class of \(H\),
and \(\lvert H\rvert=3\lvert G\rvert\), i. e. \(\gamma_c(H)\) is cyclic of order \(3\).
The graph \(\mathcal{G}(3,1)\) has also been drawn in
\cite[Fig. 4.3, p. 63]{As},
\cite[pp. 194--195]{LgNm},
\cite[p. 189 f.]{Ne1},
\cite[p. 46]{DEF},
\cite[\S\ 9]{EkFs}.

The two top vertices (contour squares) are abelian.
\(C_9\) is isolated and
\(C_3\times C_3\) is the root of the unique coclass tree \(\mathcal{T}(C_3\times C_3)\)
of \(\mathcal{G}(3,1)\).
All other vertices (full discs) are metabelian,
according to
Blackburn
\cite[Thm. 6, p. 26]{Bl1}.
Groups with defect \(k=0\) are represented by bigger discs than those with \(k=1\).
Numbers in angles denote the identifiers of groups
in the SmallGroup library
\cite{BEO}
and in GAP 4.4
\cite{GAP},
where we omit the orders, which are given on the left hand scale.

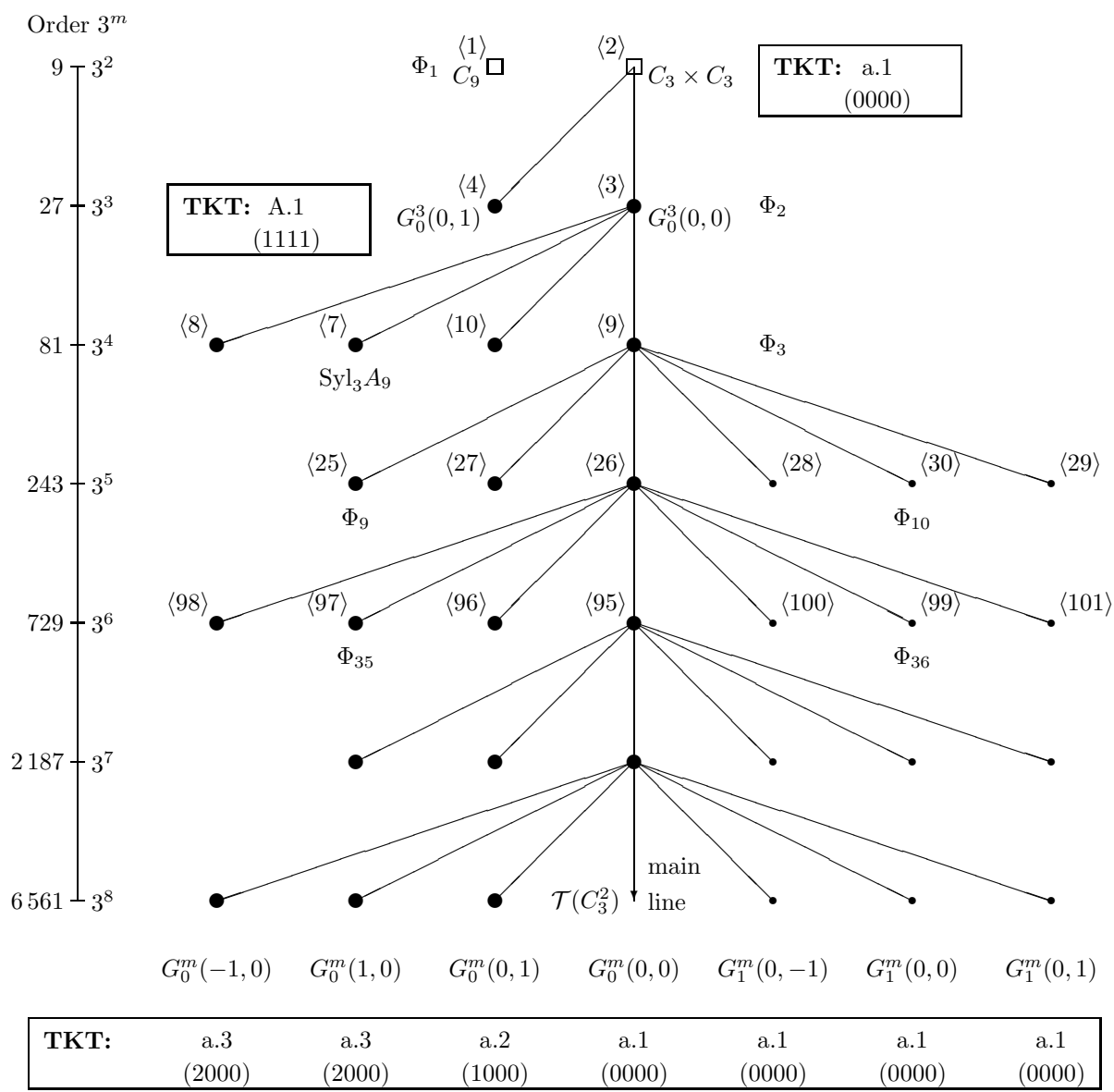
\begin{figure}[ht]
\caption{Root \(C_3\times C_3\) and branches \(\mathcal{B}(j)\), \(2\le j\le 7\), on the coclass graph \(\mathcal{G}(3,1)\)}
\label{fig:CoCl1}

\setlength{\unitlength}{1cm}
\begin{picture}(16,16)(-9,-15)

\put(-8,0.5){\makebox(0,0)[cb]{Order \(3^m\)}}
\put(-8,0){\line(0,-1){12}}
\multiput(-8.1,0)(0,-2){7}{\line(1,0){0.2}}
\put(-8.2,0){\makebox(0,0)[rc]{\(9\)}}
\put(-7.8,0){\makebox(0,0)[lc]{\(3^2\)}}
\put(-8.2,-2){\makebox(0,0)[rc]{\(27\)}}
\put(-7.8,-2){\makebox(0,0)[lc]{\(3^3\)}}
\put(-8.2,-4){\makebox(0,0)[rc]{\(81\)}}
\put(-7.8,-4){\makebox(0,0)[lc]{\(3^4\)}}
\put(-8.2,-6){\makebox(0,0)[rc]{\(243\)}}
\put(-7.8,-6){\makebox(0,0)[lc]{\(3^5\)}}
\put(-8.2,-8){\makebox(0,0)[rc]{\(729\)}}
\put(-7.8,-8){\makebox(0,0)[lc]{\(3^6\)}}
\put(-8.2,-10){\makebox(0,0)[rc]{\(2\,187\)}}
\put(-7.8,-10){\makebox(0,0)[lc]{\(3^7\)}}
\put(-8.2,-12){\makebox(0,0)[rc]{\(6\,561\)}}
\put(-7.8,-12){\makebox(0,0)[lc]{\(3^8\)}}

\put(-0.1,-0.1){\framebox(0.2,0.2){}}
\put(-2.1,-0.1){\framebox(0.2,0.2){}}
\multiput(0,-2)(0,-2){5}{\circle*{0.2}}
\multiput(-2,-2)(0,-2){6}{\circle*{0.2}}
\multiput(-4,-4)(0,-2){5}{\circle*{0.2}}
\multiput(-6,-4)(0,-4){3}{\circle*{0.2}}
\multiput(2,-6)(0,-2){4}{\circle*{0.1}}
\multiput(4,-6)(0,-2){4}{\circle*{0.1}}
\multiput(6,-6)(0,-2){4}{\circle*{0.1}}

\multiput(0,0)(0,-2){5}{\line(0,-1){2}}
\multiput(0,0)(0,-2){6}{\line(-1,-1){2}}
\multiput(0,-2)(0,-2){5}{\line(-2,-1){4}}
\multiput(0,-2)(0,-4){3}{\line(-3,-1){6}}
\multiput(0,-4)(0,-2){4}{\line(1,-1){2}}
\multiput(0,-4)(0,-2){4}{\line(2,-1){4}}
\multiput(0,-4)(0,-2){4}{\line(3,-1){6}}

\put(0,-10){\vector(0,-1){2}}
\put(0.2,-11.5){\makebox(0,0)[lc]{main}}
\put(0.2,-12){\makebox(0,0)[lc]{line}}
\put(-0.2,-12){\makebox(0,0)[rc]{\(\mathcal{T}(C_3^2)\)}}

\put(0.2,0){\makebox(0,0)[lt]{\(C_3\times C_3\)}}
\put(-2.2,0){\makebox(0,0)[rt]{\(C_9\)}}
\put(0.2,-2){\makebox(0,0)[lt]{\(G^3_0(0,0)\)}}
\put(-2.2,-2){\makebox(0,0)[rt]{\(G^3_0(0,1)\)}}
\put(-4,-4.5){\makebox(0,0)[cc]{\(\mathrm{Syl}_3A_9\)}}

\put(-3,0){\makebox(0,0)[cc]{\(\Phi_1\)}}
\put(-2.1,0.1){\makebox(0,0)[rb]{\(\langle 1\rangle\)}}
\put(-0.1,0.1){\makebox(0,0)[rb]{\(\langle 2\rangle\)}}

\put(2,-2){\makebox(0,0)[cc]{\(\Phi_2\)}}
\put(-2.1,-1.9){\makebox(0,0)[rb]{\(\langle 4\rangle\)}}
\put(-0.1,-1.9){\makebox(0,0)[rb]{\(\langle 3\rangle\)}}

\put(2,-4){\makebox(0,0)[cc]{\(\Phi_3\)}}
\put(-6.1,-3.9){\makebox(0,0)[rb]{\(\langle 8\rangle\)}}
\put(-4.1,-3.9){\makebox(0,0)[rb]{\(\langle 7\rangle\)}}
\put(-2.1,-3.9){\makebox(0,0)[rb]{\(\langle 10\rangle\)}}
\put(-0.1,-3.9){\makebox(0,0)[rb]{\(\langle 9\rangle\)}}

\put(-4,-6.5){\makebox(0,0)[cc]{\(\Phi_9\)}}
\put(-4.1,-5.9){\makebox(0,0)[rb]{\(\langle 25\rangle\)}}
\put(-2.1,-5.9){\makebox(0,0)[rb]{\(\langle 27\rangle\)}}
\put(-0.1,-5.9){\makebox(0,0)[rb]{\(\langle 26\rangle\)}}

\put(4,-6.5){\makebox(0,0)[cc]{\(\Phi_{10}\)}}
\put(2.1,-5.9){\makebox(0,0)[lb]{\(\langle 28\rangle\)}}
\put(4.1,-5.9){\makebox(0,0)[lb]{\(\langle 30\rangle\)}}
\put(6.1,-5.9){\makebox(0,0)[lb]{\(\langle 29\rangle\)}}

\put(-4,-8.5){\makebox(0,0)[cc]{\(\Phi_{35}\)}}
\put(-6.1,-7.9){\makebox(0,0)[rb]{\(\langle 98\rangle\)}}
\put(-4.1,-7.9){\makebox(0,0)[rb]{\(\langle 97\rangle\)}}
\put(-2.1,-7.9){\makebox(0,0)[rb]{\(\langle 96\rangle\)}}
\put(-0.1,-7.9){\makebox(0,0)[rb]{\(\langle 95\rangle\)}}

\put(4,-8.5){\makebox(0,0)[cc]{\(\Phi_{36}\)}}
\put(2.1,-7.9){\makebox(0,0)[lb]{\(\langle 100\rangle\)}}
\put(4.1,-7.9){\makebox(0,0)[lb]{\(\langle 99\rangle\)}}
\put(6.1,-7.9){\makebox(0,0)[lb]{\(\langle 101\rangle\)}}

\put(0,-13){\makebox(0,0)[cc]{\(G^m_0(0,0)\)}}
\put(-2,-13){\makebox(0,0)[cc]{\(G^m_0(0,1)\)}}
\put(-4,-13){\makebox(0,0)[cc]{\(G^m_0(1,0)\)}}
\put(-6,-13){\makebox(0,0)[cc]{\(G^m_0(-1,0)\)}}
\put(2,-13){\makebox(0,0)[cc]{\(G^m_1(0,-1)\)}}
\put(4,-13){\makebox(0,0)[cc]{\(G^m_1(0,0)\)}}
\put(6,-13){\makebox(0,0)[cc]{\(G^m_1(0,1)\)}}

\put(2.5,0){\makebox(0,0)[cc]{\textbf{TKT:}}}
\put(3.5,0){\makebox(0,0)[cc]{a.1}}
\put(3.5,-0.5){\makebox(0,0)[cc]{\((0000)\)}}
\put(1.8,-0.7){\framebox(2.9,1){}}
\put(-6,-2){\makebox(0,0)[cc]{\textbf{TKT:}}}
\put(-5,-2){\makebox(0,0)[cc]{A.1}}
\put(-5,-2.5){\makebox(0,0)[cc]{\((1111)\)}}
\put(-6.7,-2.7){\framebox(2.9,1){}}

\put(-8,-14){\makebox(0,0)[cc]{\textbf{TKT:}}}
\put(0,-14){\makebox(0,0)[cc]{a.1}}
\put(-2,-14){\makebox(0,0)[cc]{a.2}}
\put(-4,-14){\makebox(0,0)[cc]{a.3}}
\put(-6,-14){\makebox(0,0)[cc]{a.3}}
\put(2,-14){\makebox(0,0)[cc]{a.1}}
\put(4,-14){\makebox(0,0)[cc]{a.1}}
\put(6,-14){\makebox(0,0)[cc]{a.1}}
\put(0,-14.5){\makebox(0,0)[cc]{\((0000)\)}}
\put(-2,-14.5){\makebox(0,0)[cc]{\((1000)\)}}
\put(-4,-14.5){\makebox(0,0)[cc]{\((2000)\)}}
\put(-6,-14.5){\makebox(0,0)[cc]{\((2000)\)}}
\put(2,-14.5){\makebox(0,0)[cc]{\((0000)\)}}
\put(4,-14.5){\makebox(0,0)[cc]{\((0000)\)}}
\put(6,-14.5){\makebox(0,0)[cc]{\((0000)\)}}
\put(-8.7,-14.7){\framebox(15.4,1){}}

\end{picture}
\end{figure}

The symbols \(\Phi_s\) denote isoclinism families
\cite{Hl,Ef,Jm}.
The principalization types or transfer kernel types, briefly TKT,
\cite[Thm. 2.5, p. 478, and Tbl. 6--7, pp. 492--493]{Ma2}
in the bottom rectangle concern
all vertices located vertically above, except indicated otherwise.
Periodicity of length \(2\)
\cite[p. 275]{EkLg},
\(\mathcal{B}(j)\simeq\mathcal{B}(j+2)\) for \(j\ge 4\),
sets in with branch \(\mathcal{B}(4)\), having root of order \(3^4\).

\medskip
\noindent
For \(G\) of coclass \(\mathrm{cc}(G)=1\),
it remains to investigate the structure of the following \(3\)-class groups:

\begin{itemize}

\item
\(\mathrm{Cl}_3(N_1)\) of order \(3^{m-1}=3^3\)
for \(m=4\) (where \(\lbrack\chi_2(G),\gamma_2(G)\rbrack=1\), \(k=0\)),

\item
\(\mathrm{Cl}_3(N_1)\) of order \(3^{m-2}=3^3\)
for \(m=5\) and \(\lbrack\chi_2(G),\gamma_2(G)\rbrack=\gamma_4(G)>1\), \(k=1\),

\item
all four \(\mathrm{Cl}_3(N_i)\) with \(1\le i\le 4\)
of order \(3^2\)
for \(m=3\) (where \(k=0\)).

\end{itemize}


\subsection{Second \(3\)-class groups \(G\) of coclass \(\mathrm{cc}(G)\ge 2\)}
\label{ss:LowHom}

Metabelian \(3\)-groups \(G\) of coclass \(\mathrm{cc}(G)\ge 2\)
must have at least one bicyclic factor \(\gamma_3(G)/\gamma_4(G)\).
Similarly as in \S\ \(2\) of
\cite{Ma1},
we declare
an isomorphism invariant \(e=e(G)\) of \(G\) by
\(e+1=\min\lbrace 3\le j\le m\mid 1\le\lvert\gamma_j(G)/\gamma_{j+1}(G)\rvert\le 3\rbrace\).
This invariant \(2\le e\le m-1\) characterizes
the first cyclic factor \(\gamma_{e+1}(G)/\gamma_{e+2}(G)\) of the lower central series of \(G\),
except \(\gamma_2(G)/\gamma_3(G)\), which is always cyclic.
We can calculate \(e\)
from the \(3\)-exponent \(n\) of the order \(\lvert G\rvert=3^n\)
and the class \(\mathrm{cl}(G)=m-1\), resp. the index \(m\) of nilpotency, of \(G\)
by the formula \(e=n-m+2\).
Since the coclass of \(G\) is given by \(\mathrm{cc}(G)=n-\mathrm{cl}(G)=n-m+1\),
we have the relation \(e=\mathrm{cc}(G)+1\).

For a group \(G\) of coclass \(\mathrm{cc}(G)\ge 2\)
we need a generalization of the group \(\chi_2(G)\).
Denoting by \(m\) the index of nilpotency of \(G\),
we let \(\chi_j(G)\) with \(2\le j\le m-1\)
be the centralizers
of two-step factor groups \(\gamma_j(G)/\gamma_{j+2}(G)\)
of the lower central series, that is,
the biggest subgroups of \(G\) with the property
\(\lbrack\chi_j(G),\gamma_j(G)\rbrack\le\gamma_{j+2}(G)\).
They form an ascending chain of characteristic subgroups of \(G\),
\(\gamma_2(G)\le\chi_2(G)\le\ldots\le\chi_{m-2}(G)<\chi_{m-1}(G)=G\),
which contain the commutator subgroup \(\gamma_2(G)\), and
\(\chi_j(G)\) coincides with \(G\) if and only if \(j\ge m-1\).
Similarly as in \S\ \(2\) of
\cite{Ma1},
we characterize the smallest two-step centralizer
different from the derived subgroup
by an isomorphism invariant
\(s=s(G)=\min\lbrace 2\le j\le m-1\mid\chi_j(G)>\gamma_2(G)\rbrace\).

The following assumptions 
for a metabelian \(3\)-group \(G\) of coclass \(\mathrm{cc}(G)\ge 2\)
with abelianization \(G/\gamma_2(G)\) of type \((3,3)\)
can always be satisfied, according to
\cite[Satz 3.4.5, p. 94]{Ne1}
(see our appendix).

Let \(G\) be a metabelian \(3\)-group of coclass \(\mathrm{cc}(G)\ge 2\)
with abelianization \(G/\gamma_2(G)\) of type \((3,3)\).
Assume that \(G\) has order \(\lvert G\rvert=3^n\),
class \(\mathrm{cl}(G)=m-1\), and invariant \(e=n-m+2\ge 3\),
where \(4\le m<n\le 2m-3\).
Let generators of \(G=\langle x,y\rangle\) be selected such that
\(\gamma_3(G)=\langle y^3,x^3,\gamma_4(G)\rangle\),
\(x\in G\setminus\chi_s(G)\) if \(s<m-1\),
and \(y\in\chi_s(G)\setminus\gamma_2(G)\).
Suppose that a fixed order of the four maximal normal subgroups of \(G\) is defined by
\(M_i=\langle g_i,\gamma_2(G)\rangle\) with
\(g_1=y\), \(g_2=x\), \(g_3=xy\), and \(g_4=xy^{-1}\).
Let the main commutator of \(G\) be declared by
\(s_2=t_2=\lbrack y,x\rbrack\in\gamma_2(G)\)
and higher commutators recursively by
\(s_j=\lbrack s_{j-1},x\rbrack\), \(t_j=\lbrack t_{j-1},y\rbrack\in\gamma_j(G)\)
for \(j\ge 3\).
Starting with the powers \(\sigma_3=y^3\), \(\tau_3=x^3\in\gamma_3(G)\), let
\(\sigma_j=\lbrack\sigma_{j-1},x\rbrack\), \(\tau_j=\lbrack\tau_{j-1},y\rbrack\in\gamma_j(G)\)
for \(j\ge 4\).
With exponents \(-1\le\alpha,\beta,\gamma,\delta,\rho\le 1\) as parameters, let
the following relations be satisfied

\begin{equation}
\label{eqn:LowRel}
s_2^3=\sigma_4\sigma_{m-1}^{-\rho\beta}\tau_4^{-1},\quad
\ s_3\sigma_3\sigma_4=\sigma_{m-2}^{\rho\beta}\sigma_{m-1}^\gamma\tau_e^\delta,\quad
\ t_3^{-1}\tau_3\tau_4=\sigma_{m-2}^{\rho\delta}\sigma_{m-1}^\alpha\tau_e^\beta,\quad
\ \tau_{e+1}=\sigma_{m-1}^{-\rho}.
\end{equation}

\noindent
Finally, let \(\lbrack\chi_s(G),\gamma_e(G)\rbrack=\gamma_{m-k}(G)\)
with the \textit{defect of commutativity} \(0\le k=k(G)\le 1\) of \(G\).
Then, the defect is \(k=0\) if and only if \(\rho=0\).\\
By \(G_\rho^{m,n}(\alpha,\beta,\gamma,\delta)\) we denote 
the representative of an isomorphism class of
metabelian \(3\)-groups \(G\) with \(G/\gamma_2(G)\) of type \((3,3)\),
of coclass \(\mathrm{cc}(G)=n-m+1\ge 2\),
class \(\mathrm{cl}(G)=m-1\), and order \(\lvert G\rvert=3^n\),
which satisfies the relations
(\ref{eqn:LowRel})
with a fixed system of parameters
\((\alpha,\beta,\gamma,\delta,\rho)\).

\begin{theorem}
\label{t:LowHom}
(Incomplete TTT \((\tau_1(G),\tau_2(G))\) for groups \(G\) with \(\mathrm{cc}(G)\ge 3\), \(\mathrm{cl}(G)\ge 5\))

The structure of the \(3\)-class groups \(\mathrm{Cl}_3(N_i)\)
of the four unramified cyclic cubic extension fields \(N_1,\ldots,N_4\)
of an arbitrary base field \(K\)
having a \(3\)-class group \(\mathrm{Cl}_3(K)\) of type \((3,3)\)
and a second \(3\)-class group \(G=\mathrm{Gal}(\mathrm{F}_3^2(K)\vert K)\)
of coclass \(\mathrm{cc}(G)\ge 2\), order \(\lvert G\rvert=3^n\), class \(\mathrm{cl}(G)=m-1\),
and invariant \(e=n-m+2\ge 3\), where \(4\le m<n\le 2m-3\),
is given by the following nearly homocyclic abelian \(3\)-groups.

\begin{eqnarray}
\label{e:LowHom}
\mathrm{Cl}_3(N_1)&\simeq&
\begin{cases}
\mathrm{A}(3,m-1),&\text{ if }\lbrack\chi_s(G),\gamma_e(G)\rbrack=1,\ k=0,\ m\ge 5\,,\\
\mathrm{A}(3,m-2),&\text{ if }\lbrack\chi_s(G),\gamma_e(G)\rbrack=\gamma_{m-1}(G),\ k=1,\ m\ge 6\,,
\end{cases}\\
\mathrm{Cl}_3(N_2)&\simeq&\mathrm{A}(3,e)\text{ for }e\ge 4\,,\\
\mathrm{Cl}_3(N_i)&\simeq&\mathrm{A}(3,3)\text{ for }3\le i\le 4,
\text{ if }\Gamma_i\not\simeq G_0^4(1,0)\simeq\mathrm{Syl}_3A_9\,.
\end{eqnarray}

\end{theorem}

\begin{proof}

For each \(1\le i\le 4\), equation
(\ref{eqn:TwoStgTow})
specifies the order of the Galois group of the \(i\)th little two-stage tower
as \(\lvert\Gamma_i\rvert=\lvert G\rvert/\lvert\gamma_2(M_i)\rvert\),
where \(\lvert G\rvert=3^n\) by assumption.
According to
\cite[Cor. 3.2, p. 480]{Ma1}, the orders of the commutator subgroups of the \(M_i\) are
\[\lvert\gamma_2(M_i)\rvert=
\begin{cases}
3^{e-2}&\text{ for }i=1,\ \lbrack\chi_s(G),\gamma_e(G)\rbrack=1,\ k=0,\\
3^{e-1}&\text{ for }i=1,\ \lbrack\chi_s(G),\gamma_e(G)\rbrack=\gamma_{m-1}(G),\ k=1,\\
3^{m-3}&\text{ for }i=2,\\
3^{n-4}&\text{ for }3\le i\le 4.
\end{cases}\]
Using the relation \(e=n-m+2\), we obtain
\[\lvert\Gamma_i\rvert=
\begin{cases}
3^m&\text{ for }i=1,\ \lbrack \chi_s(G),\gamma_e(G)\rbrack =1,\\
3^{m-1}&\text{ for }i=1,\ \lbrack \chi_s(G),\gamma_e(G)\rbrack =\gamma_{m-1}(G),\\
3^{e+1}&\text{ for }i=2,\\
3^4&\text{ for }3\le i\le 4.
\end{cases}\]
Since \(\Gamma_i\),
having an abelian maximal subgroup \(A_i\simeq\mathrm{Cl}_3(N_i)\),
is a metabelian \(3\)-group of coclass \(\mathrm{cc}(G)=1\) by
\cite[Kor., p. 9]{HeSm}
or also by
\cite[Cor. 3.1--3.2, pp. 476 and 480]{Ma1},
the structure of the abelian maximal normal subgroup \(A_i\) of \(\Gamma_i\)
is given by

\begin{eqnarray*}
\label{e:LowHomPrf}
A_1&\simeq&
\begin{cases}
\mathrm{A}(3,m-1),&\text{ if }\lbrack\chi_s(G),\gamma_e(G)\rbrack=1,\ k=0,\ m\ge 5,\\
\mathrm{A}(3,m-2),&\text{ if }\lbrack\chi_s(G),\gamma_e(G)\rbrack=\gamma_{m-1}(G),\ k=1,\ m-1\ge 5,
\end{cases}\\
A_2&\simeq&\mathrm{A}(3,e),\text{ for }e+1\ge 5,\\
A_i&\simeq&\mathrm{A}(3,3)\text{ for }3\le i\le 4,
\text{ if }\Gamma_i\not\simeq G_0^4(1,0)\simeq\mathrm{Syl}_3A_9,
\end{eqnarray*}

\noindent
according to
\cite[Thm. 3.4]{Bl},
for \(1\le i\le 2\),
and by an immediate analysis of the four isomorphism classes of
metabelian \(3\)-groups of order \(3^m\) with index of nilpotency \(m=4\),
for \(3\le i\le 4\).
Representatives for these four isomorphism classes are three groups
\(G_0^4(0,0)\),
\(G_0^4(0,1)\),
\(G_0^4(-1,0)\),
whose abelian maximal normal subgroup is nearly homocyclic of type \((9,3)\),
and the exceptional group
\(G_0^4(1,0)\simeq\mathrm{Syl}_3A_9\),
which has the elementary abelian \(3\)-group of type \((3,3,3)\)
as its abelian maximal normal subgroup.
See the level of order \(3^4=81\), i. e. the stem of \(\Phi_3\), in Figure
\ref{fig:CoCl1}.
\end{proof}

\noindent
For \(G\) of coclass \(\mathrm{cc}(G)\ge 2\), it remains to investigate
the structure of the following \(3\)-class groups:

\begin{itemize}

\item
\(\mathrm{Cl}_3(N_1)\) of order \(3^{m-k-1}=3^3\)
for small values of the index of nilpotency \(m\),
namely \(m=4\) (where \(\lbrack\chi_s(G),\gamma_e(G)\rbrack=1\), \(k=0\)), and
\(m=5\) if \(\lbrack\chi_s(G),\gamma_e(G)\rbrack=\gamma_{m-1}(G)\), \(k=1\),

\item
\(\mathrm{Cl}_3(N_2)\) of order \(3^e=3^3\)
for groups \(G\) of coclass \(\mathrm{cc}(G)=2\) (where \(e=3\)),

\item
\(\mathrm{Cl}_3(N_i)\), \(3\le i\le 4\),
of order \(3^3\)
without restrictions for the parameters \(m\ge 4\) and \(e\ge 3\).

\end{itemize}

\renewcommand{\arraystretch}{1.0}
\begin{table}[ht]
\caption{\(3\)-class groups to be investigated for certain parameters \(m,n,e,k\) of \(G\)}
\label{tab:LowUnk}
\begin{center}
\begin{tabular}{|c|c|c|c|cccc|}
\hline
      \(m\) &      \(n\) &      \(e\) &      \(k\) & \(\mathrm{Cl}_3(N_1)\) & \(\mathrm{Cl}_3(N_2)\) & \(\mathrm{Cl}_3(N_3)\) & \(\mathrm{Cl}_3(N_4)\) \\
\hline
      \(4\) &      \(5\) &      \(3\) &      \(0\) &             \(\times\) &             \(\times\) &             \(\times\) &             \(\times\) \\
      \(5\) &      \(6\) &      \(3\) &      \(1\) &             \(\times\) &             \(\times\) &             \(\times\) &             \(\times\) \\
\hline
      \(5\) &      \(6\) &      \(3\) &      \(0\) &                        &             \(\times\) &             \(\times\) &             \(\times\) \\
  \(\ge 6\) &  \(\ge 7\) &      \(3\) &  \(\ge 0\) &                        &             \(\times\) &             \(\times\) &             \(\times\) \\
\hline
      \(5\) &      \(7\) &      \(4\) &      \(0\) &                        &                        &             \(\times\) &             \(\times\) \\
  \(\ge 6\) &  \(\ge 8\) &  \(\ge 4\) &  \(\ge 0\) &                        &                        &             \(\times\) &             \(\times\) \\
\hline
\end{tabular}
\end{center}
\end{table}

In Table
\ref{tab:LowUnk},
we give an overview of all systems \((m,n,e,k)\) of parameters
of the second \(3\)-class group \(G\) of coclass \(\mathrm{cc}(G)\ge 2\),
for which the \(3\)-class groups \(\mathrm{Cl}_3(N_i)\),
\(1\le i\le 4\), marked by the symbol \(\times\) have to be analyzed.


\section{Searching for \(3\)-class groups of \(3\)-rank three}
\label{s:MaxLowExo}

Several authors, namely Scholz
\cite[p. 218]{So},
Kisilevsky
\cite[Thm. 3, p. 205]{Ki1},
Heider and Schmithals
\cite[Satz 7, p. 11]{HeSm},
and Brink
\cite[pp. 51--52]{Br},
have pointed out the theoretical possibility
that the elementary abelian \(3\)-group of order \(27\)
can occur as the \(3\)-class group \(\mathrm{Cl}_3(N)\)
of an unramified cyclic cubic extension field \(N\)
of a base field \(K\) with \(3\)-class group \(\mathrm{Cl}_3(K)\) of type \((3,3)\).
The most explicit result among these statements is due to Heider and Schmithals.
They prove that the occurrence of an
elementary abelian \(3\)-class group \(\mathrm{Cl}_3(N)\) of order \(27\)
is restricted to extensions \(N\vert K\) satisfying condition \((\mathrm{B})\)
in the sense of Taussky
\cite[p. 435]{Ta},
that is, having a partial principalization without fixed point
\cite[\S\ 2.2, p. 476]{Ma2}.

Since the nearly homocyclic abelian \(3\)-groups \(\mathrm{A}(3,n)\) with \(n\ge 2\),
which generally occur as \(\mathrm{Cl}_3(N)\),
according to the theorems
\ref{t:MaxHom}
and
\ref{t:LowHom},
are of \(3\)-rank two,
the elementary abelian \(3\)-group of type \((3,3,3)\)
is the unique possiblity for \(\mathrm{Cl}_3(N)\)
to be of \(3\)-rank three.
In contrast, many other possibilities arise when
\(\mathrm{Cl}_3(K)\) is of type \((9,3)\)
\cite{Ma5}.

Unfortunately, it was impossible to find a numerical example,
let alone a general criterion,
for the occurrence of \(3\)-class groups of type \((3,3,3)\)
in the bibliography, up to now.
At first it was completely unknown,
whether \(3\)-class groups of type \((3,3,3)\) exist at all,
if they appear sporadically or stochastically,
or if their occurrence is ruled by deterministic laws.

In the present paper we systematically analyze this question
by means of presentations
for the second \(3\)-class group \(G=\mathrm{Gal}(\mathrm{F}_3^2(K)\vert K)\),
which have been given by Blackburn
\cite[pp. 82--84]{Bl}
in the case of coclass \(\mathrm{cc}(G)=1\)
and by Nebelung
\cite[Satz 3.4.5, p. 94]{Ne1}
(see our appendix)
in the case of coclass \(\mathrm{cc}(G)\ge 2\).
We arrive at the surprising result that
the transfer target type (TTT) \(\tau(K)\),
resp. the number \(\varepsilon=\varepsilon(K)\)
of elementary abelian \(3\)-groups of type \((3,3,3)\)
among the \(3\)-class groups \(\mathrm{Cl}_3(N_i)\), \(1\le i\le 4\),
is connected with
the transfer kernel type (TKT) \(\varkappa(K)\),
that is the principalization type
of \(3\)-classes of \(K\) in \(N_i\), \(1\le i\le 4\),
by strict rules.

These connections offer new algorithmic possibilities
for computing the principalization type of \(3\)-classes
of a \textit{quadratic} base field \(K=\mathbb{Q}(\sqrt{D})\)
with \(3\)-class group of type \((3,3)\),
independently from the classical algorithm for complex quadratic fields
by Scholz and Taussky
\cite{SoTa}
and its modification for real quadratic fields
by Heider and Schmithals
\cite{HeSm}.
The new algorithm is based on determining the structure
of the \(3\)-class groups \(\mathrm{Cl}_3(N)\)
of the unramified cyclic cubic extensions \(N\vert K\),
that is, of four \(S_3\)-fields \(N_1,\ldots,N_4\) of absolute degree six,
and is described in \S\
\ref{s:CompTech}.

With the aid of an implementation of this algorithm
in the number theoretical computer algebra system PARI/GP
\cite{PARI},
we have computed the principalization type of \(3\)-classes
of all \(4\,596\) quadratic base fields \(K\)
with \(3\)-class group of type \((3,3)\)
and discriminant \(-10^6<D<10^7\).
The resulting extensive statistics of principalization types and
second \(3\)-class groups \(G=\mathrm{Gal}(\mathrm{F}_3^2(K)\vert K)\)
will be presented in \S\
\ref{s:NumerTab}.

\smallskip
As in \S\
\ref{s:NearHom},
we distinguish second \(3\)-class groups
\(G=\mathrm{Gal}(\mathrm{F}_3^2(K)\vert K)\)
of coclass \(\mathrm{cc}(G)=1\) in \S\
\ref{ss:MaxExo}
and of coclass \(\mathrm{cc}(G)\ge 2\) in \S\
\ref{ss:LowExo}.


\subsection{Second \(3\)-class groups \(G\) of coclass \(\mathrm{cc}(G)=1\)}
\label{ss:MaxExo}

\begin{theorem}
\label{t:MaxExo}
(Transfer target type \(\tau(G)\) for groups \(G\) with \(\mathrm{cc}(G)=1\), \(\mathrm{cl}(G)\le 4\))

The structure of the \(3\)-class groups \(\mathrm{Cl}_3(N_i)\)
of the four unramified cyclic cubic extensions \(N_i\) \((1\le i\le 4)\)
of an arbitrary base field \(K\)
with \(3\)-class group \(\mathrm{Cl}_3(K)\) of type \((3,3)\)
and second \(3\)-class group \(G=\mathrm{Gal}(\mathrm{F}_3^2(K)\vert K)\)
of coclass \(\mathrm{cc}(G)=1\) and order \(\lvert G\rvert=3^m\)
ist given by the following array
for small indices of nilpotency \(3\le m=\mathrm{cl}(G)+1\le 5\).\\
In the case \(m=3\), it is supposed that the generating element \(y\) of the maximal normal subgroup
\(M_1=\mathrm{Gal}(\mathrm{F}_3^2(K)\vert N_1)=\langle y,s_2\rangle\) of \(G\)
has order \(3\).

\begin{eqnarray}
\label{e:MaxExo}
\mathrm{Cl}_3(N_1)&\simeq&
\begin{cases}
\mathrm{A}(3,2),&\text{ if }m=3,\\
\mathrm{A}(3,3),&\text{ if }m=4,\ G\not\simeq G_0^4(1,0)\simeq\mathrm{Syl}_3A_9,\\
C_3\times C_3\times C_3,
&\text{ if }m=4,\ G\simeq G_0^4(1,0)\simeq\mathrm{Syl}_3A_9,\\
\mathrm{A}(3,3),&\text{ if }k=1,\ m=5,
\end{cases}\\
\mathrm{Cl}_3(N_i)&\simeq&
\begin{cases}
\mathrm{A}(3,2),&\text{ if }G\simeq G_0^3(0,0),\\
C_9,&\text{ if }G\simeq G_0^3(0,1),
\end{cases}\text{ for }2\le i\le 4,\text{ if }m=3.
\end{eqnarray}

\end{theorem}

\begin{proof}

We make use of the well-known properties
of the six isomorphism classes of metabelian \(3\)-groups
\(G_a^m(z,w)\) of coclass \(\mathrm{cc}(G)=1\)
with the smallest indices of nilpotency \(3\le m\le 4\).

We start with the distinguished extension \(N_1\).

First, let \(\lbrack\chi_2(G),\gamma_2(G)\rbrack=1\), that is \(k=0\).
Then the group \(\Gamma_1\) is isomorphic to \(G\)
and is of order \(\lvert\Gamma_1\rvert=3^m\),
according to equation
(\ref{eqn:TwoStgTow})
and
\cite[Cor. 3.1, p. 476]{Ma1}.

For \(m=4\), the abelian maximal normal subgroup \(A_1\simeq M_1\)
is either nearly homocyclic of type \((9,3)\) and thus isomorphic to \(\mathrm{A}(3,3)\),
if \(G\) belongs to one of the three isomorphism classes
\(G_0^4(0,0)\), \(G_0^4(0,1)\), \(G_0^4(-1,0)\),
or elementary abelian of type \((3,3,3)\),
if \(G\) is isomorphic to \(G_0^4(1,0)\),
the \(3\)-Sylow subgroup \(\mathrm{Syl}_3A_9\)
of the alternating group of degree \(9\).
See the stem of \(\Phi_3\) in Figure
\ref{fig:CoCl1}.

However, for \(m=3\) the structure
of the abelian maximal normal subgroup \(A_1\simeq M_1\)
can only be nearly homocyclic of type \((3,3)\) and thus isomorphic to \(\mathrm{A}(3,2)\),
in the case of the extra special group \(G\simeq G_0^3(0,0)\)
of exponent \(3\)
with four abelian maximal normal subgroups of type \((3,3)\),
as well as in the case of the extra special group \(G\simeq G_0^3(0,1)\)
of exponent \(9\)
with one abelian maximal normal subgroup
\(A_1\simeq M_1=\langle y,s_2\rangle\) of type \((3,3)\),
which is distinguished by our choice of the generating element \(y\),
and three further cyclic maximal normal subgroups of order \(9\).
See the level of order \(3^3=27\), i. e. the stem of \(\Phi_2\), in Figure
\ref{fig:CoCl1}.

Now we consider
\(\lbrack\chi_2(G),\gamma_2(G)\rbrack=\gamma_{m-1}(G)\), that is \(k=1\),
with \(m=5\) and thus \(\gamma_{4}(G)>1\).
According to equation
(\ref{eqn:TwoStgTow})
and
\cite[Cor. 3.1, p. 476]{Ma1},
the group \(\Gamma_1\simeq G/\gamma_{4}(G)\)
is the immediate predecessor of \(G\)
\cite[p. 182]{Ne1}
on the coclass graph \(\mathcal{G}(3,1)\).
Thus it is isomorphic to \(G_0^4(0,0)\)
having the unique abelian maximal normal subgroup \(A_1\simeq\mathrm{A}(3,3)\).
See the vertex \(\langle 81,9\rangle\) in Figure
\ref{fig:CoCl1}.

For the other three extensions \(N_i\) with \(2\le i\le 4\),
all groups \(\Gamma_i\simeq G/\gamma_3(G)\simeq G\)
\cite[Cor. 3.1, p. 476]{Ma1}
coincide with \(G\), since \(\gamma_3(G)=1\) is trivial for \(m=3\).
The structure of the abelian maximal normal subgroup \(A_i\simeq M_i\)
is either of type \((3,3)\) and thus isomorphic to \(\mathrm{A}(3,2)\),
if \(G\) is isomorphic to \(G_0^3(0,0)\),
or cyclic of order \(9\),
if \(G\) is isomorphic to \(G_0^3(0,1)\),
taking into account that \(M_i=\langle xy^{i-2},s_2\rangle=\langle xy^{i-2}\rangle\)
with elements \(xy^{i-2}\) of order \(9\) whose third power coincides with \(s_2\).
See the vertices \(\langle 27,3\rangle\) and \(\langle 27,4\rangle\) in Figure
\ref{fig:CoCl1}.
\end{proof}

\begin{corollary}
\label{c:MaxExo}
(TTT \(\tau(G)\) of \(C_3\times C_3\) and
stem groups \(G\) in isoclinism families \(\Phi_2,\Phi_3,\Phi_{10}\))\\
Table \ref{tab:MaxExo} gives
the structure of \(3\)-class groups \(\mathrm{Cl}_3(N_i)\), \(1\le i\le 4\),
and invariant \(\varepsilon=\#\lbrace 1\le i\le 4\mid\mathrm{Cl}_3(N_i)\simeq (3,3,3)\rbrace\),
for 3-groups \(G\in\mathcal{G}(3,1)\) of small nilpotency class \(1\le\mathrm{cl}(G)=m-1\le 4\)
in dependence on the principalization or transfer kernel type (TKT) \(\varkappa\)
\cite[Thm. 2.4--2.5, p. 478]{Ma2}.

\end{corollary}

\renewcommand{\arraystretch}{1.0}
\begin{table}[ht]
\caption{TTT \(\tau(G)\) for \(G\) of coclass \(\mathrm{cc}(G)=1\)
and index of nilpotency \(2\le m\le 5\)}
\label{tab:MaxExo}
\begin{center}
\begin{tabular}{|cc|lc|c|cccc|c|}
\hline
 \(m\) & \(k\) & Type           & \(\varkappa\) & \(\mathrm{Cl}_3(\mathrm{F}_3^1(K))\) & \(\mathrm{Cl}_3(N_1)\) & \(\mathrm{Cl}_3(N_2)\) & \(\mathrm{Cl}_3(N_3)\) & \(\mathrm{Cl}_3(N_4)\) & \(\varepsilon\) \\
\hline                          
 \(2\) & \(0\) & a.1            & \((0000)\)    &                                \(1\) &                \((3)\) &                \((3)\) &                \((3)\) &                \((3)\) &           \(0\) \\
\hline                          
 \(3\) & \(0\) & a.1            & \((0000)\)    &                              \((3)\) &              \((3,3)\) &              \((3,3)\) &              \((3,3)\) &              \((3,3)\) &           \(0\) \\
 \(3\) & \(0\) & A.1            & \((1111)\)    &                              \((3)\) &              \((3,3)\) &                \((9)\) &                \((9)\) &                \((9)\) &           \(0\) \\
\hline                          
 \(4\) & \(0\) & a.1            & \((0000)\)    &                            \((3,3)\) &              \((9,3)\) &              \((3,3)\) &              \((3,3)\) &              \((3,3)\) &           \(0\) \\
 \(4\) & \(0\) & a.2            & \((1000)\)    &                            \((3,3)\) &              \((9,3)\) &              \((3,3)\) &              \((3,3)\) &              \((3,3)\) &           \(0\) \\
 \(4\) & \(0\) & a.3            & \((2000)\)    &                            \((3,3)\) &              \((9,3)\) &              \((3,3)\) &              \((3,3)\) &              \((3,3)\) &           \(0\) \\
 \(4\) & \(0\) & a.3\({}^\ast\) & \((2000)\)    &                            \((3,3)\) &            \((3,3,3)\) &              \((3,3)\) &              \((3,3)\) &              \((3,3)\) &           \(1\) \\
\hline
 \(5\) & \(1\) & a.1            & \((0000)\)    &                            \((9,3)\) &              \((9,3)\) &              \((3,3)\) &              \((3,3)\) &              \((3,3)\) &           \(0\) \\
\hline
\end{tabular}
\end{center}
\end{table}

\begin{example}
\label{ex:MaxExo}
The first occurrences of second \(3\)-class groups
\(G=\mathrm{Gal}(\mathrm{F}_3^2(K)\vert K)\) of coclass \(\mathrm{cc}(G)=1\)
with invariants \(m=n=4\), \(e=2\)
for real quadratic fields \(K=\mathbb{Q}(\sqrt{D})\)
with discriminant \(0<D<10^7\)
and \(3\)-class group of type \((3,3)\)
turned out to be the following.
The smallest discriminant \(D\)
with TKT \(\mathrm{a.3}\), resp. \(\mathrm{a.2}\),
is \(32\,009\), resp. \(72\,329\), according to
\cite[Tbl. 7, p. 24]{HeSm}.
The smallest discriminant \(D\) with TKT \(\mathrm{a.3}{}^\ast\),
is \(142\,097\), known from
\cite[Part IV]{Ma0}.
However,
its special feature \(\varepsilon=1\), resp. \(\mathrm{Cl}_3(N_1)\simeq C_3\times C_3\times C_3\),
was unknown up to \(2009\).
\end{example}

\begin{conjecture}
\label{cnj:MaxTrm}
For quadratic fields,
the TKT \(\mathrm{a.1}\)
cannot occur with a second \(3\)-class group
of defect \(k=0\),
i. e., an infinitely capable vertex
located on the mainline of coclass graph \(\mathcal{G}(3,1)\).
\end{conjecture}


\subsection{Second \(3\)-class groups \(G\) of coclass \(\mathrm{cc}(G)\ge 2\)}
\label{ss:LowExo}

Suppose the order of a \(3\)-class group
\(\mathrm{Cl}_3(N_i)\simeq M_i/\gamma_2(M_i)\) with \(1\le i\le 4\)
has turned out to be \(27\).
The basic idea for deciding whether this \(3\)-class group
is elementary abelian of type \((3,3,3)\) or nearly homocyclic of type \((9,3)\)
consists in estimating the order of the cosets \(\overline{\upsilon}=\upsilon\gamma_2(M_i)\)
of the generators
\(\upsilon\in\lbrace g_i,s_2,\sigma_3,\ldots,\sigma_{m-1},\tau_3,\ldots,\tau_e\rbrace\)
of \(M_i=\langle g_i,\gamma_2(G)\rangle\)
\cite[Thm. 3.3, Proof, pp. 478--479]{Ma1}
with respect to the commutator subgroup \(\gamma_2(M_i)\).
If all these orders are bounded from above by \(3\),
then we have an elementary abelian \(3\)-class group of type \((3,3,3)\),
otherwise a nearly homocyclic \(3\)-class group of type \((9,3)\).

To reduce the investigation to the most important generator \(g_i\),
we first summarize general facts
concerning the columns of Table
\ref{tab:LowUnk},
that is, the four \(3\)-class groups \(\mathrm{Cl}_3(N_i)\), \(1\le i\le 4\),
in the following three \S\S\ 
\ref{sss:LowFst}--\ref{sss:LowTrd}.


\subsubsection{The distinguished \(3\)-class group \(\mathrm{Cl}_3(N_1)\simeq M_1/\gamma_2(M_1)\)}
\label{sss:LowFst}

According to
\cite[Cor. 3.2, p. 480]{Ma1},
the maximal subgroup \(M_1<G=\langle x,y\rangle\)
with generator \(y\), distinguished by the conditions 
\(\gamma_3(G)=\langle y^3,x^3,\gamma_4(G)\rangle\) and
\(y\in\chi_s(G)\setminus\gamma_2(G)\), has the following properties.

\begin{eqnarray*}
\label{e:LowFst}
M_1&=&\langle y,\gamma_2(G)\rangle=\langle y,s_2,\sigma_3,\ldots,\sigma_{m-1},\tau_3,\ldots,\tau_e\rangle,\\
\gamma_2(M_1)&=&\langle t_3,\tau_4,\ldots,\tau_{e+1}\rangle.
\end{eqnarray*}

\noindent
Since \(\tau_4,\ldots,\tau_e\in\gamma_2(M_1)\) for \(e\ge 4\),
the order of the cosets \(\overline{\tau_4},\ldots,\overline{\tau_e}\) equals \(1\).

\noindent
The relation \(\tau_3^3\tau_4^3\tau_5=1\) for third powers
\cite[Thm. 3.3, Proof, p. 478]{Ma1}
implies \(\tau_3^3=\tau_4^{-3}\tau_5^{-1}\in\gamma_2(M_1)\) and \(\mathrm{ord}(\overline{\tau_3})\le 3\),
for any \(e\ge 3\).

\noindent
For the order of the cosets
\(\overline{s_2},\overline{\sigma_3},\ldots,\overline{\sigma_{m-1}}\),
we cannot ensure the upper bound \(3\), in general.
However,
in the two cases \(m=4\), \(n=5\), \(\rho=0\) and \(m=5\), \(n=6\), \(\rho=\pm 1\)
to be investigated, according to Table
\ref{tab:LowUnk},
this estimate is possible.
For \(m=4\), \(n=5\), \(\rho=0\), we have
the nilpotency relation \(\sigma_4=\sigma_m=1\),
\(\sigma_3^3=\sigma_4^{-3}\sigma_5^{-1}=1\), and
\(s_2^3=\sigma_4\tau_4^{-1}=1\) by
(\ref{eqn:LowRel}),
since \(e=3\), \(\rho=0\), and thus \(\tau_4=\tau_{e+1}=1\).
For \(m=5\), \(n=6\), \(\rho=\pm 1\), we have
the nilpotency relation \(\sigma_5=\sigma_m=1\),
\(\sigma_4^3=\sigma_5^{-3}\sigma_6^{-1}=1\),
\(\sigma_3^3=\sigma_4^{-3}\sigma_5^{-1}=1\), and
\(s_2^3=\sigma_4\sigma_{m-1}^{-\rho\beta}\tau_4^{-1}=\sigma_4^{1-\rho(\beta-1)}\) by
(\ref{eqn:LowRel}),
where \(\sigma_4=\sigma_{m-1}=\tau_{e+1}^{-\rho}=\tau_4^{-\rho}\in\gamma_2(M_1)\).

\noindent
Consequently, we must only determine the order of the coset of the generator
\(y\) with third power \(y^3=\sigma_3\).


\subsubsection{The distinguished \(3\)-class group \(\mathrm{Cl}_3(N_2)\simeq M_2/\gamma_2(M_2)\)}
\label{sss:LowSnd}

By
\cite[Cor. 3.2, p. 480]{Ma1},
the maximal subgroup \(M_2<G=\langle x,y\rangle\)
with generator \(x\), distinguished by the conditions 
\(\gamma_3(G)=\langle y^3,x^3,\gamma_4(G)\rangle\) and
\(x\in G\setminus\chi_s(G)\) if \(s<m-1\),
has the following properties.

\begin{eqnarray*}
\label{e:LowSnd}
M_2&=&\langle x,\gamma_2(G)\rangle=\langle x,s_2,\sigma_3,\ldots,\sigma_{m-1},\tau_3,\ldots,\tau_e\rangle,\\
\gamma_2(M_2)&=&\langle s_3,\sigma_4,\ldots,\sigma_{m-1}\rangle.
\end{eqnarray*}

\noindent
Since \(\sigma_4,\ldots,\sigma_{m-1}\in\gamma_2(M_2)\) for \(m\ge 5\),
the order of the cosets
\(\overline{\sigma_4},\ldots,\overline{\sigma_{m-1}}\) equals \(1\).

\noindent
The relation \(\sigma_3^3\sigma_4^3\sigma_5=1\) for third powers
\cite[Thm. 3.3, Proof, p. 478]{Ma1}
implies \(\sigma_3^3=\sigma_4^{-3}\sigma_5^{-1}\in\gamma_2(M_1)\) and \(\mathrm{ord}(\overline{\sigma_3})\le 3\),
for any \(m\ge 4\).

\noindent
Since only the case \(e=3\) is to be investigated, according to Table
\ref{tab:LowUnk},
we have 
\(\tau_3^3=\tau_4^{-3}=\tau_{e+1}^{-3}=\sigma_{m-1}^{3\rho}\in\gamma_2(M_2)\) by
(\ref{eqn:LowRel}),
because \(\rho=0\) for \(m=4\) and \(\sigma_{m-1}\in\gamma_2(M_2)\) for \(m\ge 5\).

\noindent
Finally, we have
\(s_2^3=\sigma_4\sigma_{m-1}^{-\rho\beta}\tau_4^{-1}
=\sigma_4\sigma_{m-1}^{-\rho(\beta-1)}\in\gamma_2(M_2)\),
since \(\sigma_4\in\gamma_2(M_2)\) for \(m\ge 4\),
\(\rho=0\) for \(m=4\),
and \(\sigma_{m-1}\in\gamma_2(M_2)\) for \(m\ge 5\),
whence \(\mathrm{ord}(\overline{s_2})\le 3\).

\noindent
Therefore, it only remains to investigate the order of the coset of the generator
\(x\) with third power \(x^3=\tau_3\).


\subsubsection{The other \(3\)-class groups \(\mathrm{Cl}_3(N_i)\simeq M_i/\gamma_2(M_i)\), \(3\le i\le 4\)}
\label{sss:LowTrd}

According to
\cite[Cor. 3.2, p. 480]{Ma1},
the maximal subgroups \(M_i<G=\langle x,y\rangle\) with \(3\le i\le 4\)
with generators \(xy\) and \(xy^{-1}\), having third powers
\((xy)^3,(xy^{-1})^3\in\zeta_{1+k}(G)\) in the first or second centre of \(G\)
\cite[Lem. 3.4.11, p. 105]{Ne1},
have the following properties.

\begin{eqnarray*}
\label{e:LowTrd}
M_3&=&\langle xy,\gamma_2(G)\rangle=\langle xy,s_2,\sigma_3,\ldots,\sigma_{m-1},\tau_3,\ldots,\tau_e\rangle,\\
\gamma_2(M_3)&=&\langle s_3t_3,\gamma_4(G)\rangle,\\
M_4&=&\langle xy^{-1},\gamma_2(G)\rangle=\langle xy^{-1},s_2,\sigma_3,\ldots,\sigma_{m-1},\tau_3,\ldots,\tau_e\rangle,\\
\gamma_2(M_4)&=&\langle s_3t_3^{-1},\gamma_4(G)\rangle,\\
\gamma_4(G)&=&\langle \sigma_4,\ldots,\sigma_{m-1},\tau_4,\ldots,\tau_e\rangle.
\end{eqnarray*}

\noindent
Since
\(\sigma_4,\ldots,\sigma_{m-1},\tau_4,\ldots,\tau_e\in\gamma_4(G)<\gamma_2(M_i)\)
for \(e\ge 4\), \(m\ge 5\),
the order of the cosets
\(\overline{\sigma_4},\ldots,\overline{\sigma_{m-1}}\) and
\(\overline{\tau_4},\ldots,\overline{\tau_e}\)
equals \(1\).

\noindent
Due to the relations
\(\sigma_3^3\sigma_4^3\sigma_5=1\) and \(\tau_3^3\tau_4^3\tau_5=1\)
for third powers
\cite[Thm. 3.3, Proof, p. 478]{Ma1},
we have
\(\mathrm{ord}(\overline{\sigma_3})\le 3\)
and \(\mathrm{ord}(\overline{\tau_3})\le 3\).

\noindent
Finally, we have
\(s_2^3=\sigma_4\sigma_{m-1}^{-\rho\beta}\tau_4^{-1}\in\gamma_2(M_i)\) by
(\ref{eqn:LowRel}),
since \(\sigma_4,\tau_4\in\gamma_2(M_i)\) for \(m\ge 4\), \(e\ge 3\),
\(\rho=0\) for \(m=4\)
and \(\sigma_{m-1}\in\gamma_2(M_i)\) for \(m\ge 5\),
whence \(\mathrm{ord}(\overline{s_2})\le 3\).

\noindent
Thus,
it only remains to determine the order of the coset of the generator
\(xy\), resp. \(xy^{-1}\),
for \(i=3\), resp. \(i=4\).
We summarize the results of \S\S\
\ref{sss:LowFst}--\ref{sss:LowTrd}
in Lemma
\ref{l:LowExo}.

\begin{lemma}
\label{l:LowExo}

For each of the four \(3\)-class groups
\(\mathrm{Cl}_3(N_i)\simeq M_i/\gamma_2(M_i)\), \(1\le i\le 4\),
the decision
whether \(\mathrm{Cl}_3(N_i)\) is of type \((9,3)\) or of type \((3,3,3)\),
in the case of \(3\)-class number \(\mathrm{h}_3(N_i)=3^3\),
exclusively depends on the order of the generator \(g_i\)
of \(M_i=\langle g_i,\gamma_2(G)\rangle\)
with respect to the commutator subgroup \(\gamma_2(M_i)\).\\
The order of all the generators
\(s_2,\sigma_3,\ldots,\sigma_{m-1},\tau_3,\ldots,\tau_e\)
of \(\gamma_2(G)\)
with respect to \(\gamma_2(M_i)\)
is uniformly bounded from above by \(3\).

\end{lemma}


After the preliminaries in the last three sections
we come to the details of the rows of Table
\ref{tab:LowUnk}
in the following four sections.

\begin{figure}[ht]
\caption{Sporadic groups and roots of coclass trees on the coclass graph \(\mathcal{G}(3,2)\)}
\label{fig:Typ33CoCl2}

\setlength{\unitlength}{1cm}
\begin{picture}(16,15)(-1,-12)
\put(0,2.5){\makebox(0,0)[cb]{Order \(3^n\)}}
\put(0,2){\line(0,-1){12}}
\multiput(-0.1,2)(0,-2){7}{\line(1,0){0.2}}
\put(-0.2,2){\makebox(0,0)[rc]{\(9\)}}
\put(0.2,2){\makebox(0,0)[lc]{\(3^2\)}}
\put(-0.2,0){\makebox(0,0)[rc]{\(27\)}}
\put(0.2,0){\makebox(0,0)[lc]{\(3^3\)}}
\put(-0.2,-2){\makebox(0,0)[rc]{\(81\)}}
\put(0.2,-2){\makebox(0,0)[lc]{\(3^4\)}}
\put(-0.2,-4){\makebox(0,0)[rc]{\(243\)}}
\put(0.2,-4){\makebox(0,0)[lc]{\(3^5\)}}
\put(-0.2,-6){\makebox(0,0)[rc]{\(729\)}}
\put(0.2,-6){\makebox(0,0)[lc]{\(3^6\)}}
\put(-0.2,-8){\makebox(0,0)[rc]{\(2\,187\)}}
\put(0.2,-8){\makebox(0,0)[lc]{\(3^7\)}}
\put(-0.2,-10){\makebox(0,0)[rc]{\(6\,561\)}}
\put(0.2,-10){\makebox(0,0)[lc]{\(3^8\)}}

\put(2.2,2.2){\makebox(0,0)[lc]{\(C_3\times C_3\)}}
\put(1.9,1.9){\framebox(0.2,0.2){}}
\put(2,2){\line(0,-1){2}}
\put(2,0){\circle*{0.2}}
\put(2.2,0.2){\makebox(0,0)[lc]{\(G^3_0(0,0)\)}}

\put(2,0){\line(1,-4){1}}
\put(2,0){\line(1,-2){2}}
\put(2,0){\line(1,-1){4}}
\put(2,0){\line(3,-2){6}}
\put(2,0){\line(2,-1){8}}
\put(2,0){\line(5,-2){10}}
\put(2,0){\line(3,-1){12}}
\put(7,-1){\makebox(0,0)[lc]{Edges of depth \(2\) forming the interface}}
\put(8,-1.5){\makebox(0,0)[lc]{between \(\mathcal{G}(3,1)\) and \(\mathcal{G}(3,2)\)}}

\put(2,-4){\makebox(0,0)[cc]{\(\Phi_6\)}}
\multiput(3,-4)(1,0){2}{\circle{0.2}}
\put(3.1,-3.9){\makebox(0,0)[lb]{\(\langle 5\rangle\)}}
\put(4.1,-3.9){\makebox(0,0)[lb]{\(\langle 7\rangle\)}}
\multiput(6,-4)(2,0){5}{\circle{0.2}}
\put(6.1,-3.9){\makebox(0,0)[lb]{\(\langle 9\rangle\)}}
\put(8.1,-3.9){\makebox(0,0)[lb]{\(\langle 4\rangle\)}}
\put(10.1,-3.9){\makebox(0,0)[lb]{\(\langle 3\rangle\)}}
\put(12.1,-3.9){\makebox(0,0)[lb]{\(\langle 6\rangle\)}}
\put(14.1,-3.9){\makebox(0,0)[lb]{\(\langle 8\rangle\)}}

\put(2,-5){\makebox(0,0)[cc]{\textbf{TKT:}}}
\put(3,-5){\makebox(0,0)[cc]{D.10}}
\put(3,-5.5){\makebox(0,0)[cc]{\((2241)\)}}
\put(4,-5){\makebox(0,0)[cc]{D.5}}
\put(4,-5.5){\makebox(0,0)[cc]{\((4224)\)}}
\put(1.5,-5.75){\framebox(3,1){}}

\put(5.5,-6.5){\makebox(0,0)[cc]{\(\Phi_{43}\)}}
\put(6,-4){\line(0,-1){2}}
\put(5.9,-5.9){\makebox(0,0)[rc]{\(\langle 57\rangle\)}}
\put(6,-4){\line(1,-4){0.5}}
\multiput(6,-6)(0.5,0){2}{\circle{0.1}}
\multiput(6,-6)(0.5,0){2}{\line(0,-1){2}}
\multiput(5.95,-8.05)(0.5,0){2}{\framebox(0.1,0.1){}}

\put(7.5,-6.5){\makebox(0,0)[cc]{\(\Phi_{42}\)}}
\put(8,-4){\line(0,-1){2}}
\put(7.9,-5.9){\makebox(0,0)[rc]{\(\langle 45\rangle\)}}
\put(8,-4){\line(1,-4){0.5}}
\put(8,-4){\line(1,-2){1}}
\put(8,-4){\line(3,-4){1.5}}
\multiput(8,-6)(0.5,0){4}{\circle{0.1}}
\multiput(8,-6)(0.5,0){2}{\line(0,-1){2}}
\multiput(7.95,-8.05)(0.5,0){2}{\framebox(0.1,0.1){}}
\multiput(7.9,-7.9)(0.5,0){2}{\makebox(0,0)[rc]{\(4\ast\)}}

\put(5,-8.5){\makebox(0,0)[cc]{\textbf{TKT:}}}
\put(6.25,-8.5){\makebox(0,0)[cc]{G.19}}
\put(6.25,-9){\makebox(0,0)[cc]{\((2143)\)}}
\put(8.75,-8.5){\makebox(0,0)[cc]{H.4}}
\put(8.75,-9){\makebox(0,0)[cc]{\((4443)\)}}
\put(4.5,-9.25){\framebox(5,1){}}


\put(11.25,-6.5){\makebox(0,0)[cc]{\(\Phi_{40},\Phi_{41}\)}}
\multiput(10,-4)(0,-2){2}{\line(0,-1){2}}
\multiput(10,-6)(0,-2){2}{\circle{0.2}}
\put(10.1,-5.9){\makebox(0,0)[lc]{\(\langle 40\rangle\)}}
\put(10,-8){\vector(0,-1){2}}
\put(10,-10){\makebox(0,0)[ct]{\(\mathcal{T}(\langle 729,40\rangle)\)}}
\put(10,-4){\line(3,-4){1.5}}
\put(11.5,-6){\circle{0.1}}
\put(11.4,-5.9){\makebox(0,0)[rc]{\(6\ast\)}}

\multiput(12,-4)(0,-2){2}{\line(0,-1){2}}
\multiput(12,-6)(0,-2){2}{\circle{0.2}}
\put(12.1,-5.9){\makebox(0,0)[lc]{\(\langle 49\rangle\)}}
\put(12,-8){\vector(0,-1){2}}
\put(12,-10){\makebox(0,0)[ct]{\(\mathcal{T}(\langle 243,6\rangle)\)}}

\put(13,-6.5){\makebox(0,0)[cc]{\(\Phi_{23}\)}}

\multiput(14,-4)(0,-2){2}{\line(0,-1){2}}
\multiput(14,-6)(0,-2){2}{\circle{0.2}}
\put(14.1,-5.9){\makebox(0,0)[lc]{\(\langle 54\rangle\)}}
\put(14,-8){\vector(0,-1){2}}
\put(14,-10){\makebox(0,0)[ct]{\(\mathcal{T}(\langle 243,8\rangle)\)}}

\put(9,-11){\makebox(0,0)[cc]{\textbf{TKT:}}}
\put(10,-11){\makebox(0,0)[cc]{b.10}}
\put(10,-11.5){\makebox(0,0)[cc]{\((0043)\)}}
\put(12,-11){\makebox(0,0)[cc]{c.18}}
\put(12,-11.5){\makebox(0,0)[cc]{\((0313)\)}}
\put(14,-11){\makebox(0,0)[cc]{c.21}}
\put(14,-11.5){\makebox(0,0)[cc]{\((0231)\)}}
\put(8.5,-11.75){\framebox(6,1){}}

\end{picture}
\end{figure}

To get an adequate view of \S\S\
\ref{ss:SmpSec}--\ref{ss:Sec}
it is useful to visualize that part of coclass graph \(\mathcal{G}(3,2)\)
which consists of \(3\)-groups \(G\) of coclass \(\mathrm{cc}(G)=2\)
with abelianization \(G/\gamma_2(G)\) of type \((3,3)\)
and small order \(\lvert G\rvert=3^n\) in Figure
\ref{fig:Typ33CoCl2}.
The groups \(C_3\times C_3\) and \(G^3_0(0,0)\) form
the top of the mainline of \(\mathcal{G}(3,1)\) in Figure
\ref{fig:CoCl1}.
The edges of depth \(2\) neither belong to \(\mathcal{G}(3,1)\) nor to \(\mathcal{G}(3,2)\).
The top of \(\mathcal{G}(3,2)\) at the level of order \(3^5=243\)
consists of two isolated vertices \(\langle 5\rangle\), \(\langle 7\rangle\),
two roots \(\langle 9\rangle\), \(\langle 4\rangle\) of finite trees,
a root \(\langle 3\rangle\) of an infinite tree,
and two roots \(\langle 6\rangle\), \(\langle 8\rangle\) of coclass trees.
Only the mainlines of infinite trees are shown.
Vertices denoted by contour circles are metabelian
\cite[p. 189 ff.]{Ne1}.
Groups with defect \(k=0\) are represented by bigger circles than those with \(k=1\).
Vertices denoted by small contour squares are non-metabelian
\cite[Fig. 4.6--4.7, p. 74]{As}.
The symbol \(n\ast\) denotes a batch of \(n\) siblings below a common parent.
Numbers in angles denote the identifiers of groups in the SmallGroup library
\cite{BEO}
and in GAP 4.4
\cite{GAP},
where we omit the orders, which are given on the left hand scale.
The symbols \(\Phi_s\) denote isoclinism families
\cite{Hl,Ef,Jm}.
The principalization or transfer kernel types, briefly TKT,
\cite[Tbl. 6--7, p. 492--493]{Ma2}
in rectangles concern the vertices located vertically above.


\subsection{Groups \(G\) of coclass \(\mathrm{cc}(G)=2\) with bicyclic centre and \(m=4\), \(n=5\)}
\label{ss:SmpSec}

This section corresponds to the first row of Table
\ref{tab:LowUnk}.
Here we must investigate the abelianizations of all four maximal subgroups \(M_i\), \(1\le i\le 4\).
These \(7\) groups \(\langle 243,i\rangle\), \(3\le i\le 9\),
form the stem of Hall's isoclinism family \(\Phi_6\)
\cite[p. 139]{Hl},
\cite[4.1, p. 618, and 4.5 (6), pp. 620--621]{Jm},
\cite[pp. 182--183]{Bg}
and satisfy the following special relations, by
(\ref{eqn:LowRel}):

\begin{eqnarray*}
\label{e:SmpSecRel}
n&=&5=2m-3,\\
s=e&=&n-m+2=3=m-1,\\
\lbrack \chi_s(G),\gamma_e(G)\rbrack &=&\lbrack G,\gamma_3(G)\rbrack =\gamma_4(G)=1,\ k=0,\ \rho=0,\\
\sigma_4&=&1,\tau_4=1,\\
\sigma_3^3&=&1,\tau_3^3=1,\\
s_2^3&=&1,\\
\gamma_2(G)&=&\langle s_2\rangle\times\langle\sigma_3\rangle\times\langle\tau_3\rangle\text{ of type }(3,3,3),\\
\gamma_3(G)&=&\langle\sigma_3\rangle\times\langle\tau_3\rangle=\zeta_1(G)\text{ of type }(3,3).
\end{eqnarray*}

\noindent
The commutator subgroups of the maximal normal subgroups are given by
\cite[Cor. 3.2, p. 480]{Ma1}
and
(\ref{eqn:LowRel}):

\begin{eqnarray*}
\label{e:SmpSecGrp}
\gamma_2(M_1)&=&\langle t_3\rangle=\langle\sigma_3^{-\alpha}\tau_3^{1-\beta}\rangle,\\
\gamma_2(M_2)&=&\langle s_3\rangle=\langle\sigma_3^{\gamma-1}\tau_3^\delta\rangle,\\
\gamma_2(M_3)&=&\langle s_3t_3\rangle=\langle\sigma_3^{\gamma-\alpha-1}\tau_3^{\delta-\beta+1}\rangle,\\
\gamma_2(M_4)&=&\langle s_3t_3^{-1}\rangle=\langle\sigma_3^{\alpha+\gamma-1}\tau_3^{\beta+\delta-1}\rangle.
\end{eqnarray*}

\renewcommand{\arraystretch}{1.1}
\begin{table}[ht]
\caption{Parameters, third powers, and generators for \(m=4\), \(n=5\), \(e=3\), \(k=0\)}
\label{tab:SmpSecTyp}
\begin{center}
\begin{tabular}{|l|rrrr|cccc|cccc|}
\hline
 Type & \(\alpha\) & \(\beta\) & \(\gamma\) & \(\delta\) &      \(y^3\) &    \(x^3\) &                   \((xy)^3\) &         \((xy^{-1})^3\) &           \(t_3\) &                 \(s_3\) &                \(s_3t_3\) &              \(s_3t_3^{-1}\) \\
\hline
 D.10 &      \(0\) &     \(0\) &     \(-1\) &      \(1\) & \(\sigma_3\) & \(\tau_3\) &      \(\sigma_3^{-1}\tau_3\) & \(\sigma_3\tau_3^{-1}\) &        \(\tau_3\) & \(\sigma_3^{-2}\tau_3\) & \(\sigma_3^{-2}\tau_3^2\) &            \(\sigma_3^{-2}\) \\
 D.5  &      \(1\) &     \(1\) &     \(-1\) &      \(1\) & \(\sigma_3\) & \(\tau_3\) &                 \(\tau_3^2\) &          \(\sigma_3^2\) & \(\sigma_3^{-1}\) & \(\sigma_3^{-2}\tau_3\) &                \(\tau_3\) &      \(\sigma_3^{-1}\tau_3\) \\
 G.19 &      \(0\) &    \(-1\) &     \(-1\) &      \(0\) & \(\sigma_3\) & \(\tau_3\) & \(\sigma_3^{-1}\tau_3^{-1}\) & \(\sigma_3\tau_3^{-1}\) &      \(\tau_3^2\) &       \(\sigma_3^{-2}\) & \(\sigma_3^{-2}\tau_3^2\) & \(\sigma_3^{-2}\tau_3^{-2}\) \\
 H.4  &      \(1\) &     \(1\) &      \(1\) &      \(1\) & \(\sigma_3\) & \(\tau_3\) &       \(\sigma_3^2\tau_3^2\) &                   \(1\) & \(\sigma_3^{-1}\) &              \(\tau_3\) &   \(\sigma_3^{-1}\tau_3\) &           \(\sigma_3\tau_3\) \\
\hline
 b.10 &      \(0\) &     \(0\) &      \(0\) &      \(0\) & \(\sigma_3\) & \(\tau_3\) &                        \(1\) &                   \(1\) &        \(\tau_3\) &       \(\sigma_3^{-1}\) &   \(\sigma_3^{-1}\tau_3\) & \(\sigma_3^{-1}\tau_3^{-1}\) \\
 c.18 &      \(0\) &    \(-1\) &      \(0\) &      \(1\) & \(\sigma_3\) & \(\tau_3\) &                        \(1\) &         \(\tau_3^{-2}\) &      \(\tau_3^2\) & \(\sigma_3^{-1}\tau_3\) &         \(\sigma_3^{-1}\) & \(\sigma_3^{-1}\tau_3^{-1}\) \\
 c.21 &      \(0\) &     \(0\) &      \(0\) &      \(1\) & \(\sigma_3\) & \(\tau_3\) &                   \(\tau_3\) &         \(\tau_3^{-1}\) &        \(\tau_3\) & \(\sigma_3^{-1}\tau_3\) & \(\sigma_3^{-1}\tau_3^2\) &            \(\sigma_3^{-1}\) \\
\hline
\end{tabular}
\end{center}
\end{table}

\noindent
In Table
\ref{tab:SmpSecTyp}
we calculate the third powers \(g_i^3\)
of the generators \(g_1=y,g_2=x,g_3=xy,g_4=xy^{-1}\)
of the maximal normal subgroups \(M_i=\langle g_i,\gamma_2(G)\rangle\)
and the generators of the commutator subgroups \(\gamma_2(M_i)\) with \(1\le i\le 4\)
for each of the \(7\) isomorphism classes of groups \(G\) with \(m=4\), \(n=5\),
and parameters \(\alpha,\beta,\gamma,\delta\) given by
\cite[pp. 1--3]{Ne2}.
The principalization types
\cite[Satz 6.14, p. 208]{Ne1},
\cite[Tbl. 6--7, pp. 492--493]{Ma2}
of these isomorphism classes are all different.

Generally, according to
\cite[Lem. 3.4.11, p. 105]{Ne1},
the third powers of \(g_3\) and \(g_4\) are given by\\
\((xy)^3=\sigma_3^{\alpha+\gamma}\tau_3^{\beta+\delta}\)
and
\((xy^{-1})^3=\sigma_3^{\alpha-\gamma}\tau_3^{\beta-\delta}\).

The order of the coset of \(g_i\in M_i\) with respect to \(\gamma_2(M_i)\)
is bounded from above by \(3\), if and only if the third power \(g_i^3\)
is contained in \(\gamma_2(M_i)\).

\begin{theorem}
\label{t:SmpSecStr}
(Transfer target type \(\tau(G)\) of stem groups \(G\) in isoclinism family \(\Phi_6\))

Let \(K\) be a number field
with \(3\)-class group \(\mathrm{Cl}_3(K)\) of type \((3,3)\).
Suppose that the second \(3\)-class group
\(G=\mathrm{Gal}(\mathrm{F}_3^2(K)\vert K)\) of \(K\)
is of order \(\lvert G\rvert=3^n\)
and of class \(\mathrm{cl}(G)=m-1\), where \(m=4\) and \(n=5\), i. e.,
that \(G\) is one of the seven top vertices of coclass graph \(\mathcal{G}(3,2)\) in Figure \ref{fig:Typ33CoCl2},
with invariant \(e=3\) and bicyclic centre \(\zeta_1(G)\).\\
Then the structure of the \(3\)-class groups
of the first Hilbert \(3\)-class field \(\mathrm{F}_3^1(K)\) of \(K\) and
of the four unramified cyclic cubic extensions \(N_1,\ldots,N_4\) of \(K\)
is given by Table
\ref{tab:SmpSecStr},
in dependence on the principalization type \(\varkappa\) of \(K\).
The invariant \(\varepsilon=\varepsilon(K)\) denotes the
number of elementary abelian \(3\)-class groups \(\mathrm{Cl}_3(N_i)\) of type \((3,3,3)\),
for each principalization type.
\end{theorem}

\renewcommand{\arraystretch}{1.0}
\begin{table}[ht]
\caption{\(3\)-class groups of type \((3,3,3)\) for \(m=4\), \(n=5\), \(e=3\), \(k=0\)}
\label{tab:SmpSecStr}
\begin{center}
\begin{tabular}{|lc|c|cccc|c|}
\hline
 Type & \(\varkappa\) & \(\mathrm{Cl}_3(\mathrm{F}_3^1(K))\) & \(\mathrm{Cl}_3(N_1)\) & \(\mathrm{Cl}_3(N_2)\) & \(\mathrm{Cl}_3(N_3)\) & \(\mathrm{Cl}_3(N_4)\) & \(\varepsilon\) \\
\hline
 D.10 & \((2241)\)    &                          \((3,3,3)\) &              \((9,3)\) &              \((9,3)\) &            \((3,3,3)\) &              \((9,3)\) &           \(1\) \\
 D.5  & \((4224)\)    &                          \((3,3,3)\) &            \((3,3,3)\) &              \((9,3)\) &            \((3,3,3)\) &              \((9,3)\) &           \(2\) \\
 G.19 & \((2143)\)    &                          \((3,3,3)\) &              \((9,3)\) &              \((9,3)\) &              \((9,3)\) &              \((9,3)\) &           \(0\) \\
 H.4  & \((4443)\)    &                          \((3,3,3)\) &            \((3,3,3)\) &            \((3,3,3)\) &              \((9,3)\) &            \((3,3,3)\) &           \(3\) \\
\hline
 b.10 & \((0043)\)    &                          \((3,3,3)\) &              \((9,3)\) &              \((9,3)\) &            \((3,3,3)\) &            \((3,3,3)\) &           \(2\) \\
 c.18 & \((0313)\)    &                          \((3,3,3)\) &              \((9,3)\) &              \((9,3)\) &            \((3,3,3)\) &              \((9,3)\) &           \(1\) \\
 c.21 & \((0231)\)    &                          \((3,3,3)\) &              \((9,3)\) &              \((9,3)\) &              \((9,3)\) &              \((9,3)\) &           \(0\) \\
\hline
\end{tabular}
\end{center}
\end{table}

\begin{proof}
The structure of the \(3\)-class group \(\mathrm{Cl}_3(\mathrm{F}_3^1(K))\)
of the first Hilbert \(3\)-class field of \(K\) can be obtained
from the parameters \(m=4\) and \(e=3\) by means of the following two isomorphisms
from \(\mathrm{Cl}_3(\mathrm{F}_3^1(K))\)
to the commutator subgroup \(\gamma_2(G)\) of \(G=\mathrm{Gal}(\mathrm{F}_3^2(K)\vert K)\),
according to
\cite{Ar1}
and
\cite[Satz 4.2.4, p. 131]{Ne1}
\[\mathrm{Cl}_3(\mathrm{F}_3^1(K))
\simeq\mathrm{Gal}(\mathrm{F}_3^2(K)\vert\mathrm{F}_3^1(K))\simeq\gamma_2(G)
\simeq\mathrm{A}(3,m-2)\times\mathrm{A}(3,e-2)=\mathrm{A}(3,2)\times\mathrm{A}(3,1)\]
The structure of the \(3\)-class groups \(\mathrm{Cl}_3(N_i)\)
is a consequence of Table
\ref{tab:SmpSecTyp},
since we have the following isomorphism, according to
\cite{Ar1}
\[\mathrm{Cl}_3(N_i)\simeq\mathrm{Gal}(\mathrm{F}_3^1(N_i)\vert N_i)
\simeq\mathrm{Gal}(\mathrm{F}_3^2(K)\vert N_i)/\mathrm{Gal}(\mathrm{F}_3^2(K)\vert \mathrm{F}_3^1(N_i))
\simeq M_i/\gamma_2(M_i).\]
Taking into consideration the preliminaries in \S\S\
\ref{sss:LowFst}--\ref{sss:LowTrd},
resp. in Lemma
\ref{l:LowExo},
we use the equivalence of the following statements.

\begin{eqnarray*}
\label{e:SmpSecStr}
\mathrm{Cl}_3(N_1)\simeq\mathrm{A}(3,3)&\iff&g_1^3=y^3\not\in\langle t_3\rangle=\gamma_2(M_1),\\
\mathrm{Cl}_3(N_2)\simeq\mathrm{A}(3,3)&\iff&g_2^3=x^3\not\in\langle s_3\rangle=\gamma_2(M_2),\\
\mathrm{Cl}_3(N_3)\simeq\mathrm{A}(3,3)&\iff&g_3^3=(xy)^3\not\in\langle s_3t_3\rangle=\gamma_2(M_3),\\
\mathrm{Cl}_3(N_4)\simeq\mathrm{A}(3,3)&\iff&g_4^3=(xy^{-1})^3\not\in\langle s_3t_3^{-1}\rangle=\gamma_2(M_4).
\end{eqnarray*}
\end{proof}

\begin{corollary}
\label{c:SmpSecRul}
For each of these seven isomorphism classes of the second \(3\)-class group \(G\),
a \(3\)-class group \(\mathrm{Cl}_3(N_i)\), \(1\le i\le 4\), is of type \((3,3,3)\)
if and only if the norm class group \(\mathrm{Norm}_{N_i\vert K}(\mathrm{Cl}_3(N_i))\)
becomes principal either in none or in three of the extensions \(N_\ell\), \(1\le\ell\le 4\).\\
The extensions with \(3\)-class group of type \((3,3,3)\)
always satisfy the condition \((\mathrm{B})\) of Taussky
\cite{Ta},
i. e., they have a partial principalization without fixed point, as predicted by
\cite[Satz 7, p. 11]{HeSm}.
\end{corollary}

\begin{proof}
This is an immediate consequence of the principalization types \(\varkappa\)
of these seven isomorphism classes. Let
\(\mathrm{j}_{N_\ell\vert K}:\mathrm{Cl}_p(K)\longrightarrow\mathrm{Cl}_p(N_\ell)\), \(1\le\ell\le 4\),
denote the class extension homomorphisms
\cite[\S\ 2.3, p. 477]{Ma2}.
We demonstrate two cases.

For type H.4, \(\varkappa=(4443)\), we have\\
\(\mathrm{Norm}_{N_i\vert K}(\mathrm{Cl}_3(N_i))\cap\ker(\mathrm{j}_{N_\ell\vert K})=1\),
for \(i=1,2\) and any \(1\le\ell\le 4\),\\
\(\mathrm{Norm}_{N_3\vert K}(\mathrm{Cl}_3(N_3))=\ker(\mathrm{j}_{N_4\vert K})\), and
\(\mathrm{Norm}_{N_4\vert K}(\mathrm{Cl}_3(N_4))=\ker(\mathrm{j}_{N_\ell\vert K})\),
for \(1\le\ell\le 3\).

For type b.10, \(\varkappa=(0043)\), we have\\
\(\mathrm{Norm}_{N_i\vert K}(\mathrm{Cl}_3(N_i))<\ker(\mathrm{j}_{N_\ell\vert K})=\mathrm{Cl}_3(K)\),
for any \(1\le i\le 4\) and \(\ell=1,2\),\\
and additionally
\(\mathrm{Norm}_{N_3\vert K}(\mathrm{Cl}_3(N_3))=\ker(\mathrm{j}_{N_4\vert K})\) and
\(\mathrm{Norm}_{N_4\vert K}(\mathrm{Cl}_3(N_4))=\ker(\mathrm{j}_{N_3\vert K})\).
\end{proof}

\begin{corollary}
\label{c:SmpSecQdr}
If \(K\) is a quadratic base field with \(G\in\Phi_6\),
then the three total principalization types
\(\mathrm{b.10}\), \(\mathrm{c.21}\), \(\mathrm{c.18}\)
are impossible, due to class number relations,
and the remaining four partial principalization types
\(\mathrm{D.10}\), \(\mathrm{G.19}\), \(\mathrm{H.4}\), \(\mathrm{D.5}\)
are characterized uniquely by the invariant \(\varepsilon\).
\end{corollary}

\begin{proof}
For principalization types
\(\mathrm{b.10}\), \(\mathrm{c.21}\), \(\mathrm{c.18}\)
with singulet \(\varkappa(1)=0\)
\cite[\S\ 2.2--2.3, p. 475--478]{Ma2},
the entire \(3\)-class group \(\mathrm{Cl}_3(K)\) becomes principal in \(N_1\).
Hence, a quadratic base field \(K\) must be real,
and the unramified cyclic cubic extension \(N_1\)
must be an \(S_3\)-field of type \(\alpha\)
with even \(3\)-exponent of the \(3\)-class number \(\mathrm{h}_3(N_1)\),
in contradiction to \(\mathrm{Cl}_3(N_1)\simeq\mathrm{A}(3,3)\)
\cite[Prop. 4.3--4.4, p. 484--485]{Ma1}.
\end{proof}

\begin{example}
\label{ex:SmpSecQdr}
The first occurrences of second \(3\)-class groups
\(G=\mathrm{Gal}(\mathrm{F}_3^2(K)\vert K)\) of coclass \(\mathrm{cc}(G)=2\)
with invariants \(m=4\), \(n=5\), \(e=3\) and bicyclic centre
among quadratic fields \(K=\mathbb{Q}(\sqrt{D})\)
with discriminant \(-10^6<D<0\), resp. \(0<D<10^7\),
and \(3\)-class group of type \((3,3)\)
turned out to be the following.

\begin{itemize}
\item
The smallest value \(\lvert D\rvert\) of the discriminant
of a complex quadratic field \(K\)
with principalization type \(\mathrm{D.10}\), resp. \(\mathrm{D.5}\),
is \(4\,027\) \cite[pp. 22--25]{SoTa}, resp. \(12\,131\) \cite[Tbl. 3, p. 19]{HeSm}.
\item
The smallest discriminant \(D\) of a real quadratic field \(K\)
with principalization type \(\mathrm{D.10}\), resp. \(\mathrm{D.5}\),
is \(422\,573\), resp. \(631\,769\).
Both were unknown until \(2006\), resp. \(2009\).
\end{itemize}

\end{example}

\begin{conjecture}
\label{cnj:SmpSecQdr}
For quadratic base fields,
the principalization types \(\mathrm{G.19}\) and \(\mathrm{H.4}\)
cannot occur with invariants \(m=4\), \(n=5\), and \(k=0\)
of the corresponding second \(3\)-class groups
\(\langle 243,9\rangle\), \(\langle 243,4\rangle\),
since they have terminal metabelian descendants (Figure
\ref{fig:Typ33CoCl2})
of the same principalization type
with invariants \(m=5\), \(n=6\), and bigger defect \(k=1\).
We call this the weak or restricted \textit{leaf conjecture}.
(See Thm. 1.4 and Cnj. 3.1 in
\cite{Ma4}.)
\end{conjecture}


\subsection{Groups \(G\) of coclass \(\mathrm{cc}(G)=2\) with cyclic centre and \(m=5\), \(n=6\)}
\label{ss:SmlSec}

This section corresponds to the second row of Table
\ref{tab:LowUnk}.
Again, the abelianizations of all maximal subgroups \(M_i\), \(1\le i\le 4\),
have to be analyzed.
These groups form the stem of Easterfield's isoclinism families
\(\Phi_{40},\Phi_{41},\Phi_{42},\Phi_{43}\)
\cite{Ef},
\cite[4.1, p. 619, and 4.6 (40)--(43), p. 636]{Jm}
and satisfy the following special relations, by
(\ref{eqn:LowRel}):

\begin{eqnarray*}
\label{e:SmlSecRel1}
n&=&6=2m-4,\\
s&=&4=m-1,\\
e&=&n-m+2=3=m-2,\\
\lbrack \chi_s(G),\gamma_e(G)\rbrack &=&\lbrack G,\gamma_3(G)\rbrack =\gamma_4(G)=\gamma_{m-1}(G)>1,\ k=1,\ \rho=\pm 1,\\
\sigma_5&=&1,\tau_5=1,\\
\sigma_4^3&=&1,\tau_4^3=1,\\
\sigma_3^3&=&\sigma_4^{-3}=1,\tau_3^3=\tau_4^{-3}=1,\\
s_2^3&=&\sigma_4^{1-\rho(\beta-1)},\\
s_2^3&=&1\iff\beta=0,\rho=-1\text{ or }\beta=-1,\rho=1,\\
\tau_4&=&\sigma_4^{-\rho},\\
\end{eqnarray*}

\begin{eqnarray*}
\label{e:SmlSecRel2}
\gamma_2(G)&=&\langle s_2,\sigma_3,\sigma_4,\tau_3\rangle\text{ of type }
\begin{cases}
(9,3,3),&\text{ if }s_2^3\ne 1,\\
(3,3,3,3),&\text{ if }s_2^3=1,
\end{cases}\\
\gamma_3(G)&=&\langle\sigma_3,\sigma_4,\tau_3\rangle\text{ of type }(3,3,3),\\
\gamma_4(G)&=&\langle\sigma_4\rangle=\zeta_1(G)\text{ of type }(3).
\end{eqnarray*}

\noindent
The commutator subgroups of the maximal normal subgroups are given by
\cite[Cor. 3.2, p. 480]{Ma1}
and
(\ref{eqn:LowRel}):

\begin{eqnarray*}
\label{e:SmlSecGrp}
\gamma_2(M_1)&=&\langle t_3,\sigma_4\rangle=\langle\sigma_3^{-\rho\delta}\sigma_4^{\rho-\alpha}\tau_3^{1-\beta},\sigma_4\rangle
=\langle\sigma_3^{-\rho\delta}\tau_3^{1-\beta},\sigma_4\rangle,\\
\gamma_2(M_2)&=&\langle s_3,\sigma_4\rangle=\langle\sigma_3^{\rho\beta-1}\sigma_4^{\gamma-1}\tau_3^\delta,\sigma_4\rangle
=\langle\sigma_3^{\rho\beta-1}\tau_3^\delta,\sigma_4\rangle,\\
\gamma_2(M_3)&=&\langle s_3t_3,\sigma_4\rangle=\langle\sigma_3^{\rho(\beta-\delta)-1}\tau_3^{\delta-\beta+1},\sigma_4\rangle,\\
\gamma_2(M_4)&=&\langle s_3t_3^{-1},\sigma_4\rangle=\langle\sigma_3^{\rho(\beta+\delta)-1}\tau_3^{\beta+\delta-1},\sigma_4\rangle.
\end{eqnarray*}

The dependencies
on the parameters \(\alpha,\gamma\) disappear,
since they occur in the exponent of \(\sigma_4\in\gamma_4(G)\),
but each \(\gamma_2(M_i)\) with \(1\le i\le 4\) contains
\(\gamma_4(G)=\langle\sigma_4\rangle=\langle\tau_4\rangle\).

\noindent
In Table
\ref{tab:SmlSecTyp}
we calculate the third powers \(g_i^3\)
of the generators \(g_1=y,g_2=x,g_3=xy,g_4=xy^{-1}\)
of the maximal normal subgroups \(M_i=\langle g_i,\gamma_2(G)\rangle\)
and the generators of the commutator subgroups \(\gamma_2(M_i)\), \(1\le i\le 4\), modulo \(\sigma_4\)
for each of the \(12\) isomorphism classes of groups \(G\) with \(m=5\), \(n=6\), \(\rho=\pm 1\)
\cite[pp. 4--7]{Ne2}.
Several of these isomorphism classes have the same principalization type and the same parameters
\(\beta,\delta,\rho\),
as indicated by the second column of Table
\ref{tab:SmlSecTyp}.

Generally, according to
\cite[Lem. 3.4.11, p. 105]{Ne1},
the third powers of \(g_3\) and \(g_4\) modulo \(\sigma_4\) are given by\\
\((xy)^3=\sigma_3^{\rho(\beta+\delta)}\sigma_4^{\alpha+\gamma+\rho(\beta+\delta)}\tau_3^{\beta+\delta}
\equiv\sigma_3^{\rho(\beta+\delta)}\tau_3^{\beta+\delta}\)
and
\((xy^{-1})^3=\sigma_3^{\rho(\delta-\beta)}\sigma_4^{\alpha-\gamma+\rho\beta}\tau_3^{\beta-\delta}
\equiv\sigma_3^{\rho(\delta-\beta)}\tau_3^{\beta-\delta}\).

\begin{table}[ht]
\caption{Parameters, third powers, and generators for \(m=5\), \(n=6\), \(e=3\), \(k=1\)}
\label{tab:SmlSecTyp}
\begin{center}
\begin{tabular}{|l|c|rrr|cccc|cccc|}
\hline
 Type & isom. cl. & \(\beta\) & \(\delta\) & \(\rho\) &      \(y^3\) &    \(x^3\) &        \(\overline{(xy)^3}\) & \(\overline{(xy^{-1})^3}\) & \(\overline{t_3}\) & \(\overline{s_3}\) &     \(\overline{s_3t_3}\) &   \(\overline{s_3t_3^{-1}}\) \\
\hline
 G.19 &         2 &    \(-1\) &      \(0\) &    \(1\) & \(\sigma_3\) & \(\tau_3\) & \(\sigma_3^{-1}\tau_3^{-1}\) &    \(\sigma_3\tau_3^{-1}\) &       \(\tau_3^2\) &  \(\sigma_3^{-2}\) & \(\sigma_3^{-2}\tau_3^2\) & \(\sigma_3^{-2}\tau_3^{-2}\) \\
 H.4  &         4 &     \(1\) &      \(1\) &    \(1\) & \(\sigma_3\) & \(\tau_3\) &       \(\sigma_3^2\tau_3^2\) &                      \(1\) &  \(\sigma_3^{-1}\) &         \(\tau_3\) &   \(\sigma_3^{-1}\tau_3\) &           \(\sigma_3\tau_3\) \\
\hline
 b.10 &         3 &     \(0\) &      \(0\) &   \(-1\) & \(\sigma_3\) & \(\tau_3\) &                        \(1\) &                      \(1\) &         \(\tau_3\) &  \(\sigma_3^{-1}\) &   \(\sigma_3^{-1}\tau_3\) & \(\sigma_3^{-1}\tau_3^{-1}\) \\
 b.10 &         3 &     \(0\) &      \(0\) &    \(1\) & \(\sigma_3\) & \(\tau_3\) &                        \(1\) &                      \(1\) &         \(\tau_3\) &  \(\sigma_3^{-1}\) &   \(\sigma_3^{-1}\tau_3\) & \(\sigma_3^{-1}\tau_3^{-1}\) \\
\hline
\end{tabular}
\end{center}
\end{table}

The order of the coset of \(g_i\in M_i\) with respect to \(\gamma_2(M_i)\)
is bounded from above by \(3\),
if and only if the third power \(g_i^3\) is contained in \(\gamma_2(M_i)\).

\begin{theorem}
\label{t:SmlSecStr}
(TTT \(\tau(G)\) of stem groups \(G\) in isoclinism families \(\Phi_{40},\Phi_{41},\Phi_{42},\Phi_{43}\))

Let \(K\) be a number field
with \(3\)-class group \(\mathrm{Cl}_3(K)\) of type \((3,3)\).
Suppose that
the second \(3\)-class group \(G=\mathrm{Gal}(\mathrm{F}_3^2(K)\vert K)\)
is of order \(\lvert G\rvert=3^n\)
and class \(\mathrm{cl}(G)=m-1\), where \(n=6\) and \(m=5\),
such that \(\lbrack \chi_s(G),\gamma_e(G)\rbrack =\gamma_{m-1}(G)\), \(k=1\), i.e.,
that \(G\) is one of the twelve vertices with defect \(k=1\)
of coclass graph \(\mathcal{G}(3,2)\) in Figure
\ref{fig:Typ33CoCl2},
with invariant \(e=3\) and cyclic centre \(\zeta_1(G)\).\\
Then the structure of the \(3\)-class groups
of the first Hilbert \(3\)-class field \(\mathrm{F}_3^1(K)\) of \(K\) and
of the four unramified cyclic cubic extensions \(N_1,\ldots,N_4\) of \(K\)
is given by Table \ref{tab:SmlSecStr},
in dependence on the principalization type \(\varkappa\) of \(K\)
and on the relational parameters \(\beta,\rho\) of \(G\).
The invariant \(\varepsilon\) denotes
the number of \(3\)-class groups \(\mathrm{Cl}_3(N_i)\) of type \((3,3,3)\),
for each principalization type.
\end{theorem}

\begin{table}[ht]
\caption{\(3\)-class groups of type \((3,3,3)\) for \(m=5\), \(n=6\), \(e=3\), \(k=1\)}
\label{tab:SmlSecStr}
\begin{center}
\begin{tabular}{|lc|rr|c|cccc|c|}
\hline
 Type & \(\varkappa\) & \(\beta\) & \(\rho\) & \(\mathrm{Cl}_3(\mathrm{F}_3^1(K))\) & \(\mathrm{Cl}_3(N_1)\) & \(\mathrm{Cl}_3(N_2)\) & \(\mathrm{Cl}_3(N_3)\) & \(\mathrm{Cl}_3(N_4)\) & \(\varepsilon\) \\
\hline
 G.19 & \((2143)\)    &    \(-1\) &    \(1\) &                        \((3,3,3,3)\) &              \((9,3)\) &              \((9,3)\) &              \((9,3)\) &              \((9,3)\) &           \(0\) \\
 H.4  & \((4443)\)    &     \(1\) &    \(1\) &                          \((9,3,3)\) &            \((3,3,3)\) &            \((3,3,3)\) &              \((9,3)\) &            \((3,3,3)\) &           \(3\) \\
\hline
 b.10 & \((0043)\)    &     \(0\) &   \(-1\) &                        \((3,3,3,3)\) &              \((9,3)\) &              \((9,3)\) &            \((3,3,3)\) &            \((3,3,3)\) &           \(2\) \\
 b.10 & \((0043)\)    &     \(0\) &    \(1\) &                          \((9,3,3)\) &              \((9,3)\) &              \((9,3)\) &            \((3,3,3)\) &            \((3,3,3)\) &           \(2\) \\
\hline
\end{tabular}
\end{center}
\end{table}

\begin{proof}
Similarly as in the proof of Theorem
\ref{t:SmpSecStr},
the structure of the \(3\)-class group \(\mathrm{Cl}_3(\mathrm{F}_3^1(K))\)
of the first Hilbert \(3\)-class field of \(K\)
is a consequence of \(m=5\), \(e=3\), and the isomorphisms in
\cite{Ar1}
and
\cite[Satz 4.2.4, p. 131]{Ne1}
\[\mathrm{Cl}_3(\mathrm{F}_3^1(K))\simeq\gamma_2(G)
\simeq
\begin{cases}
\mathrm{A}(3,m-2)\times\mathrm{A}(3,e-2)=\mathrm{A}(3,3)\times\mathrm{A}(3,1)
\text{ for }\rho\not\equiv\beta-1\pmod{3},\\
\mathrm{A}(3,m-3)\times\mathrm{A}(3,e-1)=\mathrm{A}(3,2)\times\mathrm{A}(3,2)
\text{ for }\rho\equiv\beta-1\pmod{3}.
\end{cases}\]
The structure of the \(3\)-class groups \(\mathrm{Cl}_3(N_i)\simeq M_i/\gamma_2(M_i)\)
follows from Table
\ref{tab:SmlSecTyp},
if we take into consideration the preparations in \S\S\
\ref{sss:LowFst}--\ref{sss:LowTrd},
resp. in Lemma
\ref{l:LowExo}.
\end{proof}

\begin{corollary}
\label{c:SmlSecRul}
For each of these twelve isomorphism classes of the second \(3\)-class group \(G\),
a \(3\)-class group \(\mathrm{Cl}_3(N_i)\), \(1\le i\le 4\),
is  of type \((3,3,3)\) if and only if
the norm class group \(\mathrm{Norm}_{N_i\vert K}(\mathrm{Cl}_3(N_i))\)
becomes principal either in none or in three of the extensions
\(N_\ell\), \(1\le\ell\le 4\).\\
The extensions with \(3\)-class group of type \((3,3,3)\)
always satisfy the condition \((\mathrm{B})\) of Taussky
\cite{Ta},
i. e., they have a partial principalization without fixed point.
\end{corollary}

\begin{proof}
This is an immediate consequence of the principalization types \(\varkappa\)
of these twelve isomorphism classes.
\end{proof}

\begin{corollary}
\label{c:SmlSecQdr}
If \(K\) is a quadratic base field with \(G\in\Phi_s\),
for some \(s\in\lbrace 40,41,42,43\rbrace\),
then the total principalization type
\(\mathrm{b.10}\) is impossible,
due to class number relations,
and the remaining two partial principalization types
\(\mathrm{G.19}\) and \(\mathrm{H.4}\)
are characterised uniquely by the invariant \(\varepsilon\).\\
Furthermore, only the second \(3\)-class groups
\(G=\langle 729,57\rangle\) and \(G=\langle 729,45\rangle\) are possible.
\end{corollary}

\begin{proof}
For the principalization type \(\mathrm{b.10}\)
with singulets \(\varkappa(1)=\varkappa(2)=0\),
the entire \(3\)-class group \(\mathrm{Cl}_3(K)\) becomes principal in \(N_1,N_2\).
In the case of a quadratic base field \(K\) this yields a similar contradiction
to \(\mathrm{Cl}_3(N_1)\simeq\mathrm{Cl}_3(N_2)\simeq\mathrm{A}(3,3)\)
as in the proof of Corollary
\ref{c:SmpSecQdr}.
For the last assertion,
we refer to Thm. 3.14 in
\cite{Ma4}.
\end{proof}

\begin{example}
\label{ex:SmlSecQdr}
The first occurrences of second \(3\)-class groups
\(G=\mathrm{Gal}(\mathrm{F}_3^2(K)\vert K)\) of coclass \(\mathrm{cc}(G)=2\)
with invariants \(m=5\), \(n=6\), \(e=3\), \(\rho=\pm 1\) and cyclic centre
among quadratic fields \(K=\mathbb{Q}(\sqrt{D})\)
with discriminant \(-10^6<D<0\), resp. \(0<D<10^7\),
and \(3\)-class group of type \((3,3)\)
turned out to be the following.

\begin{itemize}
\item
The smallest value \(\lvert D\rvert\) of the discriminant
of a complex quadratic field \(K\)
with principalization type \(\mathrm{H.4}\), resp. \(\mathrm{G.19}\),
is \(3\,896\), resp. \(12\,067\)
\cite[Tbl. 3, p. 19]{HeSm}.
\item
The smallest discriminant \(D\) of a real quadratic field \(K\)
with principalization type \(\mathrm{G.19}\), resp. \(\mathrm{H.4}\),
is \(214\,712\), resp. \(957\,013\).
Both were unknown until \(2006\), resp. \(2009\).
\end{itemize}

\end{example}


\subsection{All other groups \(G\) of coclass \(\mathrm{cc}(G)=2\) with \(m\ge 5\), \(n\ge 6\)}
\label{ss:Sec}
This section corresponds to the third and fourth row of Table
\ref{tab:LowUnk}.
Here we have to analyze the abelianizations of three maximal subgroups \(M_i\), \(2\le i\le 4\).
These groups satisfy the following general relations.
For \(m\ge 6\), we have \(\sigma_{m-2},\sigma_{m-1}\in\gamma_4(G)\).
For \(m=5\), the case \(\rho=\pm 1\) has been investigated
in the preceding section already,
and we only have to consider the remaining possibility \(\rho=0\).
Since \(\sigma_{m-1}\in\gamma_4(G)\), also for \(m=5\),
the following congruences modulo \(\gamma_4(G)\) are valid, generally.
The apparent dependencies on parameters \(\alpha,\gamma,\rho\) vanish.
Here we use
\cite[Lem. 3.4.11, p. 105]{Ne1}
and
(\ref{eqn:LowRel}).

\begin{eqnarray*}
\label{e:SecGnrRel}
(xy)^3&=&\sigma_{m-2}^{\rho(\beta+\delta)}\sigma_{m-1}^{\alpha+\gamma+\rho(\beta+\delta)}\tau_3^{\beta+\delta}
\equiv\tau_3^{\beta+\delta}\pmod{\gamma_4(G)},\\
(xy^{-1})^3&=&\sigma_{m-2}^{\rho(\delta-\beta)}\sigma_{m-1}^{\alpha-\gamma+\rho\beta}\tau_3^{\beta-\delta}
\equiv\tau_3^{\beta-\delta}\pmod{\gamma_4(G)},\\
t_3&=&\tau_3\tau_4\tau_3^{-\beta}\sigma_{m-2}^{-\rho\delta}\sigma_{m-1}^{-\alpha}
\equiv\tau_3^{1-\beta}\pmod{\gamma_4(G)},\\
s_3&=&\sigma_3^{-1}\sigma_4^{-1}\sigma_{m-2}^{\rho\beta}\sigma_{m-1}^\gamma\tau_3^\delta
\equiv\sigma_3^{-1}\tau_3^\delta\pmod{\gamma_4(G)},\\
s_3t_3&\equiv&\sigma_3^{-1}\tau_3^{\delta-\beta+1}\pmod{\gamma_4(G)},\\
s_3t_3^{-1}&\equiv&\sigma_3^{-1}\tau_3^{\beta+\delta-1}\pmod{\gamma_4(G)}.
\end{eqnarray*}

\noindent
In Table
\ref{tab:SecTyp}
we calculate the third powers \(g_i^3\)
of the generators
\(g_1=y,g_2=x,g_3=xy,g_4=xy^{-1}\)
of the maximal normal subgroups \(M_i=\langle g_i,\gamma_2(G)\rangle\)
and the generators of the commutator subgroups
\(\gamma_2(M_i)\), \(1\le i\le 4\), modulo \(\gamma_4(G)\)
for each of the isomorphism classes of groups \(G\) with \(m\ge 5\), \(n\ge 6\), \(e=3\).
Several of these classes have the same principalization type
and the same parameters \(\beta,\delta\).
The left number of isomorphism classes concerns the single case \(m=5\)
\cite[pp. 8--9]{Ne2},
the right number odd values of \(m\ge 7\)
\cite[pp. 8--9 and pp. 16--19]{Ne2}
and the middle number even values of \(m\ge 6\)
\cite[pp. 10--12 and pp. 13--15]{Ne2}.

\begin{table}[ht]
\caption{Parameters, third powers, and generators for \(m\ge 5\), \(n\ge 6\), \(e=3\)}
\label{tab:SecTyp}
\begin{center}
\begin{tabular}{|l|ccc|rr|cccc|cccc|}
\hline
 Type & \multicolumn{3}{|c|}{isom. cl.} & \(\beta\) & \(\delta\) &      \(y^3\) &    \(x^3\) & \(\overline{(xy)^3}\) & \(\overline{(xy^{-1})^3}\) & \(\overline{t_3}\) &      \(\overline{s_3}\) &     \(\overline{s_3t_3}\) &   \(\overline{s_3t_3^{-1}}\) \\
\hline
 b.10 &          1 & 7 & 9              &     \(0\) &      \(0\) & \(\sigma_3\) & \(\tau_3\) &                 \(1\) &                      \(1\) &         \(\tau_3\) &       \(\sigma_3^{-1}\) &   \(\sigma_3^{-1}\tau_3\) & \(\sigma_3^{-1}\tau_3^{-1}\) \\
 d.19 &          1 & 2 & 1              &     \(0\) &      \(0\) & \(\sigma_3\) & \(\tau_3\) &                 \(1\) &                      \(1\) &         \(\tau_3\) &       \(\sigma_3^{-1}\) &   \(\sigma_3^{-1}\tau_3\) & \(\sigma_3^{-1}\tau_3^{-1}\) \\
 d.23 &          1 & 1 & 1              &     \(0\) &      \(0\) & \(\sigma_3\) & \(\tau_3\) &                 \(1\) &                      \(1\) &         \(\tau_3\) &       \(\sigma_3^{-1}\) &   \(\sigma_3^{-1}\tau_3\) & \(\sigma_3^{-1}\tau_3^{-1}\) \\
 d.25 &          1 & 2 & 1              &     \(0\) &      \(0\) & \(\sigma_3\) & \(\tau_3\) &                 \(1\) &                      \(1\) &         \(\tau_3\) &       \(\sigma_3^{-1}\) &   \(\sigma_3^{-1}\tau_3\) & \(\sigma_3^{-1}\tau_3^{-1}\) \\
\hline                                  
 c.18 &          1 & 1 & 1              &    \(-1\) &      \(1\) & \(\sigma_3\) & \(\tau_3\) &                 \(1\) &            \(\tau_3^{-2}\) &       \(\tau_3^2\) & \(\sigma_3^{-1}\tau_3\) &         \(\sigma_3^{-1}\) & \(\sigma_3^{-1}\tau_3^{-1}\) \\
 E.6  &          1 & 1 & 1              &    \(-1\) &      \(1\) & \(\sigma_3\) & \(\tau_3\) &                 \(1\) &            \(\tau_3^{-2}\) &       \(\tau_3^2\) & \(\sigma_3^{-1}\tau_3\) &         \(\sigma_3^{-1}\) & \(\sigma_3^{-1}\tau_3^{-1}\) \\
 E.14 &          1 & 2 & 1              &    \(-1\) &      \(1\) & \(\sigma_3\) & \(\tau_3\) &                 \(1\) &            \(\tau_3^{-2}\) &       \(\tau_3^2\) & \(\sigma_3^{-1}\tau_3\) &         \(\sigma_3^{-1}\) & \(\sigma_3^{-1}\tau_3^{-1}\) \\
 H.4  &          1 & 8 & 9              &    \(-1\) &      \(1\) & \(\sigma_3\) & \(\tau_3\) &                 \(1\) &            \(\tau_3^{-2}\) &       \(\tau_3^2\) & \(\sigma_3^{-1}\tau_3\) &         \(\sigma_3^{-1}\) & \(\sigma_3^{-1}\tau_3^{-1}\) \\
\hline                                  
 c.21 &          1 & 1 & 1              &     \(0\) &      \(1\) & \(\sigma_3\) & \(\tau_3\) &            \(\tau_3\) &            \(\tau_3^{-1}\) &         \(\tau_3\) & \(\sigma_3^{-1}\tau_3\) & \(\sigma_3^{-1}\tau_3^2\) &            \(\sigma_3^{-1}\) \\
 E.8  &          1 & 1 & 1              &     \(0\) &      \(1\) & \(\sigma_3\) & \(\tau_3\) &            \(\tau_3\) &            \(\tau_3^{-1}\) &         \(\tau_3\) & \(\sigma_3^{-1}\tau_3\) & \(\sigma_3^{-1}\tau_3^2\) &            \(\sigma_3^{-1}\) \\
 E.9  &          1 & 2 & 1              &     \(0\) &      \(1\) & \(\sigma_3\) & \(\tau_3\) &            \(\tau_3\) &            \(\tau_3^{-1}\) &         \(\tau_3\) & \(\sigma_3^{-1}\tau_3\) & \(\sigma_3^{-1}\tau_3^2\) &            \(\sigma_3^{-1}\) \\
 G.16 &          1 & 8 & 9              &     \(0\) &      \(1\) & \(\sigma_3\) & \(\tau_3\) &            \(\tau_3\) &            \(\tau_3^{-1}\) &         \(\tau_3\) & \(\sigma_3^{-1}\tau_3\) & \(\sigma_3^{-1}\tau_3^2\) &            \(\sigma_3^{-1}\) \\
\hline
\end{tabular}
\end{center}
\end{table}

The order of the coset of \(g_i\in M_i\) with respect to \(\gamma_2(M_i)\)
is bounded from above by \(3\) if and only if
the third power \(g_i^3\) is contained in \(\gamma_2(M_i)\).

\begin{theorem}
\label{t:SecStr}
(TTT \(\tau(G)\) of groups \(G\) on coclass trees of \(\mathcal{G}(3,2)\))\\
Let \(K\) be a number field with \(3\)-class group
\(\mathrm{Cl}_3(K)\) of type \((3,3)\).
Suppose that the second \(3\)-class group \(G=\mathrm{Gal}(\mathrm{F}_3^2(K)\vert K)\)
of \(K\) is of order \(\lvert G\rvert=3^n\)
and of class \(\mathrm{cl}(G)=m-1\), where \(n\ge 6\) and \(m=n-1\),
i. e., that \(G\) is a vertex on one of the three coclass trees of coclass graph \(\mathcal{G}(3,2)\)
in Figure
\ref{fig:Typ33CoCl2},
with invariant \(e=3\).
In the case \(m=5\), \(n=6\) let \(\lbrack \chi_s(G),\gamma_e(G)\rbrack =1\).\\
Then the structure of the \(3\)-class groups \(\mathrm{Cl}_3(N_i)\)
of the four unramified cyclic cubic extensions \(N_i\) of \(K\)
is given by Table
\ref{tab:SecStr},
in dependence on the principalization type \(\varkappa\) of \(K\).
The invariant \(\varepsilon\) denotes the number of \(3\)-class groups
\(\mathrm{Cl}_3(N_i)\) of type \((3,3,3)\), for each principalization type.
Generally, the first two \(3\)-class groups
\(\mathrm{Cl}_3(N_1)\) and \(\mathrm{Cl}_3(N_2)\)
are nearly homocyclic.
\end{theorem}

\begin{table}[ht]
\caption{\(3\)-class groups of type \((3,3,3)\) for \(m\ge 5\), \(n\ge 6\), \(e=3\)}
\label{tab:SecStr}
\begin{center}
\begin{tabular}{|lc|cccc|c|}
\hline
 Type & \(\varkappa\) &         \(\mathrm{Cl}_3(N_1)\) & \(\mathrm{Cl}_3(N_2)\) & \(\mathrm{Cl}_3(N_3)\) & \(\mathrm{Cl}_3(N_4)\) & \(\varepsilon\) \\
\hline
 b.10 & \((0043)\)    &   A\((3,m-1)\) or A\((3,m-2)\) &              \((9,3)\) &            \((3,3,3)\) &            \((3,3,3)\) &           \(2\) \\
 d.19 & \((4043)\)    &                   A\((3,m-1)\) &              \((9,3)\) &            \((3,3,3)\) &            \((3,3,3)\) &           \(2\) \\
 d.23 & \((1043)\)    &                   A\((3,m-1)\) &              \((9,3)\) &            \((3,3,3)\) &            \((3,3,3)\) &           \(2\) \\
 d.25 & \((2043)\)    &                   A\((3,m-1)\) &              \((9,3)\) &            \((3,3,3)\) &            \((3,3,3)\) &           \(2\) \\
\hline
 c.18 & \((0313)\)    &                   A\((3,m-1)\) &              \((9,3)\) &            \((3,3,3)\) &              \((9,3)\) &           \(1\) \\
 E.6  & \((1313)\)    &                   A\((3,m-1)\) &              \((9,3)\) &            \((3,3,3)\) &              \((9,3)\) &           \(1\) \\
 E.14 & \((2313)\)    &                   A\((3,m-1)\) &              \((9,3)\) &            \((3,3,3)\) &              \((9,3)\) &           \(1\) \\
 H.4  & \((3313)\)    &   A\((3,m-1)\) or A\((3,m-2)\) &              \((9,3)\) &            \((3,3,3)\) &              \((9,3)\) &           \(1\) \\
\hline
 c.21 & \((0231)\)    &                   A\((3,m-1)\) &              \((9,3)\) &              \((9,3)\) &              \((9,3)\) &           \(0\) \\
 E.8  & \((1231)\)    &                   A\((3,m-1)\) &              \((9,3)\) &              \((9,3)\) &              \((9,3)\) &           \(0\) \\
 E.9  & \((2231)\)    &                   A\((3,m-1)\) &              \((9,3)\) &              \((9,3)\) &              \((9,3)\) &           \(0\) \\
 G.16 & \((4231)\)    &   A\((3,m-1)\) or A\((3,m-2)\) &              \((9,3)\) &              \((9,3)\) &              \((9,3)\) &           \(0\) \\
\hline
\end{tabular}
\end{center}
\end{table}

\begin{proof}
The structure of the first \(3\)-class group \(\mathrm{Cl}_3(N_1)\)
is given here only for the sake of completeness
and is contained in the statement of Theorem
\ref{t:LowHom}
already.\\
Similarly as in the proof of Theorem
\ref{t:SmpSecStr},
the structure of the other \(3\)-class groups
\(\mathrm{Cl}_3(N_i)\simeq M_i/\gamma_2(M_i)\) with \(2\le i\le 4\)
is a consequence of table
\ref{tab:SecTyp},
when the preparations in the \S\S\
\ref{sss:LowSnd}--\ref{sss:LowTrd}, resp. in Lemma
\ref{l:LowExo},
are taken into consideration.
\end{proof}

\begin{corollary}
\label{c:SecTree}
(\(\varepsilon\) as a tree invariant)\\
All metabelian groups \(G\) on the coclass tree
\(\mathcal{T}(\langle 729,40\rangle)\),
resp. \(\mathcal{T}(\langle 243,6\rangle)\),
resp. \(\mathcal{T}(\langle 243,8\rangle)\),
of coclass graph \(\mathcal{G}(3,2)\) in Figure
\ref{fig:Typ33CoCl2}
are characterized by the value
\(\varepsilon=2\),
resp. \(\varepsilon=1\),
resp. \(\varepsilon=0\).
\end{corollary}

\begin{proof}
This is a consequence of Theorem
\ref{t:SecStr}
and the diagram
\cite[p. 189 ff.]{Ne1}.
See also \cite[Thm. 3.16--3.17]{Ma4}.
\end{proof}

\begin{corollary}
\label{c:SecRul}
As before, for these isomorphism classes of the second \(3\)-class group \(G\),
the extensions with \(3\)-class group of type \((3,3,3)\)
satisfy the condition \((\mathrm{B})\) of Taussky
\cite{Ta},
i. e., they have a partial principalization without fixed point.
However, here only the following weaker statement without admissible inversion is true:
if a \(3\)-class group \(\mathrm{Cl}_3(N_i)\), \(3\le i\le 4\),
is of type \((3,3,3)\), then
the norm class group \(\mathrm{Norm}_{N_i\vert K}(\mathrm{Cl}_3(N_i))\)
becomes principal in either two or three of the extensions
\(N_\ell\), \(1\le\ell\le 4\).
\end{corollary}

\begin{proof}
This follows by evaluating the principalization type \(\varkappa\)
of the isomorphism classes.
\end{proof}

\begin{corollary}
\label{c:SecQdr}
If \(K\) is a quadratic base field with \(G\) on a coclass tree of \(\mathcal{G}(3,2)\),
then the four total principalization types
\(\mathrm{b.10}\), \(\mathrm{d.19}\), \(\mathrm{d.23}\), \(\mathrm{d.25}\)
with \(\varepsilon=2\), i. e., \(G\in\mathcal{T}(\langle 729,40\rangle)\),
are impossible, due to class number relations.
The remaining eight principalization types cannot be determined
uniquely by the invariant \(\varepsilon\) alone .

\begin{enumerate}
\item
The total principalization types
\(\mathrm{c.18}\) with \(\varepsilon=1\) and \(\mathrm{c.21}\) with \(\varepsilon=0\)
are characterized by an even \(3\)-exponent \(u\) of the first \(3\)-class number
\(\mathrm{h}_3(N_1)=3^u\).
\item
The partial principalization types
\(\mathrm{E.6}\), \(\mathrm{E.14}\) with \(\varepsilon=1\) and 
\(\mathrm{E.8}\), \(\mathrm{E.9}\) with \(\varepsilon=0\)
are determined by an odd \(3\)-exponent \(w\)
of the \(3\)-class number
\(\mathrm{h}_3(\mathrm{F}_3^1(K))=3^w\).
\item
The partial principalization types
\(\mathrm{H.4}\) with \(\varepsilon=1\) and
\(\mathrm{G.16}\) with \(\varepsilon=0\)
are characterized by an even \(3\)-exponent \(w\)
of the \(3\)-class number
\(\mathrm{h}_3(\mathrm{F}_3^1(K))=3^w\),
provided the weak leaf conjecture
\ref{cnj:SmpSecQdr}
holds.
\end{enumerate}
Here, as before, \(\mathrm{F}_3^1(K)\) denotes the first Hilbert \(3\)-class field of \(K\).
\end{corollary}

\begin{proof}
For the principalization types
\(\mathrm{b.10}\), \(\mathrm{d.19}\), \(\mathrm{d.23}\), and \(\mathrm{d.25}\)
having \(\varkappa(2)=0\),
the entire \(3\)-class group \(\mathrm{Cl}_3(K)\) becomes principal in \(N_2\).
Therefore a quadratic base field \(K\) must be real
and the unramified cyclic cubic extension \(N_2\)
must be an \(S_3\)-field of type \(\alpha\)
with even \(3\)-exponent of the \(3\)-class number \(\mathrm{h}_3(N_2)\),
in contradiction to \(\mathrm{Cl}_3(N_2)\simeq\mathrm{A}(3,3)\).
See
\cite[Prop. 4.3--4.4, p. 484--485]{Ma1}.

The parity of the \(3\)-exponent of the first \(3\)-class number
\(\mathrm{h}_3(N_1)=3^u\) turns out
to be even, \(u=m-1\equiv 0\pmod{2}\),
for the total principalization types \(\mathrm{c.18}\) and \(\mathrm{c.21}\)
with \(\varkappa(1)=0\),
to be odd, \(u=m-1\equiv 1\pmod{2}\),
for the partial principalization types
\(\mathrm{E.6}\), \(\mathrm{E.14}\) and \(\mathrm{E.8}\), \(\mathrm{E.9}\)
with \(k=0\) and thus \(w=n-2\) odd,
and to be odd, \(u=m-2\equiv 1\pmod{2}\),
for the partial principalization types \(\mathrm{H.4}\) and \(\mathrm{G.16}\)
with conjectural \(k=1\) (see Conjecture
\ref{cnj:SmpSecQdr})
and thus \(w=n-2\) even.

Here, we use the relation \(3^w=\mathrm{h}_3(\mathrm{F}_3^1(K))=\lvert\gamma_2(G)\rvert=3^{n-2}\)
and
\cite[Thm. 5.2--5.3, p. 492--493]{Ma1}.
\end{proof}

\begin{example}
\label{ex:SecQdr}
The first occurrences of second \(3\)-class groups \(G=\mathrm{Gal}(\mathrm{F}_3^2(K)\vert K)\)
on the coclass trees \(\mathcal{T}(\langle 243,6\rangle)\) and \(\mathcal{T}(\langle 243,8\rangle)\)
with invariants \(m\ge 5\), \(n\ge 6\), \(e=3\) and
bicyclic center, \(k=0\), resp. cyclic center, \(k=1\),
over quadratic base fields \(K=\mathbb{Q}(\sqrt{D})\)
with \(3\)-class group of type \((3,3)\)
and discriminant \(-10^6<D<0\), resp. \(0<D<10^7\),
are summarised in Table
\ref{tab:SecExp}.
Here, \(\lvert D\rvert\) denotes the smallest absolute value
of the discriminant of a complex quadratic field \(K\),
and \(D\) the smallest discriminant of a real quadratic field \(K\),
of the corresponding principalization type.
The earlier computations by
Scholz and Taussky
\cite{SoTa},
Heider and Schmithals
\cite{HeSm},
and Brink
\cite{Br},
are confirmed.
Cases without references,
in particular all cases with real quadratic base fields,
were unknown up to now.

\begin{table}[ht]
\caption{Examples for groups \(G\) on coclass trees with \(m\ge 5\), \(n\ge 6\), \(e=3\)}
\label{tab:SecExp}
\begin{center}
\begin{tabular}{|lc|rrr|rc|r|}
\hline
 Type           & \(\varkappa\) & \(m\) &  \(n\) & \(k\) & \(\lvert D\rvert\) & ref.        &           \(D\) \\
\hline
 c.18              & \((0313)\) & \(5\) &  \(6\) & \(0\) &         impossible &             &    \(534\,824\) \\
 E.6               & \((1313)\) & \(6\) &  \(7\) & \(0\) &        \(15\,544\) & \cite{HeSm} & \(5\,264\,069\) \\
 E.6\(\uparrow\)   & \((1313)\) & \(8\) &  \(9\) & \(0\) &       \(268\,040\) &             &       unknown   \\
 E.14              & \((2313)\) & \(6\) &  \(7\) & \(0\) &        \(16\,627\) & \cite{HeSm} & \(3\,918\,837\) \\
 E.14\(\uparrow\)  & \((2313)\) & \(8\) &  \(9\) & \(0\) &       \(262\,744\) &             &       unknown   \\
 H.4\(\uparrow\)   & \((3313)\) & \(7\) &  \(8\) & \(1\) &        \(21\,668\) & \cite{Br,Ma}& \(1\,162\,949\) \\
 H.4\(\uparrow^2\) & \((3313)\) & \(9\) & \(10\) & \(1\) &       \(446\,788\) &             &       unknown   \\
\hline
 c.21              & \((0231)\) & \(5\) &  \(6\) & \(0\) &         impossible &             &    \(540\,365\) \\
 c.21\(\uparrow\)  & \((0231)\) & \(7\) &  \(8\) & \(0\) &         impossible &             & \(1\,001\,957\) \\
 E.8               & \((1231)\) & \(6\) &  \(7\) & \(0\) &        \(34\,867\) &             & \(6\,098\,360\) \\
 E.8\(\uparrow\)   & \((1231)\) & \(8\) &  \(9\) & \(0\) &       \(370\,740\) &             &       unknown   \\
 E.9               & \((2231)\) & \(6\) &  \(7\) & \(0\) &         \(9\,748\) & \cite{SoTa} &    \(342\,664\) \\
 E.9\(\uparrow\)   & \((2231)\) & \(8\) &  \(9\) & \(0\) &       \(297\,079\) &             &       unknown   \\
 G.16              & \((4231)\) & \(7\) &  \(8\) & \(1\) &        \(17\,131\) & \cite{HeSm} & \(8\,711\,453\) \\
 G.16\(\uparrow\)  & \((4231)\) & \(9\) & \(10\) & \(1\) &       \(819\,743\) &             &       unknown   \\
\hline
\end{tabular}
\end{center}
\end{table}

\noindent
Whereas the parameters \(m=4\), \(n=5\)
for the principalization types \(\mathrm{D.10}\), \(\mathrm{D.5}\)
in example
\ref{ex:SmpSecQdr},
and the parameters \(m=5\), \(n=6\)
for the \textit{ground state} of principalization types \(\mathrm{G.19}\), \(\mathrm{H.4}\)
in example
\ref{ex:SmlSecQdr},
are determined uniquely,
we now have infinite coclass families of second \(3\)-class groups
\(G=\mathrm{Gal}(\mathrm{F}_3^2(K)\vert K)\)
with strictly increasing nilpotency class \(\mathrm{cl}(G)=m-1\),
sharing the same principalization type.
For this reason, we define \textit{excited states} of principalization types,
denoted by arrows \(\uparrow\),\(\uparrow^2\), and so on.
The index \(m\) of nilpotency can take
all odd values \(m\ge 5\)
for the principalization types \(\mathrm{c.18}\), \(\mathrm{c.21}\),
all even values \(m\ge 6\)
for the principalization types \(\mathrm{E.6}\), \(\mathrm{E.14}\), \(\mathrm{E.8}\), \(\mathrm{E.9}\), 
and all odd values \(m\ge 7\)
for principalization types \(\mathrm{H.4}\), \(\mathrm{G.16}\).

A special feature of the principalization types \(\mathrm{c.18}\) and \(\mathrm{c.21}\) 
is the location of their second \(3\)-class groups
\(G=\mathrm{Gal}(\mathrm{F}_3^2(K)\vert K)\)
as infinitely capable vertices on mainlines of coclass trees
\(\mathcal{T}(\langle 243,6\rangle)\) and \(\mathcal{T}(\langle 243,8\rangle)\).

Concrete numerical realisations are known
for the ground state with minimal index of nilpotency \(m\),
for each of these infinite coclass families.
Realisations for excited states with higher values of \(m\)
are known for the principalization types 
\(\mathrm{E.6}\), \(\mathrm{E.14}\), \(\mathrm{E.8}\), \(\mathrm{E.9}\),
\(\mathrm{H.4}\), and \(\mathrm{G.16}\)
of complex quadratic fields and
for the principalization type \(\mathrm{c.21}\)
of real quadratic fields.
\end{example}


\subsection{Groups \(G\) of coclass \(\mathrm{cc}(G)\ge 3\) with \(m\ge 5\), \(n\ge 7\)}
\label{ss:Low}
This section corresponds to the fifth and sixth row of Table
\ref{tab:LowUnk}.
Here we must investigate the abelianizations of only two maximal subgroups \(M_i\), \(3\le i\le 4\).
The third powers of the generators of these groups satisfy the following general relations,
according to
\cite[Lem. 3.4.11, p. 105]{Ne1}.

\begin{eqnarray*}
\label{e:LowGnrRel}
(xy)^3&=&\sigma_{m-2}^{\rho(\beta+\delta)}\sigma_{m-1}^{\alpha+\gamma+\rho(\beta+\delta)}\tau_e^{\beta+\delta},\\
(xy^{-1})^3&=&\sigma_{m-2}^{\rho(\delta-\beta)}\sigma_{m-1}^{\alpha-\gamma+\rho\beta}\tau_e^{\beta-\delta}.
\end{eqnarray*}

\noindent
Therefore the order
of the coset of \(xy\in M_3\) with respect to \(\gamma_2(M_3)\) and
of the coset of \(xy^{-1}\in M_4\) with respect to \(\gamma_2(M_4)\)
is certainly bounded from above by \(3\),
when \(m\ge 6\) and \(e\ge 4\),
and thus \(\sigma_{m-2},\sigma_{m-1},\tau_e\in\gamma_4(G)<\gamma_2(M_i)\), for \(3\le i\le 4\).

It remains to investigate the \(15\) isomorphism classes of groups
with \(m=5\), \(n=7\), \(e=4\),
for which \(\lbrack\chi_s(G),\gamma_e(G)\rbrack=1\) and thus \(\rho=0\)
\cite[pp. 34--35]{Ne2}.
For these isomorphism classes, third powers of the generators are given by

\begin{eqnarray*}
\label{e:LowSpcRel}
(xy)^3&=&\sigma_4^{\alpha+\gamma}\tau_4^{\beta+\delta},\\
(xy^{-1})^3&=&\sigma_4^{\alpha-\gamma}\tau_4^{\beta-\delta},
\end{eqnarray*}

\noindent
and therefore \(\mathrm{ord}(\overline{xy})\le 3\) and \(\mathrm{ord}(\overline{xy^{-1}})\le 3\),
independently from all parameters \(\alpha,\beta,\gamma,\delta\),
since \(\sigma_4,\tau_4\in\gamma_4(G)<\gamma_2(M_i)\), for \(3\le i\le 4\).

\begin{theorem}
\label{t:LowStr}
Let \(K\) be a number field
with \(3\)-classgroup \(\mathrm{Cl}_3(K)\) of type \((3,3)\)
and with second \(3\)-class group \(G=\mathrm{Gal}(\mathrm{F}_3^2(K)\vert K)\)
of order \(\lvert G\rvert=3^n\), \(n\ge 7\),
and of coclass \(\mathrm{cc}(G)\ge 3\), \(m\le n-2\), \(e\ge 4\).\\
Then the structure of the \(3\)-class groups
of the four unramified cyclic cubic extensions \(N_i\) of \(K\) is
nearly homocyclic for \(\mathrm{Cl}_3(N_1)\), \(\mathrm{Cl}_3(N_2)\),
and elementary abelian of type \((3,3,3)\) for
\(\mathrm{Cl}_3(N_3)\), \(\mathrm{Cl}_3(N_4)\),
independently from the principalization type \(\varkappa\) of \(K\).
The number of \(3\)-class groups \(\mathrm{Cl}_3(N_i)\) of type \((3,3,3)\)
is always given by \(\varepsilon=2\).
The extensions with \(3\)-class group of type \((3,3,3)\)
satisfy  Taussky's condition \((\mathrm{B})\)
\cite{Ta},
that is, they have a partial principalization without fixed point.
\end{theorem}

\begin{proof}
The nearly homocyclic structure of the first and second \(3\)-class group,
\(\mathrm{Cl}_3(N_1)\), \(\mathrm{Cl}_3(N_2)\),
is contained in the statement of Theorem
\ref{t:LowHom}
already.\\
Similarly as in the proof of Theorem
\ref{t:SmpSecStr},
the elementary abelian structure of the third and fourth \(3\)-class group,
\(\mathrm{Cl}_3(N_i)\simeq M_i/\gamma_2(M_i)\), \(3\le i\le 4\),
is a consequence of the considerations at the begin of this section,
if we take into account the preparation in \S\
\ref{sss:LowTrd}.
\end{proof}

\begin{example}
\label{ex:LowQdr}
The first occurrences of second \(3\)-class groups
\(G=\mathrm{Gal}(\mathrm{F}_3^2(K)\vert K)\)
of coclass \(\mathrm{cc}(G)\ge 3\)
with invariants \(m\ge 6\), \(n\ge 8\), \(e\ge 4\),
and bicyclic center, \(k=0\), resp. cyclic center, \(k=1\),
over quadratic fields \(K=\mathbb{Q}(\sqrt{D})\)
with \(3\)-class group of type \((3,3)\)
and discriminant \(-10^6<D<0\), resp. \(0<D<10^7\),
are summarized in Table
\ref{tab:LowExp}.
Here, \(\lvert D\rvert\) denotes the smallest absolute value of the discriminant
of a complex quadratic field \(K\)
and \(D\) the smallest discriminant of a real quadratic field \(K\)
of the corresponding principalization type.
The earlier computations by Brink
\cite{Br}
are confirmed.

\begin{table}[ht]
\caption{Examples for groups \(G\) with \(\mathrm{cc}(G)\ge 3\), \(m\ge 6\), \(n\ge 8\), \(e\ge 4\)}
\label{tab:LowExp}
\begin{center}
\begin{tabular}{|lc|rrrr|c|rc|r|}
\hline
 Type             & \(\varkappa\) & \(m\) &  \(n\) & \(e\) & \(k\) & \(\mathrm{Cl}_3(\mathrm{F}_3^1(K))\) & \(\lvert D\rvert\) & ref.      &           \(D\) \\
\hline
 b.10             & \((0043)\)    & \(6\) &  \(8\) & \(4\) & \(1\) &                        \((9,9,3,3)\) &         impossible &           &    \(710\,652\) \\
\hline
 d.19             & \((4043)\)    & \(6\) &  \(8\) & \(4\) & \(0\) &                        \((9,9,3,3)\) &         impossible &           & \(2\,328\,721\) \\
 d.23             & \((1043)\)    & \(6\) &  \(8\) & \(4\) & \(0\) &                        \((9,9,3,3)\) &         impossible &           & \(1\,535\,117\) \\
 d.25\({}^\ast\)  & \((0143)\)    & \(7\) & \(10\) & \(5\) & \(0\) &                       \((27,9,9,3)\) &         impossible &           & \(8\,491\,713\) \\
\hline
 F.7              & \((3443)\)    & \(6\) &  \(9\) & \(5\) & \(0\) &                        \((9,9,9,3)\) &       \(124\,363\) &           &       unknown   \\
 F.7\(\uparrow\)  & \((3443)\)    & \(8\) & \(11\) & \(5\) & \(0\) &                      \((27,27,9,3)\) &       \(469\,816\) &           &       unknown   \\
 F.11             & \((1143)\)    & \(6\) &  \(9\) & \(5\) & \(0\) &                        \((9,9,9,3)\) &        \(27\,156\) &\cite{Br,Ma}&      unknown   \\
 F.11\(\uparrow\) & \((1143)\)    & \(8\) & \(11\) & \(5\) & \(0\) &                      \((27,27,9,3)\) &       \(469\,787\) &           &       unknown   \\
 F.12             & \((1343)\)    & \(6\) &  \(9\) & \(5\) & \(0\) &                        \((9,9,9,3)\) &        \(31\,908\) & \cite{Br} &       unknown   \\
 F.12\(\uparrow\) & \((1343)\)    & \(8\) & \(11\) & \(5\) & \(0\) &                      \((27,27,9,3)\) &       \(249\,371\) &           &       unknown   \\
 F.12\(\uparrow^2\)&\((1343)\)    & \(8\) & \(13\) & \(7\) & \(0\) &                     \((27,27,27,9)\) &       \(423\,640\) &           &       unknown   \\
 F.13             & \((3143)\)    & \(6\) &  \(9\) & \(5\) & \(0\) &                        \((9,9,9,3)\) &        \(67\,480\) & \cite{Br} & \(8\,321\,505\) \\
 F.13\(\uparrow\) & \((3143)\)    & \(8\) & \(11\) & \(5\) & \(0\) &                      \((27,27,9,3)\) &       \(159\,208\) &           & \(8\,127\,208\) \\
\hline
 G.16r            & \((1243)\)    & \(7\) & \(10\) & \(5\) & \(1\) &                       \((27,9,9,3)\) &       \(290\,703\) &           &       unknown   \\
 G.16i            & \((1243)\)    & \(7\) & \(10\) & \(5\) & \(1\) &                        \((9,9,9,9)\) &       \(135\,059\) &           &       unknown   \\
 G.19r            & \((2143)\)    & \(7\) & \(10\) & \(5\) & \(1\) &                       \((27,9,9,3)\) &        \(96\,827\) &           &       unknown   \\
 G.19r\(\uparrow\)& \((2143)\)    & \(9\) & \(12\) & \(5\) & \(1\) &                      \((81,27,9,3)\) &       \(509\,160\) &           &       unknown   \\
 G.19i            & \((2143)\)    & \(7\) & \(10\) & \(5\) & \(1\) &                        \((9,9,9,9)\) &       \(199\,735\) &           &       unknown   \\
 H.4r             & \((3343)\)    & \(7\) & \(10\) & \(5\) & \(1\) &                       \((27,9,9,3)\) &       \(256\,935\) &           &       unknown   \\
 H.4r\(\uparrow\) & \((3343)\)    & \(9\) & \(12\) & \(5\) & \(1\) &                      \((81,27,9,3)\) &       \(678\,804\) &           &       unknown   \\
 H.4i             & \((3343)\)    & \(7\) & \(10\) & \(5\) & \(1\) &                        \((9,9,9,9)\) &       \(186\,483\) &           &       unknown   \\
\hline
\end{tabular}
\end{center}
\end{table}

\noindent
Similarly as in example
\ref{ex:SecQdr},
these principalization types belong to infinite families of
second \(3\)-class groups \(G=\mathrm{Gal}(\mathrm{F}_3^2(K)\vert K)\).
Here, additionally to the nilpotency class \(\mathrm{cl}(G)=m-1\),
the coclass \(\mathrm{cc}(G)=e-1=n-m+1\) can also take infinitely many values.

For the principalization types
\(\mathrm{F.7}\), \(\mathrm{F.11}\), \(\mathrm{F.12}\), \(\mathrm{F.13}\)
with defect \(k=0\),
all even values \(m\ge 6\) and all odd values \(5\le e\le m-1\) are possible.

The very rare principalization types
\(\mathrm{d.19}\), \(\mathrm{d.23}\), \(\mathrm{d.25}\),
also having \(k=0\),
play a unique exceptional role
\cite[Thm. 3.4, p. 491]{Ma2},
since their second \(3\)-class groups
\(G=\mathrm{Gal}(\mathrm{F}_3^2(K)\vert K)\)
can appear either as terminal metabelian vertices (leaves)
with even values \(m\ge 6\) and even values \(4\le e\le m-1\)
or as infinitely capable vertices on mainlines of coclass trees
with odd values \(m\ge 7\) and odd values \(5\le e\le m-1\).
 
Assuming the weak leaf conjecture
\ref{cnj:SmpSecQdr},
the following types show up with defect \(k=1\) only. 
For the principalization type \(\mathrm{b.10}\)
all even values \(m\ge 6\) and all even values \(4\le e\le m-2\) can occur,
and for the principalization types
\(\mathrm{G.16}\), \(\mathrm{G.19}\), and \(\mathrm{H.4}\)
all odd values \(m\ge 7\) and all odd values \(5\le e\le m-2\) are admissible.

Concrete numerical realizations by complex quadratic fields are known
for the ground state of these principalization types with minimal index of nilpotency \(m\),
but only partially by real quadratic base fields.
\end{example}


\section{Implementing the principalization algorithm for quadratic fields}
\label{s:CompTech}

In this section we describe the computational techniques
used to achieve the numerical results presented in \S\
\ref{s:NumerTab}.
The new principalization algorithm via class group structure
has been implemented for quadratic fields,
having a \(3\)-class group of type \((3,3)\),
with the aid of program scripts
written for the number theoretical computer algebra system PARI/GP
\cite{Be,PARI}.
We refer to the relevant methods of this software package
by printing their names in \texttt{typewriter} font with trailing parentheses.

\subsection{Generating polynomials for non-Galois cubic fields \(L\)}
\label{ss:GenPol}

\subsubsection{Simply real cubic fields}
\label{sss:GenPolSRC}

Generating polynomials \(p(X)=X^3-bX^2+cX-d\) of third degree
for simply real cubic fields \(L\) of signature \((1,1)\)
are obtained in the following manner.
Suppose the intended upper bound
for the absolute value of the field discriminant is
\(\lvert \mathrm{d}(L)\rvert\le U\).
Then candidates for the coefficients \(b\), \(c\), and \(d\)
run over three nested loops
\(1\le b\le\lfloor 3+2\root{4}\of{U}\rfloor\),
\(1\le c\le\lfloor\frac{b^2+\sqrt{U}}{3}\rfloor\),
and
\(1\le d\le\lfloor\frac{b^2-3+2\sqrt{U}}{6}\rfloor\),
with bounds due to Godwin and Angell,
given by Fung and Williams
\cite[(2.5), p. 315]{FuWi}.
For each triplet \((b,c,d)\)
the following tests are performed.

\begin{enumerate}

\item
Reducible polynomials \(p(X)\) are eliminated with the aid of
\texttt{polisirreducible()}.

\item
For irreducible polynomials
the discriminant \(\mathrm{d}(L)\) of the cubic field \(L=\mathbb{Q}(\vartheta)\),
generated by the real zero \(\vartheta\) of \(p(X)\),
is calculated by means of
\texttt{nfdisc()}
and tested for \(-U\le \mathrm{d}(L)<0\).

\item
By
\texttt{poldisc()}
the discriminant \(\mathrm{d}(p)\) of the polynomial is computed
and its index \(\mathrm{i}(p)\) is determined, using the formula
\(\mathrm{d}(p)=\mathrm{i}(p)^2\cdot \mathrm{d}(L)\).
Polynomials with indices bigger than the bound
\[
\frac{\sqrt{124b^2+432b+4\sqrt{U}+729}}{3\sqrt{3}}
\]
are skipped, thus discouraging superfluous isomorphic fields
\cite[(2.6), p. 315]{FuWi}.

\item
Field discriminants \(\mathrm{d}(L)=f^2\cdot\mathrm{d}(K)\)
\cite[\S\ 1, p. 832]{Ma3}
are restricted to fundamental discriminants with conductor \(f=1\) by
\texttt{isfundamental()}.

\item
Finally, the \(3\)-class group \(\mathrm{Cl}_3(K)\) of the complex quadratic subfield \(K\)
of the Galois closure \(N\) of \(L\),
which is unramified with conductor \(f=1\) over \(K\)
and has discriminant \(\mathrm{d}(N)=\mathrm{d}(K)^3\)
\cite[Abstract, p. 831]{Ma3},
is restricted to the type \((3,3)\)
with the aid of
\texttt{quadclassunit()},
thereby eliminating the numerous cyclic \(3\)-class groups.

\end{enumerate}


\subsubsection{Totally real cubic fields}
\label{sss:GenPolTRC}

Trace free generating polynomials \(p(X)=X^3-cX-d\)
for totally real cubic fields \(L\) of signature \((3,0)\)
are collected in the following way.
If the desired upper bound for the field discriminant is
\(\mathrm{d}(L)\le U\), then
candidates for the coefficients \(c\) and \(d\)
run over two nested loops
\(1\le c\le\lfloor\sqrt{U}\rfloor\)
and
\(1\le d\le\lfloor\sqrt{\frac{4c^3}{27}}\rfloor\),
with bounds given by Llorente and Quer
\cite[\S~3, p. 586]{LlQu}.
For each pair \((c,d)\)
the following tests are performed.

\begin{enumerate}

\item
Reducible polynomials \(p(X)\) are skipped with the aid of
\texttt{polisirreducible()}.

\item
For irreducible polynomials
the discriminant \(\mathrm{d}(L)\) of the cubic field \(L=\mathbb{Q}(\xi)\),
generated by a zero \(\xi\) of \(p(X)\),
is calculated by means of
\texttt{nfdisc()}
and checked for \(0<\mathrm{d}(L)\le U\).

\item
A further bound
\cite[Thm. 3, p. 584]{LlQu}
is imposed on the linear coefficient
\[c\le
\begin{cases}
\lfloor\frac{\sqrt{\mathrm{d}(L)}}{3}\rfloor, & \text{ if } 27\mid \mathrm{d}(L),\\
\lfloor\sqrt{\mathrm{d}(L)}\rfloor,           & \text{ otherwise}.
\end{cases}
\]

\item
By
\texttt{poldisc()}
the discriminant \(\mathrm{d}(p)\) of the polynomial is computed
and its index \(\mathrm{i}(p)\) is determined, using the formula
\(\mathrm{d}(p)=\mathrm{i}(p)^2\cdot \mathrm{d}(L)\).
Polynomials with index bigger than the bound
\[
\begin{cases}
\lfloor 2\sqrt{\frac{c}{3}}\rfloor, & \text{ if } 27\mid \mathrm{d}(L),\\
\lfloor 2\sqrt{c}\rfloor,           & \text{ otherwise},
\end{cases}
\]
are eliminated
\cite[Thm. 3, p. 584]{LlQu}.

\item
Field discriminants \(\mathrm{d}(L)\) are restricted to fundamental discriminants by
\texttt{isfundamental()}.

\item
Finally, the \(3\)-class group \(\mathrm{Cl}_3(K)\) of the real quadratic subfield \(K\)
of the Galois closure \(N\) of \(L\)
is restricted to the type \((3,3)\)
with the aid of
\texttt{quadclassunit()}.

\end{enumerate}


\subsection{Structure of \(3\)-class groups \(\mathrm{Cl}_3(N)\) of \(S_3\)-fields \(N\)}
\label{ss:ClsGrpStr}

The generating polynomials of \S\
\ref{ss:GenPol}
are stored as quadruplets \((D,b,c,d)\) for \(D<0\),
ordered by descending discriminants \(D\),
resp. as triplets \((D,c,d)\) for \(D>0\),
ordered by ascending discriminants \(D\).
The bounds in
\cite{FuWi},
resp.
\cite{LlQu},
provide a warranty that,
for each discriminant \(D=\mathrm{d}(K)=\mathrm{d}(L)\),
generating polynomials for all four non-isomorphic cubic fields \(L\)
sharing the same discriminant \(\mathrm{d}(L)=D\)
\cite[Cor. 3.1, p. 838]{Ma3}
are contained in the list.
Now the polynomials are iterated through the list in a single loop
and for each of them the following steps are executed.

\begin{enumerate}

\item
The regulator \(\mathrm{R}(L)\) and the class number \(\mathrm{h}(L)\) of \(L\)
are calculated with the aid of
\texttt{bnfinit()}
using the flag \(1\),
which ensures that a fundamental system of units of \(L\) is determined.
An indicator is stored,
if the class number \(\mathrm{h}(L)\) is divisible by \(9\) or \(27\).

\item
By means of
\texttt{polcompositum()},
applied to the cubic polynomial
\(p(X)=X^3-bX^2+cX-d\), resp. \(p(X)=X^3-cX-d\),
and the quadratic polynomial \(q(X)=X^2-D\),
a generating polynomial \(s(X)\) of sixth degree for the Galois closure \(N\) of \(L\)
is calculated.

\item
The polynomial \(s(X)\) of sixth degree is used to determine
the structure of the class group \(\mathrm{Cl}(N)\) of the normal field \(N\)
with the aid of
\texttt{bnfinit()},
where the flag is set to \(1\).
An indicator is stored,
if the first three abelian type invariants \((n_1,n_2,n_3,\ldots)\) of
the class group structure are all divisible by \(3\), that is,
if the \(3\)-class group  \(\mathrm{Cl}_3(N)\) of \(N\)
is elementary abelian of type \((3,3,3)\).

\end{enumerate}

\noindent
The results are evaluated in the following way.
The structures of the \(3\)-class groups
\(\mathrm{Cl}_3(N_i)\), \(1\le i\le 4\),
of quadruplets \(N_1,\ldots,N_4\) of \(S_3\)-fields
sharing the same discriminant \(D^3\),
form the transfer target type \(\tau(K)\) and
determine the number \(\varepsilon(K)\) of groups of type \((3,3,3)\).
If some \(3\)-class group \(\mathrm{Cl}_3(N_i)\) is of type \((3,3)\),
that is, isomorphic to \(\mathrm{A}(3,2)\),
then the second \(3\)-class group \(G\) of \(K\) 
is of coclass \(\mathrm{cc}(G)=1\), by Theorem
\ref{t:MaxHom},
and the algorithm can be terminated.
Otherwise \(G\) is of coclass \(\mathrm{cc}(G)\ge 2\) 
and the algorithm can be terminated
only if no indicator of \(9\mid\mathrm{h}(L_i)\) has been stored,
for any \(1\le i\le 4\).
The transfer target type \(\tau(K)\) and the invariant \(\varepsilon(K)\) determine
the principalization type (transfer kernel type) \(\varkappa(K)\)
and the structure of the second \(3\)-class group
\(G=\mathrm{G}_3^2(K)=\mathrm{Gal}(\mathrm{F}_3^2(K)\vert K)\)
of the quadratic field \(K\) in two important special cases:
uniquely for sporadic \(G\) of coclass \(\mathrm{cc}(G)=2\),
according to Theorems
\ref{t:SmpSecStr}
and
\ref{t:SmlSecStr},
and up to separation of types \(\mathrm{a.2}\) and \(\mathrm{a.3}\)
for \(\mathrm{cc}(G)=1\),
according to Theorem
\ref{t:MaxExo}.


\subsection{First Hilbert \(3\)-class field \(\mathrm{F}_3^1(K)\) of \(K\)}
\label{ss:FstHilb}

For the quadruplets \(L_1,\ldots,L_4\) of cubic fields,
sharing the same discriminant \(D\),
which have been marked by an indicator of \(9\mid\mathrm{h}(L_i)\),
for some \(1\le i\le 4\),
a further step must be appended, when the second \(3\)-class group
\(G=\mathrm{Gal}(\mathrm{F}_3^2(K)\vert K)\) of the quadratic field \(K\)
is of coclass \(\mathrm{cc}(G)\ge 2\).
A generating polynomial \(f(X)\) of eighteenth degree for the
first Hilbert \(3\)-class field \(\mathrm{F}_3^1(K)=N_i\cdot L_j\) of \(K\)
is calculated by means of
\texttt{polcompositum()},
applied to
the generating polynomial \(s(X)\) of sixth degree of \(N_i\)
and the cubic generating polynomial \(p(X)\) of \(L_j\),
for some \(1\le i\ne j\le 4\).
The polynomial \(f(X)\) of eighteenth degree is used to determine
the structure of the class group \(\mathrm{Cl}(\mathrm{F}_3^1(K))\) of
the first Hilbert \(3\)-class field \(\mathrm{F}_3^1(K)\) of the quadratic field \(K\)
with the aid of
\texttt{bnfinit()},
where the flag is set to \(1\).
Whereas computations for fields of third and sixth degree
are usually a matter of less than a second in PARI/GP,
the CPU time for a field of degree \(18\) may reach a few minutes, occasionally.

Together with the transfer target type \(\tau(K)\)
and the number \(\varepsilon(K)\) of groups of type \((3,3,3)\),
the structure of the \(3\)-class group \(\mathrm{Cl}_3(\mathrm{F}_3^1(K))\)
determines the principalization type \(\varkappa(K)\)
and the structure of the second \(3\)-class group
\(G=\mathrm{Gal}(\mathrm{F}_3^2(K)\vert K)\)
of the quadratic field \(K\) in the following manner:
up to separation of types \(\mathrm{E.6}\), \(\mathrm{E.14}\),
resp. \(\mathrm{E.8}\), \(\mathrm{E.9}\),
for \(\mathrm{cc}(G)=2\) by means of Theorem
\ref{t:SecStr},
and up to separation of types
\(\mathrm{F.7}\), \(\mathrm{F.11}\), \(\mathrm{F.12}\), \(\mathrm{F.13}\),
resp. \(\mathrm{G.16}\), \(\mathrm{G.19}\),\(\mathrm{H.4}\),
resp. \(\mathrm{d.19}\), \(\mathrm{d.23}\),\(\mathrm{d.25}\),
for \(\mathrm{cc}(G)\ge 3\) by means of Theorem
\ref{t:LowStr}
and
\cite[Thm. 5.1--5.3, pp. 491--494]{Ma1}.

In our paper
\cite{Ma1},
we have indicated another computational technique,
trying to avoid the use of the highly sophisticated system PARI/GP.
The invariants of the second \(3\)-class group
\(G=\mathrm{Gal}(\mathrm{F}_3^2(K)\vert K)\),
namely the class \(\mathrm{cl}(G)=m-1\), coclass \(\mathrm{cc}(G)=e-1\),
and the order \(\lvert G\rvert=3^n\), \(n=\mathrm{cl}(G)+\mathrm{cc}(G)\),
are determined with the aid of \(3\)-class numbers \(\mathrm{h}(L_i)\), \(1\le i\le 4\),
of cubic fields, using
\cite[Thm. 5.1--5.3]{Ma1}.
These class numbers can be computed by means of our own implementation
of the classical Voronoi algorithms for calculating
integral bases
\cite{Vo1}
and fundamental systems of units
\cite{Vo2}
for simply or totally real cubic fields
and subsequent application of
the analytic class number formula and Euler product method.
The transfer kernel type \(\varkappa\) of \(G\) 
is assumed as an experimental input data in
\cite{Ma1},
produced by some unspecified principalization algorithm,
for instance the classical algorithm in \cite{SoTa,HeSm,Br,Ma}.
Consequently, the difficult determination of the defect \(k\)
with the aid of the Hilbert \(3\)-class field \(\mathrm{F}_3^1(K)\) of \(K\)
is circumvented,
when the weak leaf conjecture
\ref{cnj:SmpSecQdr}
is assumed to hold.
We have successfully applied this cumbersome classical procedure
to the restricted range \(-10^5<D<10^6\) of quadratic discriminants,
for which we needed \(7\) years from \(2003\) to \(2009\).
However, the extension to \(-10^6<D<10^7\) definitely requires
the high speed performance of PARI/GP or MAGMA
and the automatizability of our new principalization algorithm.
Including the manual evaluation it was done within \(5\) months in \(2010\). 
CPU time added up to a total of a few weeks.


\section{Numerical results on second \(3\)-class groups of \(4\,596\) quadratic fields}
\label{s:NumerTab}

By means of the principalization algorithm,
implemented in the PARI/GP programs of \S\
\ref{s:CompTech},
the principalization type \(\varkappa(K)\)
\cite[Tbl. 6--7, pp. 492--493]{Ma2}
and the structure of the second \(3\)-class group
\(G=\mathrm{G}_3^2(K)=\mathrm{Gal}(\mathrm{F}_3^2(K)\vert K)\)
\cite{Ma1}
has been determined for
the \(2\,020\) complex quadratic fields \(K=\mathbb{Q}(\sqrt{D})\)
with discriminant \(-10^6<D<0\) and for
the \(2\,576\) real quadratic fields \(K=\mathbb{Q}(\sqrt{D})\)
with discriminant \(0<D<10^7\),
having a \(3\)-class group \(\mathrm{Cl}_3(K)\) of type \((3,3)\).
The results of these extensive computations
reveal reliable statistical tendencies concerning
the distribution of the groups \(G\) on
the sporadic part and various coclass trees of the coclass graphs
\(\mathcal{G}(3,r)\), \(1\le r\le 6\),
\cite{LgNm,LgMk,EkLg}
for a total of \(4\,596\) quadratic base fields.

Each of the following tables gives the characterizing
\textit{transfer target type} (TTT) \(\tau\),
as a novelty which was unknown up to now,
and the minimal discriminant
and absolute frequency
of various transfer kernel types (TKTs) \(\varkappa\).
Here, \(11\) tables are arranged according to the sign of the discriminant \(D\)
and the \textit{graph theoretic location} of the second \(3\)-class group \(G\),
either on sporadic parts or on branches of coclass trees,
which constitute the coclass graphs,
thus providing necessary information
for drawing graphical diagrams of coclass graphs in the subsequent paper
\cite{Ma4}.
In contrast, the four Tables 2--5 in our previous paper
\cite[\S\ 6, pp. 496--499]{Ma1}
were arranged according to the number \(\nu\)
of total principalizations.

\small{

\renewcommand{\arraystretch}{1.0}
\begin{table}[ht]
\caption{Principalization types with \(G\in\mathcal{G}(3,1)\) for \(D>0\)}
\label{tab:RealMax}
\begin{center}
\begin{tabular}{|lc|c|c|cc|c|r|c|lr|}
\hline
 type            & \(\varkappa\) & \(j\) & \(\mathrm{h}_3(L_1)\) & \(\mathrm{Cl}_3(N_1)\) & \(\varepsilon\) & \(\mathrm{Cl}_3(\mathrm{F}_3^1(K))\) &      min. \(D\) & ref.        & \multicolumn{2}{|c|}{freq.}                                     \\
\hline
 a.2             & \((1000)\)    & \(3\) &                 \(3\) &              \((9,3)\) &           \(0\) &                            \((3,3)\) &     \(72\,329\) & \cite{HeSm} & \multirow{2}{*}{\(\Bigr\rbrace\)} & \multirow{2}{*}{\(1\,386\)} \\
 a.3             & \((2000)\)    & \(3\) &                 \(3\) &              \((9,3)\) &           \(0\) &                            \((3,3)\) &     \(32\,009\) & \cite{HeSm} &                                   &                             \\
\cline{10-11}
 a.3\({}^\ast\)  & \((2000)\)    & \(3\) &                 \(3\) &            \((3,3,3)\) &           \(1\) &                            \((3,3)\) &    \(142\,097\) & \cite{Ma0}  &                                   &                     \(697\) \\
\hline
 a.1             & \((0000)\)    & \(5\) &                 \(9\) &              \((9,9)\) &           \(0\) &                            \((9,9)\) &     \(62\,501\) & \cite{HeSm} &                                   &                     \(147\) \\
\cline{10-11}
 a.2\(\uparrow\) & \((1000)\)    & \(5\) &                 \(9\) &             \((27,9)\) &           \(0\) &                            \((9,9)\) &    \(790\,085\) &             & \multirow{2}{*}{\(\Bigr\rbrace\)} &     \multirow{2}{*}{\(72\)} \\
 a.3\(\uparrow\) & \((2000)\)    & \(5\) &                 \(9\) &             \((27,9)\) &           \(0\) &                            \((9,9)\) &    \(494\,236\) &             &                                   &                             \\
\hline
 a.1\(\uparrow\) & \((0000)\)    & \(7\) &                \(27\) &            \((27,27)\) &           \(0\) &                          \((27,27)\) & \(2\,905\,160\) &             &                                   &                       \(1\) \\
\hline
\multicolumn{9}{|r|}{total:}																																																																	&  & \(2\,303\) \\
\hline
\end{tabular}
\end{center}
\end{table}

}
\normalsize

Table
\ref{tab:RealMax}
characterizes the most frequent types \(\varkappa\)
of real quadratic fields with \(G\) of coclass \(1\)
by means of \(\mathrm{h}_3(L_1)\) and \(\mathrm{Cl}_3(N_1)\).
For the other fields with \(2\le i\le 4\), we always have
\(\mathrm{h}_3(L_i)=3\) and \(\mathrm{Cl}_3(N_i)\) of type \((3,3)\).
The second \(3\)-class group \(G\) is a vertex of depth \(1\)
on an odd branch \(\mathcal{B}(j)\), \(j\in\lbrace 3,5,7\rbrace\),
of the unique coclass tree \(\mathcal{T}(\mathrm{C}_3\times\mathrm{C}_3)\)
of \(\mathcal{G}(3,1)\) in Figure
\ref{fig:CoCl1}.
Mainline groups of depth \(0\) do not occur
and are probably impossible for quadratic fields,
as stated in our Conjecture
\ref{cnj:MaxTrm}.
The types a.2 and a.3, resp. a.2\(\uparrow\) and a.3\(\uparrow\),
can only be separated with the aid of the classical principalization algorithm
\cite{SoTa,HeSm}.
Among the \(2\,576\) real quadratic fields \(K=\mathbb{Q}(\sqrt{D})\)
with discriminant \(0<D<10^7\)
and \(3\)-class group \(\mathrm{Cl}_3(K)\) of type \((3,3)\),
the dominating part of \(2\,303\) fields, that is \(89.4\%\),
has a second \(3\)-class group \(G\) of coclass \(\mathrm{cc}(G)=1\).
Branch \(\mathcal{B}(3)\) is populated most densely
by \(\frac{697}{2303}=30.3\%\) groups of type a.3\({}^\ast\) and
\(\frac{1386}{2303}=60.2\%\) groups of types a.2 and a.3.

\small{

\begin{table}[ht]
\caption{Principalization types with \(G\in\mathcal{T}(\langle 2187,64\rangle)\subset\mathcal{G}(3,3)\) for \(D>0\)}
\label{tab:RealMix}
\begin{center}
\begin{tabular}{|lc|cc|ccccc|c|r|r|}
\hline
 type  & \(\varkappa\) & \(\mathrm{h}_3(L_1)\) & \(\mathrm{h}_3(L_2)\) & \(\mathrm{Cl}_3(N_1)\) & \(\mathrm{Cl}_3(N_2)\) & \(\mathrm{Cl}_3(N_3)\) & \(\mathrm{Cl}_3(N_4)\) & \(\varepsilon\) & \(\mathrm{Cl}_3(\mathrm{F}_3^1(K))\) &      min. \(D\) &    fr. \\
\hline
 b.10  & \((0043)\)    &                 \(9\) &                 \(9\) &              \((9,9)\) &              \((9,9)\) &            \((3,3,3)\) &            \((3,3,3)\) &           \(2\) &                        \((9,9,3,3)\) &    \(710\,652\) &  \(8\) \\
 d.19  & \((4043)\)    &                 \(9\) &                 \(9\) &             \((27,9)\) &              \((9,9)\) &            \((3,3,3)\) &            \((3,3,3)\) &           \(2\) &                        \((9,9,3,3)\) & \(2\,328\,721\) &  \(1\) \\
 d.23  & \((1043)\)    &                 \(9\) &                 \(9\) &             \((27,9)\) &              \((9,9)\) &            \((3,3,3)\) &            \((3,3,3)\) &           \(2\) &                        \((9,9,3,3)\) & \(1\,535\,117\) &  \(1\) \\
\hline
\multicolumn{11}{|r|}{total:}																																																				 	                                    & \(10\) \\
\hline
\end{tabular}
\end{center}
\end{table}

}
\normalsize

Table
\ref{tab:RealMix}
characterizes the rare types \(\varkappa\)
of real quadratic fields with \(G\) of coclass \(3\)
by means of \(\mathrm{h}_3(L_i)\), \(1\le i\le 2\),
and \(\mathrm{Cl}_3(N_i)\), \(1\le i\le 4\).
For these cases we have \(\mathrm{h}_3(L_i)=3\), for \(3\le i\le 4\),
and the second \(3\)-class group \(G\)
is a vertex of depth \(1\) on the odd branch \(\mathcal{B}(7)\)
of the tree with root \(\langle 2187,64\rangle\simeq G_0^{5,7}(0,0,0,0)\)
\cite[p. 189 ff.]{Ne1}
of \(\mathcal{G}(3,3)\).
The types d.19, d.23, d.25
\cite[Thm. 3.4, p. 491]{Ma2}
can only be separated by the classical principalization algorithm \cite{SoTa,HeSm}.
Among the \(2\,576\) real quadratic fields under investigation,
only a negligible part of \(10\) fields, that is \(0.4\%\),
has a second \(3\)-class group \(G\) of coclass \(\mathrm{cc}(G)=3\).

Whereas groups on \(\mathcal{G}(3,1)\), \(\mathcal{G}(3,3)\),
and generally groups with odd coclass,
are impossible for complex quadratic fields
\cite[Thm. 4.2, p. 489]{Ma1},
we now proceed to coclass graphs
which are populated by second \(3\)-class groups of
quadratic fields with either sign of the discriminant.

\small{

\begin{table}[ht]
\caption{Principalization types with sporadic \(G\in\mathcal{G}(3,2)\) for \(D>0\)}
\label{tab:RealPartSpor}
\begin{center}
\begin{tabular}{|lc|cc|ccccc|c|r|lr|}
\hline
 type & \(\varkappa\) & \(\mathrm{h}_3(L_1)\) & \(\mathrm{h}_3(L_2)\) & \(\mathrm{Cl}_3(N_1)\) & \(\mathrm{Cl}_3(N_2)\) & \(\mathrm{Cl}_3(N_3)\) & \(\mathrm{Cl}_3(N_4)\) & \(\varepsilon\) & \(\mathrm{Cl}_3(\mathrm{F}_3^1(K))\) &      min. \(D\) & \multicolumn{2}{|c|}{freq.}                                 \\
\hline                                                                                                                                                                     
 G.19 & \((2143)\)    &                 \(3\) &                 \(3\) &              \((9,3)\) &              \((9,3)\) &              \((9,3)\) &              \((9,3)\) &           \(0\) &                        \((3,3,3,3)\) &    \(214\,712\) &                                   &                  \(11\) \\
 D.10 & \((2241)\)    &                 \(3\) &                 \(3\) &              \((9,3)\) &              \((9,3)\) &            \((3,3,3)\) &              \((9,3)\) &           \(1\) &                          \((3,3,3)\) &    \(422\,573\) &                                   &                  \(93\) \\
 D.5  & \((4224)\)    &                 \(3\) &                 \(3\) &            \((3,3,3)\) &              \((9,3)\) &            \((3,3,3)\) &              \((9,3)\) &           \(2\) &                          \((3,3,3)\) &    \(631\,769\) &                                   &                  \(47\) \\
 H.4  & \((4443)\)    &                 \(3\) &                 \(3\) &            \((3,3,3)\) &            \((3,3,3)\) &              \((9,3)\) &            \((3,3,3)\) &           \(3\) &                          \((9,3,3)\) &    \(957\,013\) &                                   &                  \(27\) \\
\hline
\multicolumn{11}{|r|}{total:}																																																				 	  &  & \(178\) \\
\hline
\end{tabular}
\end{center}
\end{table}

}
\normalsize

The second \(3\)-class groups \(G\) for types \(\varkappa\) in Table
\ref{tab:RealPartSpor}
are sporadic vertices on \(\mathcal{G}(3,2)\) in Figure 
\ref{fig:Typ33CoCl2}.
They are determined uniquely by \(\varepsilon\) (Cor.
\ref{c:SmpSecQdr}).
Among the \(263\) groups \(G\) of even coclass for real quadratic fields,
a contribution of \(178\) groups, that is \(67.7\%\),
is sporadic of coclass \(\mathrm{cc}(G)=2\).
This is the adequate kind of relative frequencies
for comparison with complex quadratic fields.

\small{

\begin{table}[ht]
\caption{Principalization types with sporadic \(G\in\mathcal{G}(3,2)\) for \(D<0\)}
\label{tab:CompSpor}
\begin{center}
\begin{tabular}{|lc|cc|ccccc|c|r|lr|}
\hline
 type               & \(\varkappa\) & \(\mathrm{h}(L_1)\) & \(\mathrm{h}(L_2)\) & \(\mathrm{Cl}_3(N_1)\) & \(\mathrm{Cl}_3(N_2)\) & \(\mathrm{Cl}_3(N_3)\) & \(\mathrm{Cl}_3(N_4)\) & \(\varepsilon\) & \(\mathrm{Cl}_3(\mathrm{F}_3^1(K))\) & min.\(\lvert D\rvert\) & \multicolumn{2}{|c|}{freq.}                                   \\
\hline
 G.19               & \((2143)\)    &                 \(3\) &                 \(3\) &              \((9,3)\) &              \((9,3)\) &              \((9,3)\) &              \((9,3)\) &           \(0\) &                        \((3,3,3,3)\) &             \(12\,067\) &                                        &                  \(94\)  \\
 D.10               & \((2241)\)    &                 \(3\) &                 \(3\) &              \((9,3)\) &              \((9,3)\) &            \((3,3,3)\) &              \((9,3)\) &           \(1\) &                          \((3,3,3)\) &              \(4\,027\) &                                        &                  \(667\) \\
 D.5                & \((4224)\)    &                 \(3\) &                 \(3\) &            \((3,3,3)\) &              \((9,3)\) &            \((3,3,3)\) &              \((9,3)\) &           \(2\) &                          \((3,3,3)\) &             \(12\,131\) &                                        &                  \(269\) \\
 H.4                & \((4443)\)    &                 \(3\) &                 \(3\) &            \((3,3,3)\) &            \((3,3,3)\) &              \((9,3)\) &            \((3,3,3)\) &           \(3\) &                          \((9,3,3)\) &              \(3\,896\) &                                        &                 \(297\)  \\
\hline
\multicolumn{11}{|r|}{total:}																																																				 	                              &      & \(1\,327\) \\
\hline
\end{tabular}
\end{center}
\end{table}

}
\normalsize

Table
\ref{tab:CompSpor}
is the complex counterpart of Table
\ref{tab:RealPartSpor}.
Among the \(2\,020\) complex quadratic fields \(K=\mathbb{Q}(\sqrt{D})\)
with discriminant \(-10^6<D<0\)
and \(3\)-class group \(\mathrm{Cl}_3(K)\) of type \((3,3)\)
a considerable part of \(1\,327\) fields, that is \(65.7\%\),
has a sporadic second \(3\)-class group \(G\) of coclass \(\mathrm{cc}(G)=2\).
Type D.10 with a relative frequency of \(\frac{667}{2020}=33.0\%\)
is the absolute high-champ for complex quadratic fields.

\small{

\begin{table}[ht]
\caption{Principalization types with \(G\in\mathcal{T}(\langle 243,6\rangle)\subset\mathcal{G}(3,2)\) for \(D>0\)}
\label{tab:RealTreeQ}
\begin{center}
\begin{tabular}{|lc|c|cc|ccccc|c|r|lr|}
\hline
 type            & \(\varkappa\) & \(j\) & \(\mathrm{h}_3(L_1)\) & \(\mathrm{h}_3(L_2)\) & \(\mathrm{Cl}_3(N_1)\) & \(\mathrm{Cl}_3(N_2)\) & \(\mathrm{Cl}_3(N_3)\) & \(\mathrm{Cl}_3(N_4)\) & \(\varepsilon\) & \(\mathrm{Cl}_3(\mathrm{F}_3^1(K))\) &      min. \(D\) & \multicolumn{2}{|c|}{freq.}                                 \\
\hline                                                                                                                                                                     
 c.18            & \((0313)\)    & \(6\) &                 \(9\) &                 \(3\) &              \((9,9)\) &              \((9,3)\) &            \((3,3,3)\) &              \((9,3)\) &           \(1\) &                          \((9,3,3)\) &    \(534\,824\) & & \(29\) \\
\cline{13-14}                                                     
 E.6             & \((1313)\)    & \(6\) &                 \(9\) &                 \(3\) &             \((27,9)\) &              \((9,3)\) &            \((3,3,3)\) &              \((9,3)\) &           \(1\) &                          \((9,9,3)\) & \(5\,264\,069\) & \multirow{2}{*}{\(\Bigr\rbrace\)} &  \multirow{2}{*}{\(7\)} \\
 E.14            & \((2313)\)    & \(6\) &                 \(9\) &                 \(3\) &             \((27,9)\) &              \((9,3)\) &            \((3,3,3)\) &              \((9,3)\) &           \(1\) &                          \((9,9,3)\) & \(3\,918\,837\) &                                   &                         \\
\cline{13-14}                                                     
 H.4\(\uparrow\) & \((3313)\)    & \(6\) &                 \(9\) &                 \(3\) &             \((27,9)\) &              \((9,3)\) &            \((3,3,3)\) &              \((9,3)\) &           \(1\) &                         \((27,9,3)\) & \(1\,162\,949\) &                                   &                   \(3\) \\
\hline
\multicolumn{12}{|r|}{total:}																																																				 	  &  & \(39\) \\
\hline
\end{tabular}
\end{center}
\end{table}

}
\normalsize

The second \(3\)-class groups \(G\) for types \(\varkappa\) in Table
\ref{tab:RealTreeQ}
are vertices on the coclass tree with root
\(\langle 243,6\rangle\simeq G_0^{4,5}(0,-1,0,1)\) of \(\mathcal{G}(3,2)\).
\(G\) is a vertex
of depth \(0\) for type c.18,
of depth \(1\) for types E.6, E.14, and
of depth \(2\) for type H.4,
on the even branch \(\mathcal{B}(6)\),
\cite[Tbl. 2, p. 266]{AHL},
\cite[Fig. 4.8, p. 76]{As},
\cite[p. 189 ff.]{Ne1}.
The types E.6, E.14
can only be separated by the classical principalization algorithm
\cite{SoTa,HeSm}.
Type H.4 can be identified either by the first Hilbert \(3\)-class field
or by the classical principalization algorithm.
Among the \(263\) groups of even coclass for real quadratic fields,
a fraction of \(39\) groups, that is \(14.8\%\),
populates this tree.

\small{

\begin{table}[ht]
\caption{Principalization types with \(G\in\mathcal{T}(\langle 243,6\rangle)\subset\mathcal{G}(3,2)\) for \(D<0\)}
\label{tab:CompTreeQ}
\begin{center}
\begin{tabular}{|lc|c|cc|ccccc|c|r|lr|}
\hline
 type               & \(\varkappa\) & \(j\) & \(\mathrm{h}(L_1)\) & \(\mathrm{h}(L_2)\) & \(\mathrm{Cl}_3(N_1)\) & \(\mathrm{Cl}_3(N_2)\) & \(\mathrm{Cl}_3(N_3)\) & \(\mathrm{Cl}_3(N_4)\) & \(\varepsilon\) & \(\mathrm{Cl}_3(\mathrm{F}_3^1(K))\) & min.\(\lvert D\rvert\) & \multicolumn{2}{|c|}{freq.}                                   \\
\hline
 E.6                & \((1313)\)    & \(6\) &                 \(9\) &                 \(3\) &             \((27,9)\) &              \((9,3)\) &            \((3,3,3)\) &              \((9,3)\) &           \(1\) &                          \((9,9,3)\) &             \(15\,544\) & \multirow{2}{*}{\(\Bigr\rbrace\)}      & \multirow{2}{*}{\(186\)} \\
 E.14               & \((2313)\)    & \(6\) &                 \(9\) &                 \(3\) &             \((27,9)\) &              \((9,3)\) &            \((3,3,3)\) &              \((9,3)\) &           \(1\) &                          \((9,9,3)\) &             \(16\,627\) &                                        &                          \\
\cline{13-14}                                                                                                                                               
 H.4\(\uparrow\)    & \((3313)\)    & \(6\) &                 \(9\) &                 \(3\) &             \((27,9)\) &              \((9,3)\) &            \((3,3,3)\) &              \((9,3)\) &           \(1\) &                         \((27,9,3)\) &             \(21\,668\) &                                        &                   \(63\) \\
\hline
 E.6\(\uparrow\)    & \((1313)\)    & \(8\) &                \(27\) &                 \(3\) &            \((81,27)\) &              \((9,3)\)  &            \((3,3,3)\) &              \((9,3)\) &           \(1\) &                        \((27,27,3)\) &            \(268\,040\) & \multirow{2}{*}{\(\Bigr\rbrace\)}      & \multirow{2}{*}{\(15\)}  \\
 E.14\(\uparrow\)   & \((2313)\)    & \(8\) &                \(27\) &                 \(3\) &            \((81,27)\) &              \((9,3)\) &            \((3,3,3)\) &              \((9,3)\) &           \(1\) &                        \((27,27,3)\) &            \(262\,744\) &                                        &                          \\               
\cline{13-14}
 H.4\(\uparrow^2\)  & \((3313)\)    & \(8\) &                \(27\) &                 \(3\) &            \((81,27)\) &              \((9,3)\) &            \((3,3,3)\) &              \((9,3)\) &           \(1\) &                        \((81,27,3)\) &            \(446\,788\) &                                        &                    \(6\) \\
\hline
\multicolumn{12}{|r|}{total:}																																																				 	  &  & \(270\) \\
\hline
\end{tabular}
\end{center}
\end{table}

}
\normalsize

Again, we continue opposing the complex analog to Table
\ref{tab:RealTreeQ}
in Table
\ref{tab:CompTreeQ}.
Among the \(2\,020\) complex quadratic fields under investigation,
a considerable fraction of \(270\) fields, that is \(13.4\%\),
has a second \(3\)-class group \(G\) on the coclass tree \(\mathcal{T}(\langle 243,6\rangle)\) of \(\mathcal{G}(3,2)\).
However, the groups now populate two even branches \(\mathcal{B}(j)\), \(j\in\lbrace 6,8\rbrace\), of this tree.

\small{

\begin{table}[ht]
\caption{Principalization types with \(G\in\mathcal{T}(\langle 243,8\rangle)\subset\mathcal{G}(3,2)\) for \(D>0\)}
\label{tab:RealTreeU}
\begin{center}
\begin{tabular}{|lc|c|cc|ccccc|c|r|lr|}
\hline
 type             & \(\varkappa\) & \(j\) & \(\mathrm{h}_3(L_1)\) & \(\mathrm{h}_3(L_2)\) & \(\mathrm{Cl}_3(N_1)\) & \(\mathrm{Cl}_3(N_2)\) & \(\mathrm{Cl}_3(N_3)\) & \(\mathrm{Cl}_3(N_4)\) & \(\varepsilon\) & \(\mathrm{Cl}_3(\mathrm{F}_3^1(K))\) &      min. \(D\) & \multicolumn{2}{|c|}{freq.}                                 \\
\hline                                                                                                                                                                     
 c.21             & \((0231)\)    & \(6\) &                 \(9\) &                 \(3\) &              \((9,9)\) &              \((9,3)\) &              \((9,3)\) &              \((9,3)\) &           \(0\) &                          \((9,3,3)\) &    \(540\,365\) & & \(25\) \\
\cline{13-14}                                                     
 E.8              & \((1231)\)    & \(6\) &                 \(9\) &                 \(3\) &             \((27,9)\) &              \((9,3)\) &              \((9,3)\) &              \((9,3)\) &           \(0\) &                          \((9,9,3)\) & \(6\,098\,360\) & \multirow{2}{*}{\(\Bigr\rbrace\)} & \multirow{2}{*}{\(14\)} \\
 E.9              & \((2231)\)    & \(6\) &                 \(9\) &                 \(3\) &             \((27,9)\) &              \((9,3)\) &              \((9,3)\) &              \((9,3)\) &           \(0\) &                          \((9,9,3)\) &    \(342\,664\) &                                   &                         \\
\cline{13-14}                                                     
 G.16             & \((4231)\)    & \(6\) &                 \(9\) &                 \(3\) &             \((27,9)\) &              \((9,3)\) &              \((9,3)\) &              \((9,3)\) &           \(0\) &                         \((27,9,3)\) & \(8\,711\,453\) &                                   &                   \(2\) \\
\hline                                                    
 c.21\(\uparrow\) & \((0231)\)    & \(8\) &                \(27\) &                 \(3\) &            \((27,27)\) &              \((9,3)\) &              \((9,3)\) &              \((9,3)\) &           \(0\) &                         \((27,9,3)\) & \(1\,001\,957\) & &  \(2\) \\
\hline                                                                                                                                                                
\multicolumn{12}{|r|}{total:}																																																				 	  &  & \(43\) \\
\hline
\end{tabular}
\end{center}
\end{table}

}
\normalsize

The second \(3\)-class groups \(G\) for types \(\varkappa\) in Table 
\ref{tab:RealTreeU}
are vertices on the coclass tree with root
\(\langle 243,8\rangle\simeq G_0^{4,5}(0,0,0,1)\) of \(\mathcal{G}(3,2)\).
\(G\) is a vertex
of depth \(0\) for type c.21,
of depth \(1\) for types E.8, E.9, and
of depth \(2\) for type G.16,
on the even branches \(\mathcal{B}(j)\), \(j\in\lbrace 6,8\rbrace\),
\cite[Tbl. 2, p. 266]{AHL}, \cite[Fig. 4.8, p. 76]{As}, \cite[p. 189 ff.]{Ne1}.
The types E.8, E.9
can only be separated by the classical principalization algorithm 
\cite{SoTa,HeSm}.
Type G.16 can be identified either by the first Hilbert \(3\)-class field
or by the classical principalization algorithm.
Among the \(263\) groups of even coclass,
a fraction of \(43\) groups, that is \(16.3\%\),
populates this tree.

\small{

\begin{table}[ht]
\caption{Principalization types with \(G\in\mathcal{T}(\langle 243,8\rangle)\subset\mathcal{G}(3,2)\) for \(D<0\)}
\label{tab:CompTreeU}
\begin{center}
\begin{tabular}{|lc|c|cc|ccccc|c|r|lr|}
\hline
 type               & \(\varkappa\) & \(j\) & \(\mathrm{h}(L_1)\) & \(\mathrm{h}(L_2)\) & \(\mathrm{Cl}_3(N_1)\) & \(\mathrm{Cl}_3(N_2)\) & \(\mathrm{Cl}_3(N_3)\) & \(\mathrm{Cl}_3(N_4)\) & \(\varepsilon\) & \(\mathrm{Cl}_3(\mathrm{F}_3^1(K))\) & min.\(\lvert D\rvert\) & \multicolumn{2}{|c|}{freq.}                                   \\
\hline
 E.8                & \((1231)\)    & \(6\) &                 \(9\) &                 \(3\) &             \((27,9)\) &              \((9,3)\) &              \((9,3)\) &              \((9,3)\) &           \(0\) &                          \((9,9,3)\) &             \(34\,867\) & \multirow{2}{*}{\(\Bigr\rbrace\)}      & \multirow{2}{*}{\(197\)} \\
 E.9                & \((2231)\)    & \(6\) &                 \(9\) &                 \(3\) &             \((27,9)\) &              \((9,3)\) &              \((9,3)\) &              \((9,3)\) &           \(0\) &                          \((9,9,3)\) &              \(9\,748\) &                                        &                          \\
\cline{13-14}                                                                                                                                               
 G.16               & \((4231)\)    & \(6\) &                 \(9\) &                 \(3\) &             \((27,9)\) &              \((9,3)\) &              \((9,3)\) &              \((9,3)\) &           \(0\) &                         \((27,9,3)\) &             \(17\,131\) &                                        &                  \(79\)  \\
\hline
 E.8\(\uparrow\)    & \((1231)\)    & \(8\) &                \(27\) &                 \(3\) &            \((81,27)\) &              \((9,3)\) &              \((9,3)\) &              \((9,3)\) &           \(0\) &                        \((27,27,3)\) &            \(370\,740\) & \multirow{2}{*}{\(\Bigr\rbrace\)}      & \multirow{2}{*}{\(13\)}  \\
 E.9\(\uparrow\)    & \((2231)\)    & \(8\) &                \(27\) &                 \(3\) &            \((81,27)\) &              \((9,3)\) &              \((9,3)\) &              \((9,3)\) &           \(0\) &                        \((27,27,3)\) &            \(297\,079\) &                                        &                          \\               
\cline{13-14}
 G.16\(\uparrow\)   & \((4231)\)    & \(8\) &                \(27\) &                 \(3\) &            \((81,27)\) &              \((9,3)\) &              \((9,3)\) &              \((9,3)\) &           \(0\) &                        \((81,27,3)\) &            \(819\,743\) &                                        &                    \(2\) \\
\hline
\multicolumn{12}{|r|}{total:}																																																				 	  &  & \(291\) \\
\hline
\end{tabular}
\end{center}
\end{table}

}
\normalsize

Table
\ref{tab:CompTreeU}
is the complex analog of Table
\ref{tab:RealTreeU}.
Among the \(2\,020\) complex quadratic fields under investigation,
a considerable fraction of \(291\) fields, that is \(14.4\%\),
has a second \(3\)-class group \(G\) on the coclass tree \(\mathcal{T}(\langle 243,8\rangle)\) of \(\mathcal{G}(3,2)\).

\small{

\begin{table}[ht]
\caption{Principalization types with \(G\in\mathcal{G}(3,4)\) for \(D>0\)}
\label{tab:RealLow}
\begin{center}
\begin{tabular}{|lc|cc|ccccc|c|r|r|}
\hline
 type             & \(\varkappa\) & \(\mathrm{h}_3(L_1)\) & \(\mathrm{h}_3(L_2)\) & \(\mathrm{Cl}_3(N_1)\) & \(\mathrm{Cl}_3(N_2)\) & \(\mathrm{Cl}_3(N_3)\) & \(\mathrm{Cl}_3(N_4)\) & \(\varepsilon\) & \(\mathrm{Cl}_3(\mathrm{F}_3^1(K))\) &      min. \(D\) &    fr. \\
\hline
 F.13             & \((3143)\)    &                 \(9\) &                 \(9\) &             \((27,9)\) &             \((27,9)\) &            \((3,3,3)\) &            \((3,3,3)\) &           \(2\) &                        \((9,9,9,3)\) & \(8\,321\,505\) &                   \(1\) \\
 F.13\(\uparrow\) & \((3143)\)    &                \(27\) &                 \(9\) &            \((81,27)\) &             \((27,9)\) &            \((3,3,3)\) &            \((3,3,3)\) &           \(2\) &                      \((27,27,9,3)\) & \(8\,127\,208\) &                   \(1\) \\
\hline
 d.25\({}^\ast\)  & \((0143)\)    &                \(27\) &                 \(9\) &            \((27,27)\) &             \((27,9)\) &            \((3,3,3)\) &            \((3,3,3)\) &           \(2\) &                       \((27,9,9,3)\) & \(8\,491\,713\) &  \(1\) \\
\hline
\multicolumn{11}{|r|}{total:}																																																				 	                                    & \(3\) \\
\hline
\end{tabular}
\end{center}
\end{table}

}
\normalsize

The extremely rare second \(3\)-class groups \(G\) for types \(\varkappa\) in Table
\ref{tab:RealLow}
are vertices on the coclass graph \(\mathcal{G}(3,4)\).
The group \(G\) is either
a sporadic vertex outside of coclass trees for type F.13,
or a vertex on the even branch \(\mathcal{B}(10)\)
of one of the five metabelian coclass trees of \(\mathcal{G}(3,4)\),
a vertex of depth \(1\) for type F.13\(\uparrow\),
and a main line vertex of depth \(0\) for type d.25\({}^\ast\)
\cite[p. 189 ff.]{Ne1},
\cite[Thm. 3.4, p. 491]{Ma2}.
The types F.13, F.13\(\uparrow\), and d.25\({}^\ast\)
must be identified by the classical principalization algorithm
\cite{SoTa,HeSm}.
Among the \(263\) groups of even coclass for real quadratic fields,
a very exotic fraction of only \(3\) groups, that is \(1.1\%\),
populates \(\mathcal{G}(3,4)\).
The population sets in with remarkable delay at \(D>8\cdot 10^6\).

\small{

\begin{table}[ht]
\caption{Principalization types with \(G\in\mathcal{G}(3,4)\) for \(D<0\)}
\label{tab:CompLow}
\begin{center}
\begin{tabular}{|lc|cc|ccccc|c|r|lr|}
\hline
 type               & \(\varkappa\) & \(\mathrm{h}(L_1)\) & \(\mathrm{h}(L_2)\) & \(\mathrm{Cl}_3(N_1)\) & \(\mathrm{Cl}_3(N_2)\) & \(\mathrm{Cl}_3(N_3)\) & \(\mathrm{Cl}_3(N_4)\) & \(\varepsilon\) & \(\mathrm{Cl}_3(\mathrm{F}_3^1(K))\) & min.\(\lvert D\rvert\) & \multicolumn{2}{|c|}{freq.}                                   \\
\hline
\cline{12-13}
 F.7                & \((3443)\)    &                 \(9\) &                 \(9\) &             \((27,9)\) &             \((27,9)\) &            \((3,3,3)\) &            \((3,3,3)\) &           \(2\) &                        \((9,9,9,3)\) &            \(124\,363\) & \multirow{4}{*}{\Huge\(\Bigr\rbrace\)} & \multirow{4}{*}{\(78\)}  \\
 F.11               & \((1143)\)    &                 \(9\) &                 \(9\) &             \((27,9)\) &             \((27,9)\) &            \((3,3,3)\) &            \((3,3,3)\) &           \(2\) &                        \((9,9,9,3)\) &             \(27\,156\) &                                        &                          \\
 F.12               & \((1343)\)    &                 \(9\) &                 \(9\) &             \((27,9)\) &             \((27,9)\) &            \((3,3,3)\) &            \((3,3,3)\) &           \(2\) &                        \((9,9,9,3)\) &             \(31\,908\) &                                        &                          \\
 F.13               & \((3143)\)    &                 \(9\) &                 \(9\) &             \((27,9)\) &             \((27,9)\) &            \((3,3,3)\) &            \((3,3,3)\) &           \(2\) &                        \((9,9,9,3)\) &             \(67\,480\) &                                        &                          \\
\cline{12-13}                                                                                                                                                                          
 F.7\(\uparrow\)    & \((3443)\)    &                \(27\) &                 \(9\) &            \((81,27)\) &             \((27,9)\) &            \((3,3,3)\) &            \((3,3,3)\) &           \(2\) &                      \((27,27,9,3)\) &            \(469\,816\) & \multirow{4}{*}{\Huge\(\Bigr\rbrace\)} & \multirow{4}{*}{\(14\)}  \\
 F.11\(\uparrow\)   & \((1143)\)    &                \(27\) &                 \(9\) &            \((81,27)\) &             \((27,9)\) &            \((3,3,3)\) &            \((3,3,3)\) &           \(2\) &                      \((27,27,9,3)\) &            \(469\,787\) &                                        &                          \\
 F.12\(\uparrow\)   & \((1343)\)    &                \(27\) &                 \(9\) &            \((81,27)\) &             \((27,9)\) &            \((3,3,3)\) &            \((3,3,3)\) &           \(2\) &                      \((27,27,9,3)\) &            \(249\,371\) &                                        &                          \\
 F.13\(\uparrow\)   & \((3143)\)    &                \(27\) &                 \(9\) &            \((81,27)\) &             \((27,9)\) &            \((3,3,3)\) &            \((3,3,3)\) &           \(2\) &                      \((27,27,9,3)\) &            \(159\,208\) &                                        &                          \\
\cline{12-13}                                                                                                                                                                  
\hline
 G.16r              & \((1243)\)    &                 \(9\) &                 \(9\) &             \((27,9)\) &             \((27,9)\) &            \((3,3,3)\) &            \((3,3,3)\) &           \(2\) &                       \((27,9,9,3)\) &            \(290\,703\) & \multirow{3}{*}{\(\Biggr\rbrace\)}     & \multirow{3}{*}{\(19\)}  \\
 G.19r              & \((2143)\)    &                 \(9\) &                 \(9\) &             \((27,9)\) &             \((27,9)\) &            \((3,3,3)\) &            \((3,3,3)\) &           \(2\) &                       \((27,9,9,3)\) &             \(96\,827\) &                                        &                          \\
 H.4r               & \((3343)\)    &                 \(9\) &                 \(9\) &             \((27,9)\) &             \((27,9)\) &            \((3,3,3)\) &            \((3,3,3)\) &           \(2\) &                       \((27,9,9,3)\) &            \(256\,935\) &                                        &                          \\
\cline{12-13}                       
 G.16i              & \((1243)\)    &                 \(9\) &                 \(9\) &             \((27,9)\) &             \((27,9)\) &            \((3,3,3)\) &            \((3,3,3)\) &           \(2\) &                        \((9,9,9,9)\) &            \(135\,059\) & \multirow{3}{*}{\(\Biggr\rbrace\)}     & \multirow{3}{*}{\(15\)}  \\
 G.19i              & \((2143)\)    &                 \(9\) &                 \(9\) &             \((27,9)\) &             \((27,9)\) &            \((3,3,3)\) &            \((3,3,3)\) &           \(2\) &                        \((9,9,9,9)\) &            \(199\,735\) &                                        &                          \\
 H.4i              & \((3343)\)    &                 \(9\) &                 \(9\) &             \((27,9)\) &             \((27,9)\) &            \((3,3,3)\) &            \((3,3,3)\) &           \(2\) &                        \((9,9,9,9)\) &            \(186\,483\) &                                        &                          \\
\hline                                                                                                                                                                          
 G.19r\(\uparrow\)  & \((2143)\)    &                \(27\) &                 \(9\) &            \((81,27)\) &             \((27,9)\) &            \((3,3,3)\) &            \((3,3,3)\) &           \(2\) &                      \((81,27,9,3)\) &            \(509\,160\) &                                        &                    \(2\) \\
 H.4r\(\uparrow\)   & \((3343)\)    &                \(27\) &                 \(9\) &            \((81,27)\) &             \((27,9)\) &            \((3,3,3)\) &            \((3,3,3)\) &           \(2\) &                      \((81,27,9,3)\) &            \(678\,804\) &                                        &                    \(3\) \\
\hline
\multicolumn{11}{|r|}{total:}																																																				 	  &  & \(131\) \\
\hline
\end{tabular}
\end{center}
\end{table}

}
\normalsize

In contrast, Table
\ref{tab:CompLow}
shows that the coclass graph \(G\in\mathcal{G}(3,4)\) accommodates
the second \(3\)-class groups \(G\) of quite a notable portion of \(131\), that is \(6.5\%\),
among the \(2\,020\) complex quadratic fields under investigation.
The group \(G\) is either
a sporadic vertex outside of coclass trees,
namely an isolated metabelian top vertex of \(\mathcal{G}(3,4)\) for the types F.7, F.11, F.12, F.13
and a terminal metabelian vertex of depth \(1\) on a finite tree for the types G.16r, G.19r, H.4r, G.16i, G.19i, H.4i,
or a vertex on the even branch \(\mathcal{B}(10)\)
of one of the five metabelian coclass trees of \(\mathcal{G}(3,4)\),
namely a vertex of depth \(1\) for the types F.7\(\uparrow\), F.11\(\uparrow\), F.12\(\uparrow\), F.13\(\uparrow\),
and a vertex of depth \(2\) for the types G.19r\(\uparrow\), H.4r\(\uparrow\).
We point out that the classical principalization algorithm
\cite{SoTa,HeSm},
which must be used to separate the types F.7, F.11, F.12, F.13, resp. G.16, G.19, H.4,
is not able to reveal that
the types \(\mathrm{G}.16\), \(\mathrm{G}.19\), and \(\mathrm{H}.4\)
with odd index of nilpotency \(m=7\) and order \(3^n\), \(n=10=2m-4\), can appear in
a regular variant (r) with \(\mathrm{Cl}_3(\mathrm{F}_3^1(K))\) of type \((27,9,9,3)\)
and an irregular variant (i) with \(\mathrm{Cl}_3(\mathrm{F}_3^1(K))\) of type \((9,9,9,9)\)
\cite[Satz 4.2.4, p. 131]{Ne1}.

\small{

\begin{table}[ht]
\caption{Principalization type with \(G\in\mathcal{G}(3,6)\) for \(D<0\)}
\label{tab:CompExot}
\begin{center}
\begin{tabular}{|lc|cc|ccccc|c|r|lr|}
\hline
 type               & \(\varkappa\) & \(\mathrm{h}(L_1)\) & \(\mathrm{h}(L_2)\) & \(\mathrm{Cl}_3(N_1)\) & \(\mathrm{Cl}_3(N_2)\) & \(\mathrm{Cl}_3(N_3)\) & \(\mathrm{Cl}_3(N_4)\) & \(\varepsilon\) & \(\mathrm{Cl}_3(\mathrm{F}_3^1(K))\) & min.\(\lvert D\rvert\) & \multicolumn{2}{|c|}{freq.}                                   \\
\hline
 F.12\(\uparrow^2\) & \((1343)\)    &                \(27\) &                \(27\) &            \((81,27)\) &            \((81,27)\) &            \((3,3,3)\) &            \((3,3,3)\) &           \(2\) &                     \((27,27,27,9)\) &            \(423\,640\) &                                        &                    \(1\) \\
\hline
\multicolumn{11}{|r|}{total:}																																																				 	  &  & \(1\) \\
\hline
\end{tabular}
\end{center}
\end{table}

}
\normalsize

Finally, in Table
\ref{tab:CompExot},
type F.12\(\uparrow^2\) with a relative frequency of \(\frac{1}{2020}=0.05\%\)
is the absolute low-champ for complex quadratic fields.
The corresponding second \(3\)-class group \(G\)
is of the biggest order \(3^{13}\), known until now.
It is the unique sporadic vertex outside of coclass trees
which appeared on the coclass graph \(\mathcal{G}(3,6)\).
Type F.12\(\uparrow^2\)
must be identified by the classical principalization algorithm
\cite{SoTa,HeSm}.

\begin{remark}
Based on the statistical evaluation of all numerical results,
we are able to conclude that the new principalization algorithm
is significantly more efficient for real quadratic fields.
The time consuming third step, \S\
\ref{ss:FstHilb},
of the algorithm
can be avoided for \(\frac{2548}{2576}=98.9\%\) of the real quadratic fields
(Tables
\ref{tab:RealMax},
\ref{tab:RealMix},
\ref{tab:RealPartSpor}
entirely and additionally the \(29\) cases of type \(\mathrm{c.18}\) in Table
\ref{tab:RealTreeQ},
the \(27\) cases of type \(\mathrm{c.21}\) in Table
\ref{tab:RealTreeU},
and the single case of type \(\mathrm{d.25}^\ast\) in Table
\ref{tab:RealLow}),
but only for \(\frac{1327}{2020}=65.7\%\) of the complex quadratic fields
(Table \ref{tab:CompSpor}).
\end{remark}


\section{Acknowledgements}
\label{s:Final}

The author expresses his gratitude to
Mike F. Newman, ANU, Canberra,
for precious aid in unifying the various approaches
\cite{AHL,As,As1,Bg,Bl,Ef,GAP,Hl,Jm,Ne1,Sr2}
to the classification of finite metabelian \(3\)-groups
and in discovering the last statement of Corollary
\ref{c:SmlSecQdr}.
Further he would like to thank
Karim Belabas, Univ. Bordeaux, and
Claus Fieker, Univ. Kaiserslautern,
for illuminating discussions concerning
the computation of arithmetical invariants of
Hilbert \(3\)-class fields over quadratic fields \cite{Be,Fi},
using the computer algebra systems PARI/GP and MAGMA.
Finally, the author is endebted to the anonymous referee
for valuable suggestions to improve the exposition.




\section{Appendix}
\label{s:Appendix}

\noindent
For the convenience of the reader,
we literally cite some original results of Blackburn
\cite{Bl},
and as a service to the mathematical community,
we give a succinct survey of Nebelung's thesis
\cite{Ne1,Ne2}.
Both of these works are used essentially in our article.

\subsection{A theorem by Blackburn}
\label{ss:Blackburn}

\noindent
Let \(p\) be an arbitrary prime and
observe that Blackburn denotes our two-step centralizer \(\chi_2(G)\)
of a \(p\)-group \(G\) (see \S\
\ref{ss:MaxHom})
by \(\gamma_1(G)\),
whereas his usage of \(\gamma_i(G)\) for \(i\ge 2\) coincides with ours,
denoting the members of the lower central series.
The following theorem shows
that all subgroups \(\gamma_i(G)\), \(1\le i\le m-p+1\), of a \(p\)-group \(G\)
of coclass \(\mathrm{cc}(G)=1\) and class \(\mathrm{cl}(G)=m-1>p\)
are regular and have the same invariants as
the nearly homocyclic abelian \(p\)-group \(\mathrm{A}(p,m-i)\) of type
\(\left(\overbrace{p^{q_i+1},\ldots,p^{q_i+1}}^{r_i\text{ times}},\overbrace{p^{q_i},\ldots,p^{q_i}}^{p-1-r_i\text{ times}}\right)\),
where \(m-i=q_i(p-1)+r_i\)
by Euclidean division with quotient \(q_i>0\) and remainder \(0\le r_i<p-1\).
In particular, if \(G\) is metabelian,
then \(\gamma_i(G)\simeq\mathrm{A}(p,m-i)\), for \(2\le i\le m-p+1\),
and if \(G\) is metabelian with defect \(k(G)=0\),
then also \(\gamma_1(G)\simeq\mathrm{A}(p,m-1)\).

\begin{theorem} 
\label{thm:Blackburn}
(see \cite[Thm. 3.4, p. 68]{Bl})
If \(G\) is a group of order \(p^m\) and class \(m-1\),
where \(m>3\),
then \(\gamma_1(G)\) is a regular \(p\)-group.
If \(m>p+1\), and for each \(i=1,2,\ldots,m-p+1\),
we write \(m-i=(p-1)q_i+r_i\) \((0\le r_i<p-1)\),
then \(\gamma_i(G)\) has
\(r_i\) invariants equal to \(q_i+1\) and
\(p-r_i-1\) invariants equal to \(q_i\).
\end{theorem}

\subsection{Main theorems of Nebelung's thesis}
\label{ss:Nebelung}

Brigitte Nebelung completed her thesis
\cite{Ne1,Ne2}
in 1989 under supervision of Wolfram Jehne at Cologne.
She had been introduced to computational group theory
by Charles R. Leedham-Green and Joachim Neub\"user.
Furthermore, she had studied the details concerning
\(3\)-groups of coclass \(2\) with two generators
in Judith A. Ascione's thesis
\cite{As},
written under supervision of Mike F. Newman.

Nebelung determined explicit parametrized presentations
for all isomorphism classes of
metabelian \(3\)-groups with abelianization of type \((3,3)\)
and arbitrary coclass greater than or equal to \(2\),
thereby extending Blackburn's results for coclass \(1\) in
\cite{Bl}.
The lower and upper central series of all these groups were shown
to have \(3\)-elementary cyclic or bicyclic factors,
the number of the latter being equal to the coclass.
For all members of the lower central series,
in particular for the abelian derived subgroup,
the structure was given by abelian type invariants.
Based on a computer calculation of a
complete and irredundant set of isomorphism classes,
listed in volume \(2\)
\cite{Ne2}
of her thesis,
Nebelung proved that the metabelian coclass trees
and their branches are arranged in periodic patterns.
Finally, the number theoretic capitulation problem
was solved for arbitrary number fields
with 3-class group of type \((3,3)\)
by calculating
the transfer kernels of any metabelian \(3\)-group
having abelianization of type \((3,3)\)
with respect to its maximal subgroups.

Throughout the sequel,
let \(4\le m\le n\) be integers and
\(G\) be a metabelian \(3\)-group
of order \(\lvert G\rvert=3^n\) and class \(\mathrm{cl}(G)=m-1\)
with derived quotient \(G/G^\prime\simeq (3,3)\).

Nebelung's starting point for analyzing the
structure of the lower and upper central series of \(G\)
was the following inconspicuous result
which proves to be very powerful.
Observe that Nebelung denotes
the members \(\gamma_j(G)\) of the lower central series by \(G_j\),
for \(j\ge 1\).

\begin{theorem} 
\label{thm:NebelungPowerBasis}
(see \cite[Satz 3.1.11, p. 57]{Ne1})
There exist generators \(x,y\) of \(G=\langle x,y\rangle\)
such that \(\gamma_3(G)=\langle x^3,y^3,\gamma_4(G)\rangle\).
\end{theorem}

\noindent
Among the generators of maximal subgroups \(M_i<G\) modulo \(G^\prime\),
\(x\) and \(y\) are indeed distinguished,
in comparison to \(xy\) and \(xy^{-1}\),
provided the class of \(G\) is not too small.
In the lattice of normal subgroups of \(G\),
the third powers
\(x^3\) and \(y^3\) are lying near the top,
namely in \(\gamma_3(G)\setminus\gamma_4(G)\),
whereas \((xy)^3\) and \((xy^{-1})^3\) are lying near the bottom,
namely in the second centre \(\zeta_2(G)\),
and for groups with defect \(k(G)=0\)
even in the centre \(\zeta_1(G)\) of \(G\),
as we shall see in Theorem
\ref{thm:NebelungPowers}.

In the following results concerning the central series
of a metabelian \(3\)-group \(G\) with \(G/G^\prime\) of type \((3,3)\),
Nebelung's usage of the invariants \(e\) and \(s\) coincides with ours in \S\
\ref{ss:LowHom},
but she denotes our two-step centralizers \(\chi_j(G)\) by \(C_j\),
for \(j\ge 1\).
We recall that the coclass of \(G\)
in dependence on the invariant \(e\ge 3\) is given by
\(\mathrm{cc}(G)=e-1\ge 2\).

\begin{theorem} 
\label{thm:NebelungLowerCentralSeries}
(see \cite[Satz 3.3.7, p. 70]{Ne1})

\begin{enumerate}

\item
The single-step factors of the lower central series of \(G\) are given by
\[
\gamma_j(G)/\gamma_{j+1}(G)\simeq
\begin{cases}
(3,3) & \text{ for } j=1 \text{ and } 3\le j\le e, \\
(3)   & \text{ for } j=2 \text{ and } e+1\le j\le m-1.
\end{cases}
\]

\item
Consequently, the exponent \(n\) of the order \(\lvert G\rvert=3^n\)
is related to the index of nilpotency \(m\) by the inequalities
\(m\le n\le 2m-3\)
and the invariant \(e\) is given by \(e=n+2-m\).

\item
The two-step factors \(\gamma_j(G)/\gamma_{j+2}(G)\) of the lower central series of \(G\)
are \(3\)-elementary abelian, since \(\gamma_j(G)^3\le\gamma_{j+2}(G)\), for \(j\ge 1\).

\end{enumerate}

\end{theorem}

\noindent
There are two extreme cases of the relations \(4\le m\le n\le 2m-3\).
For the smallest possible exponent of the order, \(n=m\), with given class \(m-1\),
we have \(e=2\) and \(G\) is a CF-group of coclass \(\mathrm{cc}(G)=1\)
with \textit{cyclic factors} \(\gamma_j(G)/\gamma_{j+1}(G)\), except \(G/G^\prime\).
For the biggest possible exponent of the order, \(n=2m-3\), with given class \(m-1\),
we have \(e=m-1\) and \(G\) is a BF-group of coclass \(\mathrm{cc}(G)=m-2\ge 2\)
with \textit{bicyclic factors} \(\gamma_j(G)/\gamma_{j+1}\), except \(G^\prime/\gamma_3(G)\).
BF-groups play the role of \textit{interface} groups
at the border between different coclass graphs in
\cite[\S\ 3.3.5, Dfn. 3.3 and Thm. 3.11]{Ma5}.
All groups, including the extreme cases, are BCF-groups
with \textit{bicyclic or cyclic factors}.
(Nebelung uses the terminology ZEF-groups, which means
\textit{Rang Zwei oder Eins Faktoren}.)

From now on, we exclude groups of coclass \(\mathrm{cc}(G)=1\),
and we permanently rely on the following assumptions
without repeating them explicitly.
Let \(G\) be a metabelian \(3\)-group
with commutator factor group \(G/G^\prime\) of type \((3,3)\).
Assume that \(G\) has order \(\lvert G\rvert=3^n\),
class \(\mathrm{cl}(G)=m-1\), and invariant \(e=n+2-m\ge 3\),
where \(4\le m<n\le 2m-3\).
Let generators \(x,y\) of \(G=\langle x,y\rangle\) be selected such that
\(\gamma_3(G)=\langle x^3,y^3,\gamma_4(G)\rangle\),
\(x\in G\setminus\chi_s(G)\) if \(s<m-1\),
and \(y\in\chi_s(G)\setminus G^\prime\).
(Nebelung calls such a couple \((x,y)\)
an \textit{admissible} pair of \textit{normal} generators.)
Let commutators of \(G\) be declared by
\(s_2=t_2=\lbrack y,x\rbrack\in\gamma_2(G)\)
and recursively by
\(s_j=\lbrack s_{j-1},x\rbrack\), \(t_j=\lbrack t_{j-1},y\rbrack\in\gamma_j(G)\)
for \(j\ge 3\).
Starting with the powers \(\sigma_3=y^3\), \(\tau_3=x^3\in\gamma_3(G)\), let
\(\sigma_j=\lbrack\sigma_{j-1},x\rbrack\), \(\tau_j=\lbrack\tau_{j-1},y\rbrack\in\gamma_j(G)\)
for \(j\ge 4\).

\begin{theorem} 
\label{thm:NebelungTwoStepCentralizers}
(see \cite[Satz 3.3.7, p. 70]{Ne1})

\begin{enumerate}

\item
The two-step centralizers of the lower central series of \(G\) are usually given by
\[
\chi_j(G)=
\begin{cases}
G^\prime                  & \text{ for } 1\le j\le e-1, \\
\langle y,G^\prime\rangle & \text{ for } e\le j\le m-2, \\
G                         & \text{ for } j\ge m-1.
\end{cases}
\]
except in the special case that \(e=m-2\) and \(\lbrack\chi_s(G),\gamma_e(G)\rbrack=\gamma_{m-1}(G)\), where
\[
\chi_j(G)=
\begin{cases}
G^\prime                  & \text{ for } 1\le j\le m-2, \\
G                         & \text{ for } j\ge m-1.
\end{cases}
\]

\item
For the invariants \(s\) and \(e\)
it follows that usually \(s=e\),
except in the case \(e=m-2\), \(\lbrack\chi_s(G),\gamma_e(G)\rbrack=\gamma_{m-1}(G)\), where \(s=m-1>e\).

\end{enumerate}

\end{theorem}

\noindent
Let the \textit{upper central series} of \(G\) be defined recursively by
\(\zeta_0(G)=1\) and \(\zeta_j(G)/\zeta_{j-1}(G)=\mathrm{Centre}(G/\zeta_{j-1}(G))\), for \(j\ge 1\).
In particular, \(\zeta_1(G)\) is the usual centre of \(G\).
Nebelung denotes \(\zeta_j(G)\) by \(Z_j\).
The following theorem shows that the bicyclic factors of the upper central series of \(G\)
are located near the bottom of \(G\),
whereas the bicyclic factors of the lower central series of \(G\)
appear in a position near the top of \(G\), by Theorem
\ref{thm:NebelungLowerCentralSeries}.

\begin{theorem} 
\label{thm:NebelungUpperCentralSeries}
(see \cite[Satz 3.3.18, p. 82]{Ne1})

\begin{enumerate}

\item
If \(\lbrack\chi_s(G),\gamma_e(G)\rbrack=1\), then
the factors of the upper central series of \(G\) are given by
\[
\zeta_j(G)/\zeta_{j-1}(G)\simeq
\begin{cases}
(3,3) & \text{ for } 1\le j\le e-2 \text{ and } j=m-1, \\
(3)   & \text{ for } e-1\le j\le m-2.
\end{cases}
\]
In particular,
the centre \(\zeta_1(G)=\langle\sigma_{m-1}\rangle\times\langle\tau_{e}\rangle\) of \(G\) is
bicyclic of type \((3,3)\).

\item
If \(\lbrack\chi_s(G),\gamma_e(G)\rbrack=\gamma_{m-1}(G)\), then
the factors of the upper central series of \(G\) are given by
\[
\zeta_j(G)/\zeta_{j-1}(G)\simeq
\begin{cases}
(3,3) & \text{ for } 2\le j\le e-1 \text{ and } j=m-1, \\
(3)   & \text{ for } j=1 \text{ and } e\le j\le m-2.
\end{cases}
\]
In particular,
the centre \(\zeta_1(G)=\langle\sigma_{m-1}\rangle\) of \(G\) is
cyclic of order \(3\).

\end{enumerate}

\end{theorem}

\noindent
The following main theorem gives \textit{parametrized presentations}
for all metabelian \(3\)-groups \(G\) of coclass \(\mathrm{cc}(G)\ge 2\)
having abelianization \(G/G^\prime\) of type \((3,3)\).

\begin{theorem} 
\label{thm:NebelungPresentations}
(see \cite[Thm. 3.4.5, p. 94]{Ne1})
There exists a unique matrix
\(
\begin{pmatrix}
\alpha & \beta \\
\gamma & \delta
\end{pmatrix}
\)
with entries in the finite field \(\mathbb{F}_3\),
and a unique element \(\rho\in\mathbb{F}_3\),
such that the following relations are satisfied

\begin{equation}
\label{eqn:LowRelApp}
\begin{array}{rcl}
t_3^{-1}\tau_3\tau_4                      &=& \sigma_{m-1}^\alpha\sigma_{m-2}^{\rho\delta}\tau_e^\beta,\\
s_3\sigma_3\sigma_4                       &=& \sigma_{m-1}^\gamma\sigma_{m-2}^{\rho\beta}\tau_e^\delta,\\
\tau_{e+1}                                &=& \sigma_{m-1}^{-\rho},\\
\lbrack s_3,y\rbrack=\lbrack t_3,x\rbrack &=& \sigma_{m-1}^{-\rho\delta},\\
s_2^3(s_3t_3)^3s_4t_4                     &=& \sigma_{m-1}^{\rho\beta},\\
s_2^3                                     &=& \sigma_4\sigma_{m-1}^{-\rho\beta}\tau_4^{-1}.
\end{array}
\end{equation}

\noindent
In the first two relations, put \(\sigma_{m-2}=1\) in the case of \(m=4\).\\
Concerning the parameter \(\rho\),
let \(\lbrack\chi_s(G),\gamma_e(G)\rbrack=\gamma_{m-k}(G)\) with \(0\le k\le 1\).
Then the defect of \(G\) is \(k=k(G)=0\) if and only if \(\rho=0\).
In particular, a group \(G\) with \(e=m-1\) must have \(\rho=0\).\\
The parameters of a group \(G\) of coclass \(\mathrm{cc}(G)=2\) are subject to various constraints:

\begin{enumerate}

\item
If \(m=4\), \(n=5\), then \(\rho=0\) and
\(
\begin{pmatrix}
\alpha   & \beta-1 \\
\gamma-1 & \delta
\end{pmatrix}
\in\mathrm{GL}_2(\mathbb{F}_3).
\)

\item
If \(m=5\), \(n=6\), then
\(
\begin{pmatrix}
\rho\delta  & \beta-1 \\
\rho\beta-1 & \delta
\end{pmatrix}
\in\mathrm{GL}_2(\mathbb{F}_3).
\)

\item
If \(5<m=n-1\), then
\(\beta-1\in\mathbb{F}_3\setminus\lbrace 0\rbrace\).

\end{enumerate}

\end{theorem}

\noindent
Consequently, a group \(G\) which satisfies the assumptions of the preceding theorem
is exactly the representative
\(G_\rho^{m,n}(\alpha,\beta,\gamma,\delta)\) 
of an isomorphism class of
metabelian \(3\)-groups \(G\) with \(G/G^\prime\) of type \((3,3)\),
which satisfies the relations
(\ref{eqn:LowRelApp})
with a fixed system of parameters
\(-1\le\alpha,\beta,\gamma,\delta,\rho\le 1\).

By means of a computer search for a
complete and irredundant set of isomorphism classes of
metabelian \(3\)-groups \(G\) with \(G/G^\prime\) of type \((3,3)\),
Nebelung successively determined all admissible families
\((\alpha,\beta,\gamma,\delta,\rho)\)
of parameters in the relations
for groups of fixed coclass \(e-1\) and class \(m-1\),
letting the invariant \(e\ge 3\) increase independently
and incrementing the index of nilpotency \(m\ge e+1\)
in dependence on each fixed value of \(e\).
This revealed a \textit{double periodicity}
with respect to both, coclass and class,
of primitive length \(6\) resp. \(2\).
Consequently, it was possible to list
the representatives
\(G_\rho^{m,n}(\alpha,\beta,\gamma,\delta)\) 
of isomorphism classes
in volume \(2\)
\cite{Ne2}
of Nebelung's thesis
in the form of finitely many parametrized presentations
of infinite periodic coclass sequences
where only \(m\) and \(n=e+m-2\) vary indefinitely
for fixed coclass \(e-1\).
Using the periodicity,
Nebelung proved that the metabelian coclass trees
and their branches are arranged in periodic patterns.

\begin{theorem} 
\label{thm:NebelungPowers}
(see \cite[Lem. 3.4.11, p. 105]{Ne1})
The third powers of \(xy\) and \(xy^{-1}\) are given by
\[
(xy)^3=\sigma_{m-2}^{\rho(\beta+\delta)}\sigma_{m-1}^{\alpha+\gamma+\rho(\beta+\delta)}\tau_e^{\beta+\delta}
\]
and
\[
(xy^{-1})^3=\sigma_{m-2}^{\rho(\delta-\beta)}\sigma_{m-1}^{\alpha-\gamma+\rho\beta}\tau_e^{\beta-\delta}.
\]
In contrast to the powers \(y^3=\sigma_3\) and \(x^3=\tau_3\),
which are elements of \(\gamma_3(G)\setminus\gamma_4(G)\),
the powers \((xy)^3\) and \((xy^{-1})^3\) are contained in the second centre
\(\langle\sigma_{m-1}\rangle\times\langle\sigma_{m-2}\rangle\times\langle\tau_{e}\rangle=\zeta_2(G)\)
and for \(\rho=0\) even in the centre
\(\langle\sigma_{m-1}\rangle\times\langle\tau_{e}\rangle=\zeta_1(G)\).
However, a uniform warranty that \((xy)^3,(xy^{-1})^3\in\gamma_4(G)\)
can only be given for a group \(G\) of coclass \(\mathrm{cc}(G)\ge 3\).
\end{theorem}

\begin{theorem} 
\label{thm:NebelungDerivedSubgroup}
(see \cite[Satz 4.2.4, p. 131]{Ne1})
The structure of the abelian commutator subgroup \(G^\prime\) of \(G\)
is given by the following direct product of nearly homocyclic abelian \(3\)-groups.
\[
G^\prime=
\begin{cases}
\mathrm{A}(3,m-3)\times\mathrm{A}(3,m-3) & \text{ in the irregular case,} \\
\mathrm{A}(3,m-2)\times\mathrm{A}(3,e-2) & \text{ otherwise.}
\end{cases}
\]
The irregular case is characterized by \(n=2m-4\), \(e=m-2\), \(m\equiv 1\mod{2}\), and
\[
\begin{cases}
0\ne\rho=\beta-1 & \text{ for } m=5, \\
\rho=-1          & \text{ for } m\ge 7.
\end{cases}
\]
\end{theorem}


\end{document}